\documentclass[11pt,reqno]{amsart}
\usepackage{amsmath,amsthm,amsfonts,amssymb,amscd}
\usepackage{enumerate}
\usepackage[vcentermath,enableskew]{youngtab}
\usepackage{microtype}
\usepackage{tikz}
\usepackage{mathrsfs}
\usepackage[ocgcolorlinks, linkcolor=red!90!black, citecolor=blue!80!black, pagebackref=true]{hyperref}
\usepackage{mathtools}
\usepackage{accents}

\usetikzlibrary{positioning}

\DeclareFontFamily{U}{wncy}{}
\DeclareFontShape{U}{wncy}{m}{n}{<->wncyr10}{}
\DeclareSymbolFont{mcy}{U}{wncy}{m}{n}
\DeclareMathSymbol{\cyrillicg}{\mathord}{mcy}{"67} 

\newtheorem{thm}{Theorem}[section]
\newtheorem*{thm*}{Theorem}
\newtheorem{lem}[thm]{Lemma}
\theoremstyle{definition} \newtheorem{defn}[thm]{Definition}
\newtheorem*{defn*}{Definition}
\theoremstyle{definition} \newtheorem{ex}[thm]{Example}
\newtheorem*{lem*}{Lemma}
\newtheorem{cor}[thm]{Corollary}
\newtheorem*{cor*}{Corollary}
\theoremstyle{definition} \newtheorem{rem}[thm]{Remark}

\newtheorem*{conj*}{Conjecture}

\newtheorem{prop}[thm]{Proposition}

\newcommand{\CC}{\mathbb{C}}
\newcommand{\NN}{\mathbb{N}}
\newcommand{\ZZ}{\mathbb{Z}}

\newcommand{\E}{\mathcal{E}}
\renewcommand{\th}{\textsuperscript{th}\,}

\renewcommand{\O}{\operatorname{O}}

\DeclareMathOperator{\GL}{GL}
\DeclareMathOperator{\SL}{SL}
\DeclareMathOperator{\SO}{SO}
\DeclareMathOperator{\Sp}{Sp}
\DeclareMathOperator{\Gr}{Gr}
\DeclareMathOperator{\Fl}{Fl}
\DeclareMathOperator{\vspan}{span}
\DeclareMathOperator{\rank}{rk}

\DeclareMathOperator{\LG}{LG}
\DeclareMathOperator{\OG}{OG}
\DeclareMathOperator{\Ess}{Ess}

\DeclareMathOperator{\codim}{codim}

\DeclareMathOperator{\pf}{pf}
\DeclareMathOperator{\Hom}{Hom}

\newcommand{\mshSYT}{\text{ShSYT}'}

\renewcommand{\hat}{\widehat}

\newcommand{\quand}{\quad \text{and} \quad}

\newcommand{\fkS}{\mathfrak{S}}
\newcommand{\fpf}{\textsc{fpf}}
\newcommand{\idO}{1^{\O}}

\newcommand{\idfpf}{1^{\Sp}}
\newcommand{\idK}{1^{K}}
\newcommand{\bS}{\overleftarrow{\mathfrak{S}}\hspace{-0.5mm}}
\newcommand{\ibS}{\bS^{\O}}

\newcommand{\ibSfpf}{\bS^{\Sp}}
\newcommand{\ibSK}{\bS^{K}}
\newcommand{\AK}{\mathcal{A}^K}

\newcommand{\A}{\mathcal{A}^{\O}}
\newcommand{\Afpf}{\mathcal{A}^{\Sp}}

\newcommand{\Red}{\mathcal{R}}
\newcommand{\comp}{C}
\newcommand{\iR}{\hat{\Red}^{\O}}

\newcommand{\iRK}{\hat{\Red}^{K}}

\newcommand{\iRfpf}{\hat{\Red}^{\Sp}}
\newcommand{\iF}{F^{\O}}
\newcommand{\iFfpf}{F^{\Sp}}
\newcommand{\iFO}{F^{\O}}
\newcommand{\iFSp}{F^{\Sp}}
\newcommand{\iFK}{F^{K}}
\newcommand{\iS}{\mathfrak{S}^{\O}}
\newcommand{\iSK}{\mathfrak{S}^{K}}
\newcommand{\iSfpf}{\mathfrak{S}^{\Sp}}
\newcommand{\I}{\mathcal{I}}
\newcommand{\Ifpf}{\mathcal{I}^{\fpf}}
\newcommand{\IK}{\mathcal{I}^{K}}

\newcommand{\iellK}{\hat{\ell}^{K}}
\newcommand{\iellO}{\hat{\ell}^{\O}}
\newcommand{\iellSp}{\hat{\ell}^{\Sp}}

\newcommand{\iDK}{D^{K}}
\newcommand{\iDO}{D^{\O}}
\newcommand{\iDSp}{D^{\Sp}}
\newcommand{\icO}{c^{\O}}
\newcommand{\icSp}{c^{\Sp}}

\newcommand{\dearc}{\textsf{dearc}}
\newcommand{\dearcR}{\textsf{dearc}^{\textrm R}}
\newcommand{\dearcL}{\textsf{dearc}^{\textrm L}}

\newcommand{\sq}{\square}
\newcommand{\twocyc}{\text{cyc}}

\newcommand{\eqdef}{\overset{\text{def}}{=}}
\newcommand{\ev}{\bigr\rvert}
\newcommand{\SWNE}{\vbox{\hbox{\hspace{.5mm}\rotatebox{-15}{$\scriptscriptstyle \nearrow$}}\vspace{-2.65mm}\hbox{$\leq$}}}
\newcommand{\SWNEneq}{\vbox{\hbox{\hspace{.5mm}\rotatebox{-15}{$\scriptscriptstyle \nearrow$}}\vspace{-3.05mm}\hbox{$<$}}}

\newcommand{\graph}{\mathsf{G}}
\newcommand{\vb}{\mathcal}
\newcommand{\openiXK}{{X}^K}
\newcommand{\openiXO}{{X}^{\O}}
\newcommand{\openiXSp}{{X}^{\Sp}}

\newcommand{\iXO}{\overline{X}^{\O}}
\newcommand{\iX}{\overline{X}^{\O}}
\newcommand{\iXSp}{\overline{X}^{\Sp}}
\newcommand{\iXfpf}{\overline{X}^{\Sp}}

\newcommand{\incGX}{\overline{GX}}
\newcommand{\incSGX}{\overline{LGX}}
\newcommand{\incSSGX}{\overline{OGX}}
\newcommand{\openincGX}{GX}
\newcommand{\openincSGX}{LGX}
\newcommand{\openincSSGX}{OGX}

\newcommand{\X}{\overline{X}}
\newcommand{\grGr}{G_{\Gr}}
\newcommand{\grLG}{G_{\LG}}
\newcommand{\grOG}{G_{\OG}}

\newcommand{\supersymm}{\Lambda^{\textrm{super}}}
\newcommand{\superslash}{\setminus}
\newcommand{\ann}[1]{#1^{\perp}}
\newcommand{\id}{\mathrm{id}}
\newcommand{\shfpf}{\operatorname{sh^{\fpf}}}
\newcommand{\sh}{\operatorname{sh}}
\newcommand{\ishO}{\sh^{\O}}
\newcommand{\ishSp}{\sh^{\Sp}}

\newcommand{\epath}{\mathcal{P}}

\DeclareMathOperator{\rowspan}{rowspan}

\renewcommand{\epsilon}{\varepsilon}

\begin{document}

\author{Brendan Pawlowski}
\title{Universal graph Schubert varieties}

\begin{abstract}
    We consider the loci of invertible linear maps $f : \CC^n \to {(\CC^n)}^*$ together with pairs of flags $(E_\bullet, F_\bullet)$ in $\CC^n$ such that the various restrictions $f : F_j \to E_i^*$ have specified ranks. Identifying an invertible linear map with its graph viewed as a point in a Grassmannian, we show that the closures of these loci have cohomology classes represented by the back-stable Schubert polynomials of Lam, Lee, and Shimozono. As a special case, we recover the result of Knutson, Lam, and Speyer that Stanley symmetric functions represent the classes of graph Schubert varieties.
    
    We consider similar loci where $f$ is restricted to be symmetric or skew-symmetric. Their classes are now given by back-stable versions of the polynomials introduced by Wyser and Yong to represent classes of orbit closures for the orthogonal and symplectic groups acting on the type A flag variety. Using degeneracy locus formulas of Kazarian and of Anderson and Fulton, we obtain new Pfaffian formulas for these polynomials in the vexillary case. We also give a geometric interpretation of the involution Stanley symmetric functions of Hamaker, Marberg, and the author: they represent classes of \emph{involution graph Schubert varieties} in isotropic Grassmannians.
\end{abstract}

\maketitle

\section{Introduction}

Let $G$ be a semisimple complex algebraic group. A pair $(G,K)$ where $K \subseteq G$ is a closed subgroup is a \emph{symmetric pair} if $K$ is the fixed point set of an involutive automorphism $G \to G$. For parabolic subgroups $P \subseteq G$ and $Q \subseteq K$, the product $G/P \times K/Q$ is a \emph{double flag variety} for $(G,K)$. In the case that $Q = B_K$ is a Borel subgroup of $K$, He, Nishiyama, Ochiai, and Oshima classified those $(G,K)$ and $P$ for which the $K$-action on $G/P \times K/B_K$ has finitely many orbits \cite{double-flag-varieties}.

 We consider three such cases here:
\begin{enumerate}[(1)]
    \item $G = \SL(2n)$, $K = \mathrm{S}(\GL(n) \times \GL(n)) \eqdef (\GL(n) \times \GL(n)) \cap G$, and $G/P$ the Grassmannian $\Gr(n, 2n)$ of $n$-planes in $\CC^{2n}$;
    \item $G = \Sp(2n)$, $K = \GL(n)$, and $G/P$ the Lagrangian Grassmannian $\LG(2n)$, the subvariety of $\Gr(n,2n)$ consisting of those $n$-planes on which a fixed nondegenerate skew-symmetric form on $\CC^{2n}$ vanishes;
    \item For $n$ even, $G = \SO(2n)$, $K = \GL(n)$, and $G/P$ the orthogonal Grassmannian $\OG(2n)$, one component of the subvariety of $\Gr(n,2n)$ consisting of those $n$-planes on which a fixed nondegenerate symmetric form on $\CC^{2n}$ vanishes.
\end{enumerate}
In each case, we give descriptions in terms of rank conditions for those $K$-orbits on $G/P \times K/B_K$ intersecting a certain open dense subset of $G/P$, and give formulas for the cohomology classes Poincar\'e dual to their closures.

We write $\Gr(n,2n)$ for the Grassmannian of $n$-planes in $\CC^n \oplus {\CC^n}^*$. Given a linear map $f : \CC^n \to {\CC^n}^*$, its graph $\graph(f) \eqdef \{(v, f(v)) : v \in \CC^n\}$ is a point in $\Gr(n,2n)$. The map $\Hom(\CC^n, {\CC^n}^*) \to \Gr(n,2n)$, $f \mapsto \graph(f)$, is an open embedding; let $\grGr$ be the image of the invertible maps. Let $\Fl(n)$ denote the variety of complete flags in $\CC^n$, so $K/B_K = \Fl(n) \times \Fl(n)$ in case (1) above. If $M$ is a matrix, let $M_{[i][j]}$ denote its upper-left $i \times j$ corner. Identify a permutation $w \in S_n$ with the permutation matrix having $1$'s in positions $(i, w(i))$, and let $[n] \eqdef \{1,2,\ldots,n\}$.

\begin{thm}[Proposition~\ref{prop:orbit-description-Gr} and Theorem~\ref{thm:back-stable-rep}] \label{thm:SL-classes} The $\mathrm{S}(\GL(n) \times \GL(n))$-orbits on $\grGr \times \Fl(n) \times \Fl(n)$ are the sets
\begin{equation*}
\openincGX_w \eqdef \{(\graph(f), E_\bullet, F_\bullet) : \text{$\rank(F_j \hookrightarrow \CC^n \xrightarrow{f} {\CC^n}^* \twoheadrightarrow E_i^*) = \rank w_{[i][j]}$ for $i,j \in [n]$}\}
\end{equation*}
for all $w \in S_n$. The integral cohomology class $[\incGX_w]$ Poincar\'e dual to the Zariski closure $\incGX_w$ is represented by the \emph{back-stable double Schubert polynomial} $\bS_w(x,-y)$.
  \end{thm}
Despite the name, a back-stable double Schubert polynomial is a formal power series, obtained as a limit of double Schubert polynomials; see Definition~\ref{defn:schubert}. Back-stable Schubert polynomials were introduced by Lam, Lee, and Shimozono in the context of Schubert classes in infinite flag varieties \cite{back-stable-schubert}; we do not know an explanation of Theorem~\ref{thm:SL-classes} from this point of view.

The fiber of $\incGX_w$ in $\Fl(n)$ over a fixed $(\graph(f), E_\bullet) \in \grGr \times \Fl(n)$ is a Schubert variety, and more generally the fiber of $\incGX_w$ in $\Fl(n) \times \Fl(n)$ over a fixed $\graph(f) \in \grGr$ is a double Schubert variety as described in \cite{anderson-double-schubert} and \cite{geometric-littlewood-richardson}. On the other hand, the fiber of $\incGX_w$ in $\Gr(n,2n)$ over a fixed $(E_\bullet, F_\bullet) \in \Fl(n) \times \Fl(n)$ is a \emph{graph Schubert variety} as defined by Knutson, Lam, and Speyer \cite{positroidjuggling}; accordingly, we call $\incGX_w$ a \emph{universal graph Schubert variety}. They showed that the class of a graph Schubert variety is represented by a Stanley symmetric function (see Definition~\ref{defn:schubert}). An appropriate specialization in Theorem~\ref{thm:SL-classes} gives a new proof of this fact.

\begin{defn} A linear map $f : {\CC^n} \to {\CC^n}^*$ is \emph{symmetric} if $f(v)(w) = f(w)(v)$ for $v, w \in {\CC^n}$, and \emph{skew-symmetric} if $f(v)(w) = -f(w)(v)$. \end{defn}

There are canonical (up to sign) nondegenerate symmetric and skew-symmetric forms $(-,-)^+$ and $(-,-)^-$ on ${\CC^n} \oplus {\CC^n}^*$, defined by $((v_1, \omega_1), (v_2, \omega_2))^\pm = \omega_1(v_2) \pm \omega_2(v_1)$. 
We take $\O(2n)$ and $\Sp(2n)$ to be the subgroups of $\GL(2n) = \GL({\CC^n} \oplus {\CC^n}^*)$ preserving $(-,-)^+$ and $(-,-)^-$. The \emph{Lagrangian Grassmannian} is the closed subvariety
\begin{equation*}
\LG(2n) \eqdef \{U \in \Gr(n,2n) : (U, U)^- = 0\}
\end{equation*}
of $\Gr(n,2n)$; it is a homogeneous $\Sp(2n)$-variety. The variety of points $U \in \Gr(n,2n)$ with $(U, U)^+ = 0$ has two irreducible components; the component containing $\CC^n \oplus 0$ is the \emph{orthogonal Grassmannian} $\OG(2n)$, and it is a homogeneous $\SO(2n)$-variety.

\begin{prop} A linear map $f : {\CC^n} \to {\CC^n}^*$ is symmetric if and only if $\graph(f) \in \LG(2n)$, and skew-symmetric if and only if $\graph(f) \in \OG(2n)$. \end{prop}

Let $\grLG$ be the open set of graphs of invertible symmetric linear maps ${\CC^n} \to {\CC^n}^*$ in $\LG(2n)$. Let $\grOG$ be the open set of graphs of invertible skew-symmetric linear maps ${\CC^n} \to {\CC^n}^*$ in $\OG(2n)$, assuming $n$ is even. In the next theorem, we view $\GL(n)$ as a subgroup of $\Sp(2n)$ and of $\SO(2n)$ via the embedding $g \mapsto (g^{-1})^* \oplus g$.

\begin{thm}[Proposition~\ref{prop:orbit-description-LG-OG} and Theorems~\ref{thm:back-stable-rep-LG} and \ref{thm:back-stable-rep-OG}] \label{thm:LG-OG-preview} The $\GL(n)$-orbits on $\grLG \times \Fl(n)$ are the sets
\begin{equation*}
\openincSGX_y \eqdef \{(\graph(f), E_\bullet) : \text{$\rank(E_j \hookrightarrow {\CC^n} \xrightarrow{f} {\CC^n}^* \twoheadrightarrow E_i^*) = \rank y_{[i][j]}$ for $i,j \in [n]$}\}
\end{equation*}
for $y \in S_n$ an involution. For $n$ even, the $\GL(n)$-orbits on $\grOG \times \Fl(n)$ are the sets 
\begin{equation*}
\openincSSGX_z \eqdef \{(\graph(f), E_\bullet) : \text{$\rank(E_j \hookrightarrow {\CC^n} \xrightarrow{f} {\CC^n}^* \twoheadrightarrow E_i^*) = \rank z_{[i][j]}$ for $i,j \in [n]$}\}
\end{equation*}
for $z \in S_n$ a fixed-point-free involution. The cohomology classes $[\incSGX_y]$ and $[\incSSGX_z]$ are represented by $2^{\twocyc(y)}\ibS_y$ and $\ibSfpf_z$, respectively, where $\twocyc(y)$ is the number of 2-cycles in $y$ and $\ibS_y$ and $\ibSfpf_z$ are \emph{back-stable involution Schubert polynomials}.
\end{thm}
Let $\I_n$ be the set of involutions in $S_n$, and $\Ifpf_n \subseteq \I_n$ the subset of fixed-point-free involutions.  The \emph{involution Schubert polynomials} $\iS_y$ and $\iSfpf_z$ were introduced by Wyser and Yong \cite{wyser-yong-orthogonal-symplectic}, who showed that the polynomials $\{2^{\twocyc(y)} \iS_y : y \in \I_n\}$ represent the classes of the $\O(n)$-orbit closures on $\Fl(n)$, and the polynomials $\{\iSfpf_z : z \in \Ifpf_n\}$ represent the classes of the $\Sp(n)$-orbit closures on $\Fl(n)$. The back-stable involution Schubert polynomials $\ibS_y$ and $\ibSfpf_z$ are obtained from $\iS_y$ and $\iSfpf_z$ by a limiting process; see Definition~\ref{defn:inv-schubert}.

The connection to our situation is as follows. Fix an invertible symmetric map $f : {\CC^n} \to {\CC^n}^*$, and let $\O(n) \subseteq \GL(n)$ be the subgroup preserving the symmetric form $(v, w) \mapsto f(v)(w)$. Then as $y$ ranges over $\I_n$, the fibers in $\Fl(n)$ of the various $\openincSGX_y$ over $\graph(f)$ are exactly the $\O(n)$-orbits on $\Fl(n)$. Similarly, if $f$ is skew-symmetric, then for $z \in \Ifpf_n$, the fibers of the various $\openincSSGX_z$ over $\graph(f)$ are the $\Sp(n)$-orbits on $\Fl(n)$. This reversal explains why we write representatives for $[\incSGX_y]$ and $[\incSSGX_z]$ as $2^{\twocyc(y)}\iS_y$ and $\iSfpf_z$ respectively, which at first may appear backwards.

\begin{defn}
The \emph{involution graph Schubert variety} $\incSGX_y(E_\bullet) \subseteq \LG(2n)$ associated to $y \in \I_n$ is the fiber of $\incSGX_y$ over a fixed flag $E_\bullet \in \Fl(n)$. The \emph{fixed-point-free involution graph Schubert variety} associated to $z \in \Ifpf_n$ is the fiber of $\incSSGX_y(E_\bullet) \subseteq \OG(2n)$  over a fixed flag $E_\bullet \in \Fl(n)$.  Explicitly, $\incSGX_y(E_\bullet)$ is the closure of
\begin{equation*}
      \{\graph(f) \in \grLG : \text{$\rank(E_j \hookrightarrow {\CC^n} \xrightarrow{f} {\CC^n}^* \twoheadrightarrow E_i^*) = \rank y_{[i][j]}$ for $i,j \in [n]$, $f$ symmetric}\}
\end{equation*}
and $\incSSGX_z(E_\bullet)$ is the closure of
\begin{equation*}
      \{\graph(f) \in \grOG : \text{$\rank(E_j \hookrightarrow {\CC^n} \xrightarrow{f} {\CC^n}^* \twoheadrightarrow E_i^*) = \rank z_{[i][j]}$ for $i,j \in [n]$, $f$ skew-symmetric}\}.
\end{equation*}
\end{defn}
A corollary of Theorem~\ref{thm:LG-OG-preview} is that $[\incSGX_y(E_\bullet)] \in H^*(\LG(2n),\ZZ)$ is represented by the \emph{involution Stanley symmetric function} $2^{\twocyc(y)}\iF_y$ introduced in \cite{HMP1}, and that $[\incSSGX_z(E_\bullet)] \in H^*(\OG(2n),\ZZ)$ is represented by the fixed-point-free involution Stanley symmetric function $\iFfpf_z$. It was shown algebraically and combinatorially in \cite{HMP4, HMP5} that $2^{\twocyc(y)}\iF_y$ and $\iFfpf_z$ are positive integer combinations of Schur's $Q$-functions and $P$-functions, respectively. By work of Pragacz \cite{pragacz-LG-OG}, this positivity means that these symmetric functions represents cohomology classes of subvarieties in $\LG(2n)$ and $\OG(2n)$, and part of the motivation for the current work was to find such subvarieties.

Working in the other direction, Theorem~\ref{thm:LG-OG-preview} together with Pragacz's results provides a new proof that $2^{\twocyc(y)}\iF_y$ is Schur $Q$ positive and $\iFfpf_z$ is Schur $P$ positive (Corollary~\ref{cor:schur-Q-positivity}). Their Schur $Q$ and $P$ expansions can be developed by explicit recurrences, found in \cite{HMP4, HMP5}, that are analogous to the ``transition recurrences'' of Lascoux and Sch\"utzenberger \cite{lascouxschutzenbergertree}. As a special case, every product of Schur $Q$ or $P$ functions can be written as some $2^{\twocyc(y)}\iF_y$ or $\iFfpf_z$, and so such transition recurrences give new Littlewood-Richardson rules for these families of symmetric functions, or equivalently for the Schubert bases of the integral cohomology of $\LG(2n)$ and $\OG(2n)$. Our hope is that the geometric perspective developed here will be helpful in finding similar rules in other cohomology theories---for instance, in the currently open problem of describing the Schubert structure constants of the K-theory ring $K_0(\LG(2n))$.

Other results on involution Stanley symmetric functions from \cite{HMP4, HMP5} are also clarified by the geometric perspective. For instance, ``I-Grassmannian'' and ``fpf-I-Grassmannian'' involutions were singled out on combinatorial grounds as base cases for the transition recurrences mentioned above, filling the role played by Grassmannian permutations in the classical case. The special role of these involutions is clear from the geometry: the involution graph Schubert varieties they index are simply Schubert varieties in $\LG(2n)$ and $\OG(2n)$.

Pfaffian formulas for involution Schubert polynomials and Stanley symmetric functions were given in \cite{HMP4, HMP5}, in the case of I-Grassmannian and fpf-I-Grassmannian involutions. We improve on these formulas by adding a missing base case and generalizing them to vexillary involutions. This is done by realizing $\incSGX_z(E_\bullet)$ and $\incSSGX_z(E_\bullet)$ as type C and type D Grassmannian degeneracy loci in the sense of Kazarian \cite{kazarian} and Anderson and Fulton \cite{anderson-fulton}. The interpretation of $\incGX_w$, $\incSGX_z$, and $\incSSGX_z$ as certain degeneracy loci when $w$ and $z$ are vexillary is a key element in our proofs of Theorems~\ref{thm:SL-classes} and \ref{thm:LG-OG-preview} as well.

\subsection{Outline}
In \S\ref{sec:prelim}, we recall basic facts about cohomology rings of flag varieties, degeneracy loci, and Schubert varieties, and prove some combinatorial lemmas on vexillary permutations. In \S\ref{sec:SL-SL-orbits}, we characterize the universal graph Schubert varieties $\incGX_w$ as closures of certain $\GL(n) \times \GL(n)$-orbits on $\Gr(n,2n) \times \Fl(n) \times \Fl(n)$, and show that their classes $[\incGX_w]$ are represented by double back-stable Schubert polynomials, proving Theorem~\ref{thm:SL-classes}. Sections~\ref{sec:inv} and \ref{sec:GL-orbits} recapitulate this story for involution graph Schubert varieties: in particular, in \S\ref{sec:GL-orbits} we prove Theorem~\ref{thm:LG-OG-preview} and give Pfaffian formulas for $\iS_z$ when $z$ is vexillary. The notion of ``vexillary'' and accompanying Pfaffian formulas are more delicate in the fixed-point-free case, and the necessary combinatorics is developed in \S\ref{sec:grassmannians}, where the particular case of I-Grassmannian involutions is also investigated. In Section~\ref{sec:tableaux}, we combine the results of \S\ref{sec:GL-orbits} and formulas of Ivanov to express $\iS_z$ in terms of shifted tableaux when $z$ is vexillary.

\subsection*{Acknowledgements}
I am grateful to Bill Fulton for teaching a course on degeneracy loci exactly when I needed to learn about them; I also thank Zach Hamaker, Thomas Lam, and Eric Marberg for helpful conversations, and Mark Shimozono for asking questions that eventually motivated some of the main results here.

\newpage

\section{Preliminaries}
\label{sec:prelim}

\subsection{Cohomology} \label{subsec:cohom}
Let $X$ be a smooth complex variety. Throughout we write $H^*(X)$ for the integral singular cohomology ring $H^*(X, \ZZ)$. Suppose $\pi : \vb E \twoheadrightarrow X$ is a complex vector bundle. For $x \in X$, we write ${\vb E}_x$ for the fiber $\pi^{-1}(x)$, a complex vector space. The \emph{total Chern class} $c(\vb E)$ of $\vb E$ is an element of $H^*(X)$, not necessarily homogeneous, which is zero outside degrees $0, 2, 4, \ldots, 2\rank(\vb E)$. The \emph{$d$\th Chern class} $c_d(\vb E)$ or $c(\vb E)_d$ is the degree $2d$ component of $c(\vb E)$. The properties of Chern classes we need are:
\begin{enumerate}[(a)]
    \item $c_0(\vb E) = 1$.
    \item $c(\vb F) = c(\vb E)c(\vb G)$ when $0 \to \vb E \to \vb F \to \vb G \to 0$ is a short exact sequence.
    \item $c_d(\vb E^*) = (-1)^d c_d(\vb E)$.
    \item If $\vb E$ is trivial, then $c(\vb E) = 1$.
\end{enumerate}
Since $H^*(X)$ is zero in large enough degree, $c_0(\vb E) = 1$ implies that $c(\vb E)$ is a unit. Part (b) above then says that $c(\vb E \oplus \vb F) = c(\vb E)c(\vb F)$ and $c(\vb E / \vb F) = c(\vb E)/c(\vb F)$.

Define alphabets
\begin{gather*}
    x_{+} = \{x_i : 0 < i \in \ZZ\} \quand x_{-} = \{x_i : 0 > i \in \ZZ\} \quand x = x_{-} \cup x_{+} \quand \\
        x_{a..b} = \{x_i : a \leq i \leq b, i \neq 0\} \quad \text{(for integers $a \leq b$)}.
\end{gather*}
Let $\Fl(n)$ be the set of complete flags $E_\bullet$ in $\CC^n$, i.e. chains $E_1 \subseteq E_2 \subseteq \cdots \subseteq E_n = \CC^n$ where $E_i$ is an $i$-dimensional linear subspace. For each $i$ there is a \emph{tautological vector bundle} $\vb{E}_i \twoheadrightarrow \Fl(n)$, whose fiber $(\vb{E}_i)_{E_{\bullet}}$ over a point $E_{\bullet} \in \Fl(n)$ is the subspace $E_i$. Borel showed that the map
\begin{equation} \label{eq:borel-isomorphism}
    \ZZ[x_1, \ldots, x_n] / (e_1(x_{1 \cdots n}), \ldots, e_n(x_{1 \cdots n})) \to H^*(\Fl(n))
\end{equation}
sending $x_i \mapsto c_1((\vb E_i / \vb E_{i-1})^*)$ for $i = 1, \ldots, n$ is a well-defined isomorphism, where $e_d$ is the degree $d$ elementary symmetric function \cite{borelcohom}. We write $\vb{E}_\bullet$ and $\vb{F}_{\bullet}$ for the tautological flags of bundles over the two factors of $\Fl(n) \times \Fl(n)$, and let $y_i = c_1((\vb F_i / \vb F_{i-1})^*)$, so that members of $H^*(\Fl(n) \times \Fl(n)) \simeq H^*(\Fl(n)) \otimes H^*(\Fl(n))$  can be represented by polynomials in $x_+$ and $y_+$.

The projection $\Fl(n) \to \Gr(k, n)$ sending $E_\bullet \mapsto E_k$ induces as its pullback an inclusion $H^*(\Gr(k,n)) \hookrightarrow H^*(\Fl(n))$ whose image is the subring of $(S_k \times S_{n-k})$-invariants in  $\ZZ[x_1, \ldots, x_n] / (e_1(x_{1 \cdots n}), \ldots, e_n(x_{1 \cdots n}))$. This subring is isomorphic to $\Lambda / (e_{k+1}, e_{k+2}, \ldots)$, where $\Lambda$ is the ring of symmetric functions over $\ZZ$. With this identification, the dual tautological bundle $\vb G^* \twoheadrightarrow \Gr(k,n)$ has Chern classes $e_0, e_1, \ldots, e_k$.

However, in our setting it seems more natural to consider the maps $\Fl(n) \times \Fl(n) \to \Gr(n,2n)$ sending $(E_\bullet, F_\bullet) \mapsto F_i \oplus (\CC^n/E_i)^*$, for each $i \in [n]$. Under this map, the dual tautological bundle $\vb{G}^* \twoheadrightarrow \Gr(n,2n)$ pulls back to $\vb F_i^* \oplus (\CC^n/\vb E_i)$, so the induced map $H^*(\Gr(n,2n)) \to H^*(\Fl(n)) \otimes H^*(\Fl(n))$ sends
\begin{equation} \label{eq:super-e}
    c(\vb{G}^*) \mapsto c( \vb F_i^* \oplus (\CC^n/\vb E_i)) = \frac{c(\vb F_i^*)}{c(\vb E_i)} = \frac{(1+y_1)\cdots (1+y_i)}{(1-x_1) \cdots (1-x_i)},
\end{equation}
hence $c_d(\vb{G}^*) \mapsto \sum_{a+b=d} h_a(x_{1 \cdots i}) e_b(y_{1\cdots i})$, where $h_a$ is the degree $a$ complete homogeneous symmetric function. This suggests the next definition.

\begin{defn} Let $\Delta : \Lambda \to \Lambda \otimes \Lambda$ be the coproduct on symmetric functions defined by $\Delta(e_d) = \sum_{a+b=d} e_a \otimes e_b$, and $\omega : \Lambda \to \Lambda$ the ring involution sending $e_d$ to $h_d$. The ring $\supersymm$ of \emph{supersymmetric functions} is the image of $(\omega \otimes \id) \circ \Delta$ in $\Lambda \otimes \Lambda$.
\end{defn}
 We view $\sum_i f_i \otimes g_i \in \supersymm$ as the formal power series $\sum_i f_i(x_-)g_i(y_-)$. For $f \in \Lambda$, write $f(x\superslash y) \eqdef ((\omega \otimes \id) \circ \Delta)(f)$; so, for instance, $e_d(x\superslash y) = \sum_{a+b=d} h_a(x) e_b(y)$. As suggested by \eqref{eq:super-e}, we identify $H^*(\Gr(k,n))$ with a quotient of $\supersymm$ by sending $e_d(x\superslash y)  \mapsto c_d(\vb G^*)$. Equivalently, $\sum_d h_d(x \superslash y)$ represents $1/c(\vb G)$.
\begin{rem}
    It is more common to define supersymmetric functions as the image of one of the maps $f \mapsto (1 \otimes \omega)(\Delta(f))$ or $f \mapsto (1 \otimes (-1)^{\deg(f)}\omega)(\Delta(f))$, and to write $f(x/y)$ for the image of $f$. We have used the notation $f(x\superslash y)$ instead to reflect our different convention.
\end{rem}

 Since Grassmannians and flag varieties have no odd-dimensional cohomology, the K\"unneth theorem and universal coefficient theorem imply that the natural map
\begin{equation*} H^*(\Gr(n,2n)) \otimes H^*(\Fl(n)) \otimes H^*(\Fl(n)) \to H^*(\Gr(n,2n) \times \Fl(n) \times \Fl(n)) \end{equation*}
is an isomorphism. Thus, classes in $H^*(\Gr(n,2n) \times \Fl(n) \times \Fl(n))$ can be represented by members of $\supersymm \otimes \ZZ[x_+] \otimes  \ZZ[y_+]$, which we view as formal power series in $x \cup y$. 

\subsection{Schubert polynomials and Stanley symmetric functions}

\begin{defn} A \emph{compatible sequence} for a word $a = a_1 \cdots a_\ell$ is a weakly increasing word $i_1 \leq \cdots \leq i_{\ell}$ with entries in $\ZZ \setminus \{0\}$ such that for each $j$, (1) $i_j \leq a_j$, and (2) if $a_j < a_{j+1}$ then $i_j < i_{j+1}$.  Let $\comp(a)$ be the set of compatible sequences for $a$.
\end{defn}
Our definition of compatible sequence is slightly different from usual (e.g. \cite{bergeron-billey}), in that $i$ is typically required to have positive entries.

Let $\Red(w)$ be the set of \emph{reduced words} of a permutation $w$: the minimal-length words $a_1 \cdots a_\ell$ such that $s_{a_1} \cdots s_{a_\ell} = w$, where $s_i$ is the transposition $(i,i{+}1)$. If $a$ is a word and $m \in \ZZ$, we write $m \leq a$ to mean that $m \leq a_i$ for each $i$.
\begin{ex} \label{ex:rc}
    We use bold to distinguish reduced words and compatible sequences from permutations. For instance, $\Red(2143) = \{{\bf 13}, {\bf 31}\}$ and 
    \begin{equation*}
        \comp({\bf 13}) = \{{\bf ij} : i < j, i \leq 1, j \leq 3\} \quad \text{and} \quad \comp({\bf 31}) = \{ {\bf ij} : i \leq j, i \leq 3, j \leq 1\}.
    \end{equation*}
\end{ex}

If $f$ is a formal power series in variables $x$, we write (for instance) $f\ev_{x_{-}\to 0}$ to indicate the result of setting the variables in $x_{-}$ to zero.
\begin{defn} \label{defn:schubert}
    The \emph{back-stable Schubert polynomial} \cite{back-stable-schubert} of $w \in S_n$ is
    \begin{equation*}
        \bS_w \eqdef \sum_{a \in \Red(w)} \sum_{\substack{i \in \comp(a)} } x_{i_1} \cdots x_{i_{\ell}}.
    \end{equation*}
    The \emph{Schubert polynomial} \cite{billeyjockuschstanley,lascouxschutzenbergerschubert} of $w$ is
    \begin{equation*}
        \fkS_w \eqdef \bS_w\ev_{x_{-} \to 0} = \sum_{a \in \Red(w)} \sum_{\substack{i \in \comp(a) \\ 1 \leq i} } x_{i_1} \cdots x_{i_{\ell}}.
    \end{equation*}
    The \emph{Stanley symmetric function} \cite{stanleysymm} of $w$ is
    \begin{equation*}
        F_w \eqdef \bS_w\ev_{x_{+} \to 0} = \sum_{a \in \Red(w)} \sum_{\substack{i \in \comp(a) \\ i \leq -1} } x_{i_1} \cdots x_{i_{\ell}}.
    \end{equation*}
\end{defn}
Despite its name, $\bS_w$ is not a polynomial but a formal power series in $x$. In accordance with their names, $\fkS_w$ is a polynomial in $x_+$, and $F_w$ is a symmetric function in $x_-$ (this symmetry is not obvious from our definition).

\begin{ex} As per Example~\ref{ex:rc},
    \begin{align*}
        \bS_{2143} &= \sum_{\substack{i < j \\ i \leq 1, j \leq 3}} x_i x_j + \sum_{\substack{i \leq j \\ i \leq 3, j \leq 1}} x_i x_j\\
        &= \left( x_1 x_2 + x_1 x_3 + e_1 \cdot (x_1+x_2+x_3) + e_2 \right) +  \left( x_1^2 + h_1 \cdot x_1 + h_2 \right),
    \end{align*}
    where we view the elementary and homogeneous symmetric functions $e_d$ and $h_d$ as formal power series in variables $x_-$. Setting $x_i$ to $0$ for $i < 0$ gives $\fkS_{2143} = x_1^2 + x_1 x_2 + x_1 x_3$, and setting $x_i$ to $0$ for $i > 0$ gives $F_{2143} = e_2 + h_2$.
\end{ex}

\begin{rem} Back stable Schubert polynomials can be defined in terms of ordinary Schubert polynomials. For $v \in S_m$ and $w \in S_n$, let $v \times w \in S_{m \times n}$ be the permutation
    \begin{equation*}
        i \mapsto \begin{cases}
            v(i) & \text{if $i \leq m$}\\
            w(i-m)+m & \text{if $m < i \leq n$}
        \end{cases}
    \end{equation*}
    Let $1_m \in S_m$ denote the identity permutation. Then $\bS_w = \lim_{m \to \infty} \fkS_{1_m \times w}(x_{-m\cdots n})$.  The next definition uses a similar approach.
\end{rem}

For $u, v, w \in S_n$, write $uv \doteq w$ to mean that $uv = w$ and $\ell(w) = \ell(u) + \ell(v)$. Here, $\ell(w)$ is the number of inversions of $w$, or equivalently the length of any $a \in \Red(w)$.
\begin{defn}[\cite{back-stable-schubert}, \S 4.4] The \emph{double Schubert polynomial} of $w \in S_n$ is
    \begin{equation*}
        \fkS_w(x;y) \eqdef \sum_{uv \doteq w} (-1)^{\ell(u)} \fkS_{u^{-1}}(y)\fkS_v(x).
    \end{equation*}
    The \emph{back-stable double Schubert polynomial} of $w$ is
    \begin{align*}
        \bS_w(x;y) &\eqdef \lim_{m \to \infty} \fkS_{1_m \times w}(x_{-m\cdots n}; y_{-m\cdots n}) = \sum_{uv \doteq w} (-1)^{\ell(u)} \bS_{u^{-1}}(y)\bS_v(x).
    \end{align*}
\end{defn}

The \emph{divided difference operator} $\partial_i$ sends $f \in R[x_1, x_2, \ldots]$ to
\begin{equation*}
    \partial_i f \eqdef \frac{f - s_i f}{x_i - x_{i+1}} \in R[x_1, x_2, \ldots],
\end{equation*}
where $R$ is a commutative ring and $s_i = (i,i{+}1)$ acts on $R[x_1, x_2, \ldots]$ by interchanging $x_i$ and $x_{i+1}$. If there are $x$-variables and $y$-variables, $\partial_i$ will always act on the $x$-variables and treat the $y$-variables as scalars. That is, we take the action of $\partial_i$ on $R[x_1, x_2, \ldots, y_1, y_2, \ldots]$ to be the action of $\partial_i$ on $S[x_1, x_2, \ldots]$ where $S = R[y_1, y_2, \ldots]$.
\begin{prop}[\cite{back-stable-schubert}, Theorem 4.6] \label{prop:schubert-recurrence}
    The back-stable double Schubert polynomials $\bS_w(x;y)$ satisfy the recurrence
\begin{equation*}
    \partial_i \bS_w(x;y) = \begin{cases}
        \bS_{ws_i}(x;y) & \text{if $\ell(ws_i) < \ell(w)$}\\
        0 & \text{otherwise}
    \end{cases}.
\end{equation*}
\end{prop}
Let $w_0 \in S_n$ be the reverse permutation $n(n-1)\cdots 21$. Any $w \in S_n$ can be reached starting from $w_0$ via a sequence of transformations $v \leadsto vs_i$ where $\ell(vs_i) < \ell(v)$, so Proposition~\ref{prop:schubert-recurrence} inductively determines every $\bS_w(x;y)$ once $\bS_{w_0}(x;y)$ is known.

\begin{defn} The \emph{Schubert variety} $\X_w$ associated to $w \in S_n$ with respect to a fixed flag $E_{\bullet}' \subseteq \CC^n$ is the closure of the \emph{Schubert cell}
\begin{equation*}
X_w \eqdef \{E_{\bullet} \in \Fl(n) : \rank(E_i \hookrightarrow \CC^n \twoheadrightarrow \CC^n/E_{n-j}') = \text{$\rank w_{[i][j]}$ for $i,j \in [n]$}\}.
\end{equation*}
\end{defn}
Suppose we represent a flag $E_{\bullet} \subseteq \CC^n$ by a matrix $A$ so that $E_i$ is the span of the first $i$ rows of $A$ for each $i$. Taking $E_i'$ to be the span of the standard basis vectors $e_n, \ldots, e_{n-i+1}$, the Schubert cell $X_w$ consists of flags with matrix $A$ such that $\rank A_{[i][j]} = \rank w_{[i][j]}$ for $i, j \in [n]$. To obtain the closed Schubert variety $\X_w$, replace the equalities $\rank A_{[i][j]} = \rank w_{[i][j]}$ with inequalities $\rank A_{[i][j]} \leq \rank w_{[i][j]}$.  Schubert varieties are irreducible, and $\codim \X_w = \ell(w)$.

Lascoux and Sch\"utzenberger introduced the Schubert polynomial $\fkS_w$ as a representative for the class $[\X_w] \in H^*(\Fl(n))$ Poincar\'e dual to $\X_w$ \cite{lascouxschutzenbergerschubert}. The Schubert cells $\{X_w : w \in S_n\}$ are the cells of a CW decomposition of $\Fl(n)$, and consequently $\{\fkS_w : w \in S_n\}$ is a $\ZZ$-basis of $\ZZ[x_1, \ldots, x_n]/(e_1(x_{1\cdots n}), \ldots, e_n(x_{1\cdots n})) \simeq H^*(\Fl(n))$.

The Schubert polynomials satisfy an important stability property: $\fkS_{w \times 1} = \fkS_{w}$ for any $w \in S_n$. Viewing $S_n$ as a subgroup of $S_{n+1}$ via the embedding $w \mapsto w \times 1$, let $S_{\infty} \eqdef \bigcup_{n\geq 0} S_n$. The stability property of Schubert polynomials means it is well-defined to write $\fkS_{w}$ for $w \in S_{\infty}$. Moreover, $\{\fkS_w : w \in S_{\infty}\}$ forms a basis of $\ZZ[x_1, x_2, \ldots]$ \cite[Ch. 10]{youngtableaux}.

\begin{lem} \label{lem:unique-stable-rep} Let $\alpha_i \in H^*(\Fl(n_i))$ be a sequence of classes where $(n_i)$ is a sequence tending toward $\infty$. There is at most one polynomial $f$ which represents every class $\alpha_i$. \end{lem}
    \begin{proof}
        Suppose $f$ represents every $\alpha_i$. Since $\{\fkS_w : w \in S_{\infty}\}$ is a basis of $\ZZ[x_1, x_2, \ldots]$, we can write $f = \sum_{w \in S_{n_i}} c_w \fkS_w$ for some sufficiently large $i$. Since $f$ represents $\alpha_i$, we then have $\alpha_i = \sum_{w \in S_{n_i}} c_w [\X_w]$. The classes $\{[\X_w] : w \in S_{n_i}\}$ are linearly independent, so the coefficients $c_w$ can be determined from $\alpha_i$.
    \end{proof}

\subsection{Vexillary permutations} \label{subsec:vex}

\begin{defn}
    The \emph{Rothe diagram} of $w \in S_n$ is the set
    \begin{equation*}
        D(w) \eqdef \{(i,j) \in [n] \times [n] : j < w(i), i < w^{-1}(j)\}.
    \end{equation*}
\end{defn}

\begin{defn}
    The \emph{essential set} of a set $D \subseteq \NN \times \NN$ is
    \begin{equation*}
        \Ess(D) \eqdef \{(i,j) \in D : (i,j+1) \notin D \text{ and } (i+1,j) \notin D\}.
    \end{equation*}
    That is, $\Ess(D)$ is the set of southeast corners of connected components of $D$, viewing two elements of $\NN \times \NN$ as connected if they are vertically or horizontally adjacent.
\end{defn}

\begin{defn} A permutation $w \in S_n$ is \emph{vexillary} if it avoids the pattern $2143$, i.e. there do not exist $i < j < k < l$ in $[n]$ with $w(j) < w(i) < w(l) < w(k)$. \end{defn}

Let $\SWNE$ be the partial order on $\NN \times \NN$ increasing from southwest to northeast, meaning that $(i,j) \SWNE (i',j')$ if and only if $i \geq i'$ and $j \leq j'$.  If $i \in \NN$ and $S \subseteq \NN$, we write $D_{iS}(w)$ for $\{j \in S : (i,j) \in D(w)\}$, or simply $D_{iS}$ when $w$ is understood.

The equivalences of (a) with parts (b) and (c) in the next lemma are due to Fulton \cite{fulton-double-schubert} and Wachs \cite{wachs-schubert}, respectively.
\begin{lem} \label{lem:vex-chain} The following are equivalent:
\begin{enumerate}[(a)]
\item $w \in S_n$ is vexillary.
\item $\Ess(D(w))$ is a chain under $\SWNE$.
\item The sets $D_{i\NN}(w)$ for $i \in \NN$ are totally ordered under inclusion.
\end{enumerate}
\end{lem}

\begin{ex}
    Let $w$ be the vexillary permutation $35142$. Each $\times$ in the following diagram is a point $(i, w(i))$ in matrix coordinates, with the points of the Rothe diagram $D(w)$ marked by squares: they are the points directly left of a $\times$ and directly above a $\times$. Elements of $\Ess(D(w))$ are marked by black squares. All points are drawn in matrix coordinates, with $(1,1)$ at the upper left:
    \begin{equation*}  \arraycolsep=2pt \def\arraystretch{0.9}
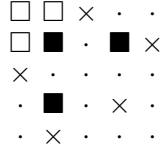
        \begin{array}{ccccc}
            \sq & \sq & \times & \cdot & \cdot \\
            \sq & \blacksquare & \cdot & \blacksquare & \times \\
            \times & \cdot & \cdot & \cdot & \cdot \\
            \cdot & \blacksquare & \cdot & \times & \cdot\\
            \cdot & \times & \cdot & \cdot & \cdot
        \end{array}
    \end{equation*}
    By contrast, the subsequence $3154$ of $v = 31524$ is a $2143$ pattern, the diagram
    \begin{equation*}  D(v) = \arraycolsep=2pt \def\arraystretch{0.9}
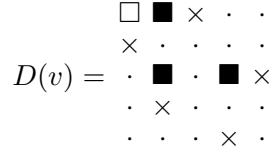
        \begin{array}{ccccc}
            \sq & \blacksquare & \times & \cdot & \cdot \\
            \times & \cdot & \cdot & \cdot & \cdot \\
            \cdot & \blacksquare & \cdot & \blacksquare & \times \\
            \cdot & \times & \cdot & \cdot & \cdot\\
            \cdot & \cdot & \cdot & \times & \cdot
        \end{array}
    \end{equation*}
    has an essential set element $(1,2)$ strictly northwest of the essential set element $(3,4)$, and the sets $D_{1\NN}(v) = \{1,2\}$ and $D_{3\NN}(v) = \{2,4\}$ are incomparable under containment.
\end{ex}

\begin{defn}
    The \emph{code} of $w \in S_n$ is the list $c(w) = (c_1(w), \ldots, c_n(w))$ where $c_i(w) = \#\{j > i : w(j) < w(i)\} = |D_{i\NN}(w)|$. The \emph{shape} $\sh(w)$ of $w$ is the transpose of the partition obtained by sorting $c(w)$ and ignoring $0$'s.
\end{defn}
We note that the shape of $w$  is more commonly defined as $\sh(w)^t$.

\begin{ex} The code of $35142$ is $(2,3,0,1,0)$, and the shape is $(3,2,1)^t = (3,2,1)$. \end{ex}

In the remainder of this subsection, we prove some lemmas which we will need to extract the rank conditions defining certain degeneracy loci, described in the next subsection, from the combinatorics of Rothe diagrams and essential sets.

\begin{lem} \label{lem:code-ks} Suppose $w$ is vexillary. Write $\Ess(D(w)) = \{(i_1, j_1) \SWNEneq \cdots \SWNEneq (i_s, j_s)\}$, and let $k_p = j_p - \rank w_{[i_p][j_p]}$ for $p \in [s]$. Then $\{k_1, \ldots, k_s\} = \{c_1(w), \ldots, c_n(w)\} \setminus \{0\}$.
\end{lem}

\begin{proof}
    We will repeatedly use the fact that $c_i(w)$ is the number of elements of row $i$ of $D(w)$, and that $k_p$ is the number of elements in row $i_p$ and columns $[j_p]$. That is, $c_i(w) = |D_{i\NN}(w)|$ and $k_p = |D_{i_p[j_p]}(w)|$.
    
    Given $i \in [n]$ with $c_i(w) > 0$, let $j$ be maximal such that $(i,j) \in D(w)$. Let $i' \geq i$ be maximal such that $(i,j), (i+1,j), \ldots, (i',j) \in D(w)$. We proceed by proving a series of claims.
    \begin{enumerate}[(a)]
        \item $(i',j) \in \Ess(D(w))$: The maximality of $j$ means that $w^{-1}(j+1) \leq i \leq i'$, so that $(i',j+1) \notin D(w)$, and likewise the maximality of $i'$ means that $(i'+1,j) \notin D(w)$.
        \item $c_i(w) = |D_{i[j]}|$: By definition $c_i(w) = |D_{i\NN}|$, and $D_{i\NN} = D_{i[j]}$ by the choice of $j$.
        \item $c_i(w) \in \{k_1, \ldots, k_s\}$: By (a), $i' = i_p$ for some $p$. The fact that $D(w)$ contains $(r,j)$ for $i \leq r \leq i'$ implies that $w(r) > j$ for all such $r$, so $D_{i[j]} = D_{r[j]}$ for such $r$. In particular, taking $r = i'$ and using (b), $c_i(w) = |D_{i[j]}|= |D_{i'[j]}| = k_p$.
    \end{enumerate}

        Conversely, take $p \in [s]$. We want to find $i$ such that $k_p = c_i(w)$. Let $S = \{j > j_p : (i_p,j) \in D(w)\}$. If $S = \emptyset$, then $k_p = c_{i_p}(w)$ and we are done. Otherwise, let $i < i_p$ be maximal so that $j_p < w(i) < \min(S)$; such an $i$ exists because $(i_p,j_p) \in \Ess(D(w))$ implies that $w^{-1}(j_p+1)$ satisfies the conditions demanded of $i$. Now:
        \begin{enumerate}[(a)]
            \item $|D_{i[j_p]}| = c_i(w)$: Suppose not, so there is some $j > j_p$ with $(i,j) \in D(w)$. Then $w^{-1}(j)$ satisfies the conditions used to choose $i$, namely: $j_p < w(w^{-1}(j)) = j$; $j < w(i) < \min(S)$ because $(i,j) \in D(w)$; and $w^{-1}(j) < i_p$ because otherwise $j \in S$, contradicting $j < \min(S)$. However, $w^{-1}(j) > i$ because $(i,j) \in D(w)$, so this would contradict the maximality of $i$.
            \item If $i \leq r \leq i_p$ than $w(r) > j_p$: The choice of $i$ says that $w(i) > j_p$, and $w(i_p) > j_p$ because $(i_p,j_p) \in D(w)$, so assume $i < r < i_p$. Since $r < i_p$ and $(i_p,j_p) \in D(w)$, we have $w(r) \neq j_p$, so suppose for the sake of contradiction that $w(r) < j_p$. Because $(i_p, \min(S)) \in D(w)$ we have $i_p < w^{-1}(\min(S))$ and $\min(S) < w(i_p)$, and now $w$ contains a $2143$ pattern: $i < r < i_p < w^{-1}(\min(S))$ and $w(r) < j_p < w(i) < \min(S) < w(i_p)$. This contradicts the assumption that $w$ is vexillary.
            \item $k_p \in \{c_1(w), \ldots, c_n(w)\}$: Part (b) implies that $D_{i[j_p]} = D_{r[j_p]}$ for $i \leq r \leq i_p$. In particular, taking $r = i_p$ and using (a) gives $k_p = |D_{i_p[j_p]}| = |D_{i[j_p]}| = c_i(w)$. \qedhere
        \end{enumerate}
    \end{proof}

    \begin{lem} \label{lem:vex-shape} Suppose $w$ is vexillary. Write $\Ess(D(w)) = \{(i_1, j_1) \SWNEneq \cdots \SWNEneq (i_s, j_s)\}$, and let $k_p = j_p - \rank w_{[i_p][j_p]}$ for $p \in [s]$, and $k_0 = 0$. If $k_{p-1} < k \leq k_p$, then $\sh(w)_{k} = \sh(w)_{k_p} = i_p - \rank w_{[i_p][j_p]}$.
    \end{lem}

    \begin{proof}
        Again we proceed by proving a series of claims, whose aim is to establish that $\{i : c_i(w) \geq k_p\} = \{i \leq i_p : (i,j_p) \in D(w)\}$. Note that the size of the lefthand set is $\sh(w)_{k_p}$, while the size of the righthand set is $i_p - \rank w_{[i_p][j_p]}$.
        \begin{enumerate}[(a)]
        \item If $(i,j_p) \in D(w)$ and $i \leq i_p$, then $c_i(w) \geq k_p$: It is not hard to see that $D(w)$ is closed under taking northwest corners in the sense that if $(a,b) \SWNE (a',b')$ are in $D(w)$, then $(a',b) \in D(w)$ also. In particular, if $j \in D_{i_p[j_p]}$, then $(i_p,j) \SWNE (i,j_p)$, so $D(w)$ contains $(i,j)$. Therefore $k_p = |D_{i_p[j_p]}| \leq |D_{i[j_p]}| \leq c_i(w)$.
        
        \item If $(i,j_p) \notin D(w)$ and $i \leq i_p$, then $c_i(w) < k_p$: These assumptions force $w(i) < j_p$, so $|D_{i[j_p]}| = |D_{i\NN}| = c_i(w)$. Given that the sets $D_{i\NN}$ for $i \in \NN$ are totally ordered under containment by Lemma~\ref{lem:vex-chain}, $(i,j_p) \notin D(w)$ and $(i_p,j_p) \in D(w)$ means $D_{i[j_p]} \subsetneq D_{i_p[j_p]}$, hence  $c_i(w) = |D_{i\NN}| = |D_{i[j_p]}| < |D_{i_p[j_p]}| = k_p$. 
        \item If $i > i_p$, then $c_i(w) < k_p$: First let us see that $c_i(w) = |D_{i[j_p]}|$. If not, there is some $(i,j) \in D(w)$ with $j > j_p$. Then $(i_p,j_p)$ is strictly north and west of $(i,j)$, and hence of any essential set element in the same connected component of $D(w)$ as $(i,j)$. By Lemma~\ref{lem:vex-chain}, this contradicts the assumption that $w$ is vexillary.
        Now, since $(i_p,j_p) \in \Ess(D(w))$, we have $w(i_p+1) \leq j_p$, so $D(w)$ contains $(i_p, w(i_p+1))$ but not $(i, w(i_p+1))$. As in (b), this implies $D_{i[j_p]} \subsetneq D_{i_p[j_p]}$, and hence $c_i(w) = |D_{i[j_p]}| < |D_{i_p[j_p]}| = k_p$.
        \end{enumerate}
        Parts (a), (b), and (c) together prove that  $\{i : c_i(w) \geq k_p\} = \{i \leq i_p : (i,j_p) \in D(w)\}$, hence $\sh(w)_{k_p} = i_p - \rank w_{[i_p][j_p]}$ by comparing cardinalities. The distinct parts of $\sh(w)^t$ are $\{k_1, \ldots, k_s\}$ by Lemma~\ref{lem:code-ks}, so if $k_{p-1} < k \leq k_p$ then $\sh(w)_k = \sh(w)_{k_p}$.
\end{proof}

We conclude with a technical lemma to be used in \S \ref{subsec:vex-inv}.

\begin{lem} \label{lem:k-step} Suppose $w$ is vexillary. Write $\Ess(D(w)) = \{(i_1, j_1) \SWNEneq \cdots \SWNEneq (i_s, j_s)\}$, and let $k_p = j_p - \rank w_{[i_p][j_p]}$ for $p \in [s]$, and $k_0 = 0$. If $(i_p,j_p-1) \notin D(w)$ for some $p > 1$, then $k_p = k_{p-1}+1$. \end{lem}
    \begin{proof}
        We consider two cases:
        \begin{itemize}
            \item Suppose $j_{p-1} < j_p$. By the northwest closure property of $D(w)$, the cell $(i_p, j_{p-1})$ is in $D(w)$ given that $(i_{p-1}, j_{p-1})$ and $(i_p, j_p)$ are. If $(i, j_{p-1}) \notin D(w)$ for some $i_p < i < i_{p-1}$, then the connected component of $(i_p, j_{p-1})$  has a southeast corner $(i',j') \in \Ess(D(w))$ strictly above row $i_{p-1}$, and strictly left of column $j_p$ given that $(i_p,j_p-1) \notin D(w)$. But then $(i_{p-1}, j_{p-1}) \SWNEneq (i',j') \SWNEneq (i_p,j_p)$, which is impossible. We conclude that $D(w)$ contains $(i,j_{p-1})$ for all $i_p \leq i \leq i_{p-1}$. This implies that $w(i) > j_{p-1}$ for all such $i$, so $D_{i_{p-1}[j_{p-1}]} = D_{i_p[j_{p-1}]}$.
            
            A similar argument shows that $D(w)$ does \emph{not} contain $(i_p,j)$ whenever $j_{p-1} < j < j_p$: otherwise there would be an essential set cell strictly right of column $j_{p-1}$, but strictly left of column $j_p$ given that $(i_p,j_p-1) \notin D(w)$. So, $D_{i_p[j_p]} = D_{i_p[j_{p-1}]} \cup \{j_p\}$. We conclude that
            \begin{equation*}
                k_{p} = |D_{i_p[j_p]}| = 1 + |D_{i_p[j_{p-1}]}| = 1 + |D_{i_{p-1}[j_{p-1}]}| = 1 + k_{p-1}.
            \end{equation*}
            \item Suppose $j_{p-1} = j_p$, so $i_{p-1} < i_p$. Since $D(w)$ contains $(i_p,j_p)$ but not $(i_p,j_p-1)$ or $(i_p+1,j_p)$, we have $w^{-1}(j_p-1) < i_p$ and $w(i_p+1) < j_p$. This inequalities together actually imply $w^{-1}(j_p-1) < i_p$ and $w(i_p+1) < j_p-1$. Also, note that $i_p+1 < i_{p-1}$ since $D(w)$ contains $(i_{p-1},j_p)$ but not $(i_p+1,j_p)$. But now $w^{-1}(j_p-1) < i_p+1 < i_{p-1} < w^{-1}(j_p)$ and $w(i_p+1) < j_p-1 < j_p < w(i_{p-1})$, so $w$ contains a $2143$ pattern, a contradiction: this case cannot occur. \qedhere
         \end{itemize}
    \end{proof}

\subsection{Degeneracy locus formulas} \label{subsec:degeneracy-loci}

Let $\vb E_1 \hookrightarrow \cdots \hookrightarrow \vb E_n$ and $\vb F_n \twoheadrightarrow \cdots \twoheadrightarrow \vb F_1$ be sequences of vector bundles over a smooth variety $X$, where $\rank \vb E_i = \rank \vb F_i = i$. Let $f : \vb E_n \to \vb F_n$ be a bundle map, so we are given a linear map $f_x : (\vb E_n)_x \to (\vb F_n)_x$ for each $x \in X$. The corresponding \emph{degeneracy locus} $\overline\Omega_w$ labeled by $w \in S_n$ is the closure of the locus $\Omega_w$ of points in $X$ over which
\begin{equation}
    \rank\left(\vb E_i \hookrightarrow \vb E_n \xrightarrow{f} \vb F_n \twoheadrightarrow \vb F_j\right) = \rank w_{[i][j]}
\end{equation}
for $i, j \in [n]$. That is, $\Omega_w = \{x \in X : \text{$\rank(f_x : (\vb E_i)_x \to (\vb F_j)_x) = \rank w_{[i][j]}$ for $i,j \in [n]$}\}$. Fulton gave a formula for the class of $\overline \Omega_w$ when things are suitably generic.

\begin{thm}[\cite{fulton-double-schubert}] \label{thm:fulton-double-schubert} If $X$ is smooth and $\overline \Omega_w$ has codimension $\ell(w)$, then $[\overline\Omega_w] = \fkS_w(x'; y')$, where $x'_i = c_1((\vb E_i/\vb E_{i-1})^*)$ and $y'_i = c_1(\vb F_i^*/\vb F_{i-1}^*)$.
\end{thm}

\begin{ex}
    Take $\vb E_{\bullet}$ to be the tautological flag of bundles over $\Fl(n)$ as before. Fix a flag $E_\bullet' \in \Fl(n)$ and let $\vb F_i$ be the trivial bundle $\CC^n / E_{n-i}'$ over $\Fl(n)$; we adopt the common abuse of simply writing $V$ for the trivial bundle $V \times \Fl(n) \to \Fl(n)$. Let $f : \CC^n = \vb E_n \to \vb F_n = \CC^n$ be the identity map. Then $\overline \Omega_w = \X_w$. Under our conventions from \S \ref{subsec:cohom}, $x_i' = x_i$ and $y_i' = 0$, so Theorem~\ref{thm:fulton-double-schubert} implies that $[\X_w]$ is represented by the (single) Schubert polynomial $\fkS_w(x) = \fkS_w(x;0)$.
\end{ex}

\begin{lem} \label{lem:degeneracy-locus-essential-set} The degeneracy locus $\overline \Omega_w$ is the closure of the locus where 
    \begin{equation*}
        \rank\left(\vb E_i \hookrightarrow \vb E_n \xrightarrow{f} \vb F_n \twoheadrightarrow \vb F_j\right) = \rank w_{[i][j]}
    \end{equation*}
    for $(i,j) \in \Ess(D(w))$.
\end{lem}

This lemma is more or less equivalent to results of Fulton in \cite{fulton-double-schubert}, but is not quite stated in the same way, so we give a proof for completeness and because we will need a similar result in a different setting later.

\begin{proof} Define  $\Omega_w^{=,\text{ess}}$, $\Omega_w^{\leq,\text{ess}}$, $\Omega_w^{=}$, and $\Omega_w^{\leq}$ as the four loci in $X$ over which the ranks $\rank(\vb E_i \hookrightarrow \vb E_n \xrightarrow{f} \vb F_n \twoheadrightarrow \vb F_j)$ are either equal to or at most the ranks $\rank w_{[i][j]}$, either for $(i,j) \in \Ess(D(w))$ or for all $(i,j) \in [n] \times [n]$. Define $M_w^{=, \text{ess}}$, $M_w^{\leq,\text{ess}}$, $M_w^{=}$, and $M_w^{\leq}$ as the sets of $g \in \Hom(\CC^n, \CC^n)$ where the ranks $\rank(\CC^i \hookrightarrow \CC^n \xrightarrow{g} \CC^n \twoheadrightarrow \CC^n/\CC^{n-j})$ are similarly constrained. We want to show that $\overline{\Omega}_w^{\hspace{1pt}=,\text{ess}} = \overline{\Omega}_w^{\hspace{1pt}=}$.

First we prove the lemma in the case that $\vb E_i = \CC^i$ and $\vb F_i = \CC^n/\CC^{n-i}$ are trivial. Specifying a bundle map $f : \vb E_n \to \vb F_n$ is then equivalent to specifying a map  $\phi : X \to \Hom(\CC^n, \CC^n)$, and (for instance)
\begin{align*}
\Omega_w^{=} &= \{x \in X : \text{$\rank(\CC^i \hookrightarrow \CC^n \xrightarrow{\phi(x)} \CC^n \twoheadrightarrow \CC^{n}/\CC^{n-j}) = \rank w_{[i][j]}$ for $i,j \in [n]$}\}\\
& = \phi^{-1}(M_w^{=}).
\end{align*}
It therefore suffices to show that $\overline{M}_w^{\hspace{1pt}=,\text{ess}} = \overline{M}_w^{\hspace{1pt}=}$.  It is clear from the definitions that $M_w^{\leq, \text{ess}} \supseteq \overline{M}_w^{\hspace{1pt}=,\text{ess}} \supseteq \overline{M}_w^{\hspace{1pt}=}$. Fulton showed that $\overline{M}_w^{\hspace{1pt}=} = M^{\leq} = M^{\leq,\text{ess}}$ \cite[Proposition 3.3 and Lemma 3.10]{fulton-double-schubert}, so we must also have $\overline{M}_w^{\hspace{1pt}=,\text{ess}} = \overline{M}_w^{\hspace{1pt}=}$.

Now suppose $\vb E_\bullet$ and $\vb F_\bullet$ are trivial over some open set $U \subseteq X$. Replacing $X$ by $U$, the relevant loci are $\Omega_w^{=,\text{ess}} \cap U$ and $\Omega_w^{=} \cap U$, and their closures (in $U$) are $\overline{\Omega}_w^{\hspace{1pt}=,\text{ess}} \cap U$ and $\overline{\Omega}_w^{\hspace{1pt}=} \cap U$. The previous paragraph implies $\overline{\Omega}_w^{\hspace{1pt}=,\text{ess}} \cap U = \overline{\Omega}_w^{\hspace{1pt}=} \cap U$. Since $X$ is covered by open sets $U$ over which $\vb E_\bullet$ and $\vb F_\bullet$ are trivial, we conclude that $\overline{\Omega}_w^{\hspace{1pt}=,\text{ess}} = \overline{\Omega}_w^{\hspace{1pt}=}$.
\end{proof}

    Suppose $\vb{V}$ is a vector bundle over $X$ with a rank $r$ subbundle $\vb{G}$ and a flag of subbundles $\vb{H}_{\mu_1} \subseteq \cdots \subseteq \vb{H}_{\mu_s} \subseteq \vb V$, where $\rank \vb H_i = \rank \vb V - i$, so $\mu_1 \geq \cdots \geq \mu_s$. Also fix a sequence $0 = k_0 < k_1 < \cdots < k_s$ of integers. The \emph{Grassmannian degeneracy locus} $\Omega^{\Gr}$ with respect to this data is the closure of the locus in $X$ over which $\dim(\vb{G} \cap \vb{H}_{\mu_i}) = k_i$ for each $i$, or equivalently $\rank\left(\vb G \hookrightarrow \vb V \twoheadrightarrow \vb V / \vb H_{\mu_i}\right) = r - k_i$. For each $k \in [k_s]$, let $p$ be such that $k_{p-1} < k \leq k_p$, and define $\lambda_{k} = \mu_p - r + k_p$ and $c(k) = c(\vb{V}) \,/\, (c(\vb{G})c(\vb{H}_{\mu_p}))$.
    \begin{thm} \label{thm:kempf-laksov} If the sequence $\lambda$ just defined is a partition and $\Omega^{\Gr}$ has codimension $|\lambda| \eqdef \sum_i \lambda_i$, then $[\Omega^{\Gr}] = \det (c(k)_{\lambda_k+t-k})_{k,t \in [\ell(\lambda)]}$.\end{thm}

This formula is a modest generalization of a formula of Kempf and Laksov \cite{kempf-laksov}. Alternatively, it can be deduced from Theorem~\ref{thm:fulton-double-schubert} as follows. Let $n = \rank \vb V$ and define $w \in S_n$ by $w_k = \lambda_{r-k+1} + k$ for $k \leq r$ and $\{w_{r+1} < \cdots < w_n\} = [n] \setminus \{w_1, \ldots, w_r\}$. Then $\Ess(D(w)) = \{(r,\mu_p) : p \in [s]\}$ and $\rank w_{[r][\mu_p]} = r-k_p$ for $i \in [s]$. Without loss of generality one can assume that the partial flags $\vb{H}_\bullet$ and $\vb G$ can be completed to complete flags in $\vb V$, so $\Omega^{\Gr} = \Omega_w$ by Lemma~\ref{lem:degeneracy-locus-essential-set}. Theorem~\ref{thm:kempf-laksov} then follows from a similar determinantal formula for $\fkS_w(x;y)$; see \cite[Proposition 9.18]{fulton-double-schubert} or \cite{wachs-schubert}.

\section{$\GL(n) \times \GL(n)$-orbits on $\Gr(n,2n) \times \Fl(n) \times \Fl(n)$}
\label{sec:SL-SL-orbits}
\subsection{Description of orbits}

Let $\grGr \eqdef \{\graph(f) \mid \text{$f : \CC^n \to {\CC^n}^*$ linear and invertible}\}$. Write elements of $\Gr(n,2n) = \Gr(n, \CC^n \oplus {\CC^n}^*)$ as row spans of $n \times 2n$ matrices whose first $n$ columns are coordinates on $\CC^n$ and last $n$ columns are coordinates on ${\CC^n}^*$. The complement of $\grGr$ is the closed locus where the Pl\"ucker coordinates in columns $[n]$ and in columns $[n+1,2n]$ both vanish, so $\grGr$ is Zariski-open.  For $w \in S_n$, define $\openincGX_w \subseteq \grGr \times \Fl(n) \times \Fl(n)$ to be
\begin{equation*}
     \{(\graph(f), E_\bullet, F_\bullet) : \text{$\rank(F_j \xrightarrow{f} E_i^*) = \rank w_{[i][j]}$ for $i,j \in [n]$}\}.
\end{equation*}
The expression $F \xrightarrow{f} E^*$ here refers to the composition $F \hookrightarrow \CC^n \xrightarrow{f} {\CC^n}^* \twoheadrightarrow E^*$ for subspaces $E, F \subseteq \CC^n$ and a linear map $f : \CC^n \to {\CC^n}^*$. In this subsection we show that the $\openincGX_w$ are the $\GL(n) \times \GL(n)$-orbits on $\grGr \times \Fl(n) \times \Fl(n)$, or equivalently the $S(\GL(n) \times \GL(n))$-orbits, as discussed in the introduction.

Suppose $X = \prod_{j \in J} X_j$ is a Cartesian product of sets, $Y = \prod_{i \in I} X_i$ for some subset $I \subseteq J$, and $\pi : X \to Y$ is the projection. For $U \subseteq X$ and $y \in Y$, we write $U(y)$ for the projection of the fiber $\pi^{-1}(\{y\}) \cap U$ onto $\prod_{j \in J \setminus I} X_j$.

\begin{ex} \label{ex:GX-schubert-fiber}
Let $e_1, \ldots, e_n$ be the standard basis of $\CC^n$, with dual basis $e_1^*, \ldots, e_n^*$, and let $f : \CC^n \to {\CC^n}^*$ be linear with $e_i \mapsto e_i^*$ for $i \in [n]$. Fix a flag $E_\bullet^0 = F_\bullet^0 \in \Fl(n)$ defined by $E_i^0 = F_i^0 = \langle e_1, \ldots, e_i \rangle$ for $i\in [n]$, and write $F_\bullet = \rowspan_\bullet N$ to mean that $F_i$ is the span of rows $1, \ldots, i$ of a matrix $N$. With respect to the projection from $\openincGX_w$ onto the first two factors of $\grGr \times \Fl(n) \times \Fl(n)$, the fiber $\openincGX_w(\graph(f), E_\bullet^0)$ is
\begin{equation*}
    \{F_\bullet \in \Fl(n) : \text{$F_\bullet = \rowspan_\bullet N$ where $\rank N_{[j][i]} = \rank w_{[i][j]}$ for $i,j \in [n]$}\}.
\end{equation*}
This is the Schubert cell $X_{w^{-1}}$ with respect to the flag $E_\bullet^0$. Consider on the other hand the projection from $\grGr \times \Fl(n) \times \Fl(n)$ onto its first and third factors. Then since $\rank(F_j\xrightarrow{f}  E_i^*) = \rank(E_i \xrightarrow{f^*} F_j^*)$, the fiber $\openincGX_w(\graph(f), F_\bullet^0)$ is
\begin{equation*}
    \{E_\bullet \in \Fl(n) : \text{$E_\bullet = \rowspan_\bullet N$ where $\rank N_{[i][j]} = \rank w_{[i][j]}$ for $i,j \in [n]$}\}.
\end{equation*}
This is the Schubert cell $X_w$ with respect to $F_\bullet^0$. We will continue the convention used here of using $E_\bullet$ and $F_\bullet$ for the two coordinates of $\Fl(n) \times \Fl(n)$.
\end{ex}

Let $B_n^+$ and $B_n^-$ be the groups of invertible upper and lower triangular matrices inside $\GL(n)$. As in Example~\ref{ex:GX-schubert-fiber}, fix $E^0_\bullet \in \Fl(n)$ defined by $E^0_i = \vspan \{e_1, \ldots, e_i\}$ for $i \in [n]$. The action of $(g_1, g_2) \in K \eqdef \GL(n) \times \GL(n)$ on $U \in \Gr(n,2n)$ is given by $(g_1, g_2)\cdot U = (g_2 \oplus (g_1^{-1})^*) U$.

\begin{prop} \label{prop:orbit-description-Gr} The sets $\openincGX_w$ for $w \in S_n$ are the $K$-orbits on $\grGr \times \Fl(n) \times \Fl(n)$. \end{prop}

\begin{proof}
    The group $K$ acts transitively on $\Fl(n) \times \Fl(n)$ with the stabilizer of $(E_\bullet, E_\bullet)$ being $B_n^+ \times B_n^+$. So, the $K$-orbits on $\grGr \times \Fl(n) \times \Fl(n)$ are the sets
    \begin{equation*}
        \bigsqcup_{(gB_n^+, hB_n^+)} (g, h)\cdot O \times \{gE^0_\bullet\} \times \{hE^0_\bullet\}
    \end{equation*}
    where $(gB_n^+, hB_n^+)$ ranges over $\GL(n)/B_n^+ \times \GL(n)/B_n^+$ and $O$ ranges over the $B_n^+ \times B_n^+$-orbits on $\grGr$. We claim that the latter orbits are the sets $\openincGX_w(E^0_\bullet, E^0_\bullet)$, which will imply the proposition given that it is straightforward to check that
    \begin{equation*}
        \openincGX_w = \bigsqcup_{(gB_n^+, hB_n^+)} (g, h)\cdot \openincGX_w(E^0_\bullet, E^0_\bullet) \times \{gE^0_\bullet\} \times \{hE^0_\bullet\}.
    \end{equation*}

    Let $(g_1, g_2) \in K$ act on $M \in \GL(n)$ by $(g_1, g_2) \cdot M = (g_1^t)^{-1} M g_2^{-1}$. If $\graph(f) \in \grGr$, then $(g_1, g_2) \cdot \graph(f) = \graph((g_1^{-1})^* \circ f \circ g_2^{-1})$.  Thus, the isomorphism $\grGr \to \GL(n)$ taking $\graph(f)$ to the matrix of $f$ with respect to the ordered bases $e_1, \ldots, e_n$ and $e_1^*, \ldots, e_n^*$ is $\GL(n)$-equivariant, so sends $B_n^+ \times B_n^+$-orbits to $B_n^+ \times B_n^+$-orbits. This isomorphism also sends $\openincGX_w(E^0_\bullet, E^0_\bullet)$ to the set $O_w = \{M \in \GL(n) : \text{$\rank M_{[i][j]} = \rank w_{[i][j]}$ for $i,j \in [n]$}\}$. The sets $O_w$ are the $B_n^+ \times B_n^+$-orbits on $\GL(n)$ by \cite[Lemma 3.1]{fulton-double-schubert}, proving the claim.
\end{proof}

It will sometimes be convenient to know that $\openincGX_{w}$ or a similar variety is a fiber bundle over some subproduct of $\Gr(n,2n) \times \Fl(n) \times \Fl(n)$. In each case this will follow from the next lemma. If a group $G$ acts on a set $X$, let $G_x$ be the stabilizer of $x \in X$ and $Gx$ the orbit of $x$.

\begin{lem} \label{lem:fiber-bundle} Let $G$ be a Lie group acting on spaces $X$ and $Y$, and hence diagonally on $X \times Y$. Suppose the action on $X$ is transitive. If $S \subseteq X \times Y$ is $G$-stable, the projection $p : S \to X$ onto the first coordinate is a fiber bundle. \end{lem}

\begin{proof}
    Fix $x \in X$. Since $G_x \subseteq G$ is a closed subgroup, the projection $q : G \twoheadrightarrow G/G_x \simeq X$ is a fiber bundle \cite[Theorem 4.3]{broecker-tom-dieck}. Let $\phi : q^{-1}(U) \to U \times G_x$ be a trivialization of $q$ over some neighborhood $U$ of $x$. Define $\psi : p^{-1}(U) \to U \times p^{-1}(x)$ by the formula $(gx,y) \mapsto (gx, x, \phi^{-1}(gx,1)^{-1}y)$, using the transitivity of the $G$-action on $X$. It is not quite obvious that $\psi(gx,y)$ actually lies in $U \times p^{-1}(x) \subseteq U \times S$ when $(gx,y) \in S$: this holds because $\phi^{-1}(gx,1)x = q(\phi^{-1}(gx,1)) = gx$, so that $(x, \phi^{-1}(gx,1)^{-1}y) = \phi^{-1}(gx,1)^{-1}\cdot (gx, y)$ is still in $S$ by the $G$-stability of $S$. The map  $\psi$ has $(gx,x,y) \mapsto (gx, \phi^{-1}(gx,1)y)$ as an inverse, and trivializes $p$ over $U$. 
\end{proof}

\begin{cor} \label{cor:GX-fiber-bundle}  The projection from $\openincGX_w$ to any proper subproduct of $\grGr \times \Fl(n) \times \Fl(n)$ is a fiber bundle. \end{cor}

\begin{lem} \label{lem:GX-description} \hfill
\begin{enumerate}[(a)]
\item $\openincGX_w$ is irreducible of codimension $\ell(w)$.
\item $\incGX_w$ is the closure in $\Gr(n,2n) \times \Fl(n) \times \Fl(n)$ of
\begin{equation*}
C_w \eqdef \{(\graph(f), E_\bullet, F_\bullet) : \text{$\rank(F_j \xrightarrow{f} E_i^*) = \rank w_{[i][j]}$ for $(i,j) \in \Ess(D(w))$}\}.
\end{equation*}
\item $\incGX_w(\graph(f)) = \overline{\openincGX_w(\graph(f))}$ for any $\graph(f) \in \grGr$.
\item The intersection $\incGX_w(\graph(f)) = \incGX_w \cap (\{\graph(f)\} \times \Fl(n) \times \Fl(n))$ is transverse. 
\end{enumerate}
\end{lem}
    \begin{proof} \hfill
        \begin{enumerate}[(a)]
        \item Given that $\openincGX_w$ is a $\GL(n)\times \GL(n)$-orbit by Proposition~\ref{prop:orbit-description-Gr}, its irreducibility follows from the irreducibility of $\GL(n) \times \GL(n)$. Since $\openincGX_{w}$ is a fiber bundle over $\grGr \times \Fl(n)$ by Lemma~\ref{lem:fiber-bundle} with fibers isomorphic to the codimension $\ell(w)$ Schubert cell $X_w$ (as per Example~\ref{ex:GX-schubert-fiber}), it has dimension $\dim(\grGr \times \Fl(n)) + \dim(\Fl(n)) - \ell(w)$, hence codimension $\ell(w)$.

        \item The projection $C_w \to \grGr \times \Fl(n)$ is a fiber bundle by Lemma~\ref{lem:fiber-bundle}. Lemma~\ref{lem:degeneracy-locus-essential-set} shows that $\overline{C_w(\graph(f), F_\bullet)}$ is a Schubert cell, irreducible of codimension $\ell(w)$, so the same is true of the fibers $C_w(\graph(f), F_\bullet)$. As in (a), this implies that $\codim C_w = \ell(w)$, and the irreducibility of $\grGr \times \Fl(n)$ and of $C_w(\graph(f), F_\bullet)$ implies that $C_w$ is irreducible. The inclusion $\incGX_w \subseteq \overline{C}_w$ is clear, so (a) forces equality.
        
        \item By Lemma~\ref{lem:fiber-bundle}, both $\openincGX_w \to \grGr$ and $\incGX_w \cap (\grGr \times \Fl(n) \times \Fl(n)) \to \grGr$ are fiber bundles, and the proof of that lemma shows that the same maps provide local trivializations for both bundles simultaneously. This reduces (c) to the easy claim that if $X$ and $U \subseteq Y$ are spaces in which points are closed, and $X \times U$ is dense in $X \times Y$, then $\{x\} \times Y = \overline{\{x\} \times U}$ for any $x \in X$.
        
        \item As in (c), this reduces by Lemma~\ref{lem:fiber-bundle} to the claim that if $X$ and $Y$ are smooth manifolds and $Y' \subseteq Y$ is a submanifold, then $(X \times Y') \cap (\{x\} \times Y) = \{x\} \times Y'$ is transverse for $x \in X$. \qedhere
        \end{enumerate}
\end{proof}

\subsection{Cohomological formulas}
\label{subsec:back-stable-reps}

In this subsection we show that $[\incGX_{w}]$ is represented by the back-stable double Schubert polynomial $\bS_w(x;-y)$.

\begin{lem} \label{lem:double-schubert-rep} The class $[\incGX_w(\graph(f))]$  is represented by $\fkS_{w}(x; -y)$. \end{lem}

\begin{proof}
    Fix an invertible linear map $f : \CC^n \to {\CC^n}^*$. Then $\openincGX_w(\graph(f))$ is the set of $(E_\bullet, F_\bullet) \in \Fl(n) \times \Fl(n)$ such that $\rank(F_j \xrightarrow{f} E_i^*) =  \rank(E_i \xrightarrow{f^*} F_j^*)  = \rank w_{[i][j]}$ for $i,j \in [n]$. Since $\codim \openincGX_w = \ell(w)$ by Lemma~\ref{lem:GX-description}(a), Fulton's degeneracy locus formula (Theorem~\ref{thm:fulton-double-schubert}) applied to the sequence $\vb E_1 \hookrightarrow \cdots \hookrightarrow \vb E_n \xrightarrow{f^*} \vb F_n^* \twoheadrightarrow \cdots \twoheadrightarrow \vb F_1^*$ gives $[\overline{\openincGX_w(\graph(f))}] = \fkS_w(x'; y')$, where $x_i' = c_1((\vb E_i / \vb E_{i-1})^*)$ and $y_i' = c_1(\vb F_i / \vb F_{i-1})$. Under the conventions from \S \ref{subsec:cohom}, $x_i' = x_i$ and $y_i' = -y_i$.  Lemma~\ref{lem:GX-description}(c) shows that $[\overline{\openincGX_w(\graph(f))}] = [\incGX_w(\graph(f))]$.
\end{proof}

For vector spaces $V \subseteq W$, let $\ann V = \{\alpha  \in W^* : \alpha|_V = 0\}$ be the annihilator of $V$.
\begin{lem} \label{lem:grassmannian-locus} If $w \in S_n$ is vexillary, then $\incGX_w$ is the closure of the subset
    \begin{equation*}
        B_w \eqdef \{(U,E_\bullet,F_\bullet) : \text{$\dim(U \cap (F_j \oplus \ann{E_i})) = j - \rank w_{[i][j]}$ for $(i,j) \in \Ess(D(w))$} \}
    \end{equation*}
    of $\Gr(n,2n) \times \Fl(n) \times \Fl(n)$.
\end{lem}
\begin{proof}
    The image of the composition $F_j \hookrightarrow \CC^n \xrightarrow{f} {\CC^n}^* \twoheadrightarrow E_i^*$ is the same as the image of $\graph(f) \cap (F_j \oplus {\CC^n}^*) \hookrightarrow \CC^n \oplus {\CC^n}^* \twoheadrightarrow E_i^*$, where the last map sends $\CC^n$ to $0$ and restricts $\alpha \in {\CC^n}^*$ to $E_i$.
    Thus, $(\graph(f), E_\bullet, F_\bullet) \in \openincGX_w$ if and only if
    \begin{equation} \label{eq:rank-nullity-equivalence}
        \rank (\graph(f) \cap (F_j \oplus {\CC^n}^*) \hookrightarrow {\CC^n} \oplus {\CC^n}^* \twoheadrightarrow E_i^*) = \rank w_{[i][j]}
    \end{equation}
    for $i,j \in [n]$. Since $\dim(\graph(f) \cap (F_j \oplus {\CC^n}^*)) = j$ because $f$ is invertible, \eqref{eq:rank-nullity-equivalence} is equivalent by the rank-nullity theorem to $\dim(\graph(f) \cap (F_j \oplus \ann{E_i})) = j-\rank w_{[i][j]}$. If we enforce these rank conditions only for $(i,j) \in \Ess(D(w))$, we get the set $B_w \cap (\grGr \times \Fl(n) \times \Fl(n))$, and Lemma~\ref{lem:GX-description} says that the closure of this set is $\incGX_w$.

    However, to show that $\overline{B_w \cap (\grGr \times \Fl(n) \times \Fl(n))} = \overline{B}_w$, we need to know that $B_w$ has no components in the complement of $\grGr \times \Fl(n) \times \Fl(n)$. Since $\grGr$ is dense in $\Gr(n,2n)$, it suffices to show that $B_w$ is irreducible. Given that $w$ is vexillary, we can write $\Ess(D(w)) = \{(i_1,j_1) \SWNEneq \cdots \SWNEneq (i_s,j_s)\}$ by Lemma~\ref{lem:vex-chain}.  We then have
    \begin{equation} \label{eq:GX-flag}
        F_{j_1} \oplus \ann{E_{i_1}}  \subseteq \cdots \subseteq F_{j_s} \oplus \ann{E_{i_s}}.
    \end{equation}
    Lemma~\ref{lem:fiber-bundle} implies that $B_w$ is a fiber bundle over $\Fl(n) \times \Fl(n)$. For fixed $(E_\bullet, F_\bullet) \in \Fl(n) \times \Fl(n)$, the closure of the fiber
    \begin{equation*}
        B_w(E_\bullet, F_\bullet) = \{U  : \text{$\dim(U \cap (F_{j_p} \oplus \ann E_{i_p})) = j_p - \rank w_{[i_p][j_p]}$ for $p \in [s]$}\}
    \end{equation*}
    is a Schubert variety in $\Gr(n,2n)$ with respect to the partial flag \eqref{eq:GX-flag}. Since Schubert varieties are irreducible, so too is each fiber $B_w(E_\bullet, F_\bullet)$ and therefore $B_w$ as well.
\end{proof}

Suppose $w$ is vexillary and write $\Ess(D(w)) = \{(i_1, j_1) \SWNEneq \cdots \SWNEneq (i_s, j_s)\}$. Lemma~\ref{lem:grassmannian-locus} shows that $\incGX_w$ is an example of a Grassmannian degeneracy locus as described in \S \ref{subsec:degeneracy-loci}. Specifically, $\incGX_w$ is the closure of the locus of points in $X = \Gr(n,2n) \times \Fl(n) \times \Fl(n)$ over which $\dim(\vb G \cap \vb H_{\mu_p}) = k_p$ for $p \in [s]$, where:
\begin{itemize}
    \item[$\scriptstyle \blacktriangleright$] $\vb V$ is the trivial bundle ${\CC^n} \oplus {\CC^n}^*$ over the first factor of $X$;
    \item[$\scriptstyle \blacktriangleright$] $\vb G$ is the tautological bundle over the first factor of $X$;
    \item[$\scriptstyle \blacktriangleright$] $\vb{H}_{\mu_p} = \vb{F}_{j_p} \oplus \ann{\vb{E}}_{i_p}$, where $\vb E_{\bullet}$ and $\vb F_{\bullet}$ are the tautological flags over the second and third factors of $X$, so $\mu_p = 2n-\rank(\vb F_{j_p} \oplus \ann{\vb{E}}_{i_p}) = n+i_p-j_p$ for $p \in [s]$;
    \item[$\scriptstyle \blacktriangleright$] $k_p = j_p - \rank w_{[i_p][j_p]}$ for $p \in [s]$, and $k_0 = 0$;
    \item[$\scriptstyle \blacktriangleright$] $\lambda_{k} = i_p - \rank w_{[i_p][j_p]}$ for $k \in [k_s]$, where $p$ is such that $k_{p-1} < k \leq k_p$.
\end{itemize}
Using this data and the setup of Theorem~\ref{thm:kempf-laksov}, we have
\begin{equation} \label{eq:GX-chern-class}
    c(k) = \frac{c(\vb V)}{c(\vb G)c(\vb{F}_{j_p} \oplus \ann{\vb{E}}_{i_p})} = \frac{1}{c(\vb G)} \frac{(1+x_1) \cdots (1+x_{i_p})}{(1-y_1)\cdots (1-y_{j_p})},
\end{equation}
where $k_{p-1} < k \leq k_p$. Here, we have used the isomorphism $\ann E \simeq (\CC^n/E)^*$ to compute $c(\ann{\vb{E}}_{i_p}) = c((\CC^n/\vb E_{i_p})^*) = 1/c(\vb E_{i_p}^*) = \frac{1}{(1+x_1)\cdots (1+x_{i_p})}$. Under the identifications of \S\ref{subsec:cohom}, $1/c(\vb G)$ is represented by $\sum_d h_d(x_- \superslash y_-)$, so we will think of $c(k)_d$ as the power series $h_d(x_{-\infty \cdots i_p} \superslash y_{-\infty \cdots j_p})$. Set $c'(k) = c(k)\ev_{c(\vb G)\to 1}$, so $c'(k_p)_d = h_d(x_{1\cdots i_p} \superslash y_{1\cdots j_p})$.

\begin{lem} \label{lem:vex-double-schubert-reps} 
Fix a vexillary $w$ and $\graph(f) \in \grGr$. Then $[\incGX_w]$ is represented by $\det(c(k)_{\lambda_k+t-k})_{k,t \in [\ell(\lambda)]}$, and $[\incGX_w(\graph(f))]$ is represented by $\det(c'(k)_{\lambda_k+t-k})_{k,t \in [\ell(\lambda)]}$, where $\lambda = \sh(w)$. \end{lem}

\begin{proof}
    By Lemma~\ref{lem:vex-shape}, the partition $\lambda$ associated to $w$ above is the same as $\sh(w)$. This implies that $\incGX_w$ has the expected codimension $|\lambda|$, given that $|\lambda| = |\sh(w)| = |D(w)| = \ell(w)$ and $\codim \incGX_w = \ell(w)$ by Lemma~\ref{lem:GX-description}(a). As $\incGX_w$ is a Grassmannian degeneracy locus $\Omega_{\lambda}$ with respect to the data described above, $[\incGX_w] = \det(c(k)_{\lambda_k+t-k})_{k,t \in [\ell(\lambda)]}$ by Theorem~\ref{thm:kempf-laksov}.
     
     Let  $i : \Fl(n) \times \Fl(n) \hookrightarrow \Gr(n,2n) \times \Fl(n) \times \Fl(n)$ be the inclusion $(E_\bullet,F_\bullet) \mapsto (\graph(f), E_\bullet, F_\bullet)$. The pullback $i^*$ on cohomology sends $1 \otimes b \otimes c$ to $b \otimes c$ and $a \otimes b \otimes c$ to $0$ when $a$ has positive degree, so $i^* c(k) = c'(k)$. Now $[\incGX_w(\graph(f))] = [i^{-1}(\incGX_w)] = i^*[\incGX_w] = \det(c'(k)_{\lambda_k+t-k})_{k,t \in [\ell(\lambda)]}$, where the second equality holds by the transversality from Lemma~\ref{lem:GX-description}(d).
\end{proof}

\begin{thm} \label{thm:vex-double-schubert-polys}  Let $w$ be vexillary, with $c(k)$, $c'(k)$, and $\lambda$ defined as above. As polynomials in $x$ and $y$, $\fkS_{w}(x; -y)$ equals $\det(c'(k)_{\lambda_k+t-k})_{k,t \in [\ell(\lambda)]}$  and $\bS_{w}(x; -y)$ equals $\det(c(k)_{\lambda_k+t-k})_{k,t \in [\ell(\lambda)]}$. \end{thm}

\begin{proof} The equality $\fkS_{w}(x; -y) = \det(c'(k)_{\lambda_k+t-k})_{k,t \in [\ell(\lambda)]}$ is \cite[Proposition 9.6(f)]{fulton-double-schubert}; see also \cite{wachs-schubert}. The permutation $1_m \times w$ is vexillary for any $m \in \NN$, with
    \begin{equation*}
        \Ess(D(1_m \times w)) = \{(i+m,j+m) : (i,j) \in \Ess(D(w))\},
    \end{equation*}
    while the partition $\lambda$ and sequence $k_{\bullet}$ are independent of $m$. Lemma~\ref{lem:vex-double-schubert-reps} therefore shows that $[\incGX_{1_m \times w}(\graph(f))]$ is represented by $\fkS_{1_m \times w}(x; -y) = \det(c'(k)_{\lambda_k+t-k})_{k,t \in [\ell(\lambda)]}$, where if $k_{p-1} < k \leq k_p$ then $c'(k) = \frac{(1+x_1) \cdots (1+x_{m+i_p})}{(1-y_1)\cdots (1-y_{m+j_p})}$. Now let $c''(k)$ be $c'(k)$ with the alphabets $x_{1 \cdots \infty}, y_{1 \cdots \infty}$ replaced by $x_{-m\cdots \infty}, y_{-m\cdots \infty}$, so that
    \begin{equation*}% \label{eq:back-stable-w}
        \bS_{w}(x;-y) = \lim_{m \to \infty} \fkS_{1_m \times w}(x_{-m\cdots n}; -y_{-m \cdots n}) = \det\left(\lim_{m \to \infty} c''(k)_{\lambda_k+t-k}\right)_{k,t \in [\ell(\lambda)]}.
    \end{equation*}
    This is the determinantal formula for $\bS_w(x;-y)$ claimed in the theorem, since
    \begin{align*}
        \lim_{m \to \infty} c''(k) &= \lim_{m \to \infty} \frac{(1+x_{-m}) \cdots (1+x_{-1})}{(1-y_{-m})\cdots (1-y_{-1})} \frac{(1+x_{1}) \cdots (1+x_{i_p})}{(1-y_{1})\cdots (1-y_{j_p})}\\
        &= \left( \sum_{d=0}^{\infty} h_d(x_- \superslash y_-) \right) \frac{(1+x_{1}) \cdots (1+x_{i_p})}{(1-y_{1})\cdots (1-y_{j_p})}\\
        &= c(k). \qedhere
    \end{align*}
\end{proof}

    Lemma~\ref{lem:vex-double-schubert-reps} and Theorem~\ref{thm:vex-double-schubert-polys} imply that $\bS_w(x;-y)$ represents $[\incGX_w]$ for vexillary $w$, but we want to show this for all $w$. Since $w = w_0 = n(n-1)\cdots 21$ is vexillary, the general result would follow if we knew that the classes $[\incGX_w]$ satisfied the same divided difference recurrence as $\bS_w(x;-y)$ (Proposition~\ref{prop:schubert-recurrence}). To that end, we recall a geometric interpretation of the divided difference operator $\partial_i$; see \cite[\S 10]{youngtableaux}.
    
    Let $\Fl_i(n)$ be the variety of partial flags
    \begin{equation*}
        E_1 \subseteq \cdots \subseteq E_{i-1} \subseteq E_{i+1} \subseteq \cdots \subseteq E_n = {\CC^n}
    \end{equation*}
    where $\dim E_p = p$. Set $X = \Fl(n)$, and let $\tilde{X}$ be the fiber product $\Fl(n) \times_{\Fl_i(n)} \Fl(n)$ with respect to the projection $\Fl(n) \to \Fl_i(n)$, so
    \begin{equation*}
        \tilde{X} = \{(E_\bullet, E_\bullet') \in \Fl(n) \times \Fl(n) : \text{$E_p = E_p'$ for $p \neq i$}\}.
    \end{equation*}
    Let $p_1, p_2 : \tilde{X} \to X$ be the projections onto the two coordinates of $\tilde{X}$.  Under the Borel isomorphism \eqref{eq:borel-isomorphism}, ${p_2}_* p_1^* : H^*(\Fl(n)) \to H^*(\Fl(n))$ corresponds to $\partial_i$. We therefore also write $\partial_i$ for the operator ${p_2}_* p_1^*$. Suppose $Y \subseteq \Fl(n)$ is a closed subvariety. Then
    \begin{equation} \label{eq:divided-difference-cases}
        \partial_i [Y] = \begin{cases}
            0 & \text{if $\dim p_2(p_1^{-1}(Y)) < 1+\dim Y$}\\
            d[p_2(p_1^{-1}(Y))] & \parbox{6.5cm}{if $p_1^{-1}(Y) \xrightarrow{p_2} p_2(p_1^{-1}(Y))$ is $d$-to-$1$ over a dense open subset of $p_2(p_1^{-1}(Y))$}
        \end{cases}
    \end{equation}

Now set $X = \Gr(n,2n) \times \Fl(n) \times \Fl(n)$ instead, let $\tilde{X}$ be the fiber product $X \times_{\Gr(n,2n) \times \Fl_i(n) \times \Fl(n)} X$, and let $P_1, P_2 : \tilde{X} \to X$ be the two canonical projections. Then ${P_2}_* P_1^*$ acts on $H^*(\Gr(n,2n) \times \Fl(n) \times \Fl(n))$ as $\mathrm{id} \otimes \partial_i \otimes \mathrm{id}$.

\begin{lem} \label{lem:divided-diff-classes} For $w \in S_n$,
    \begin{equation*}
        \partial_i [\incGX_w] = \begin{cases}
            [\incGX_{ws_i}] & \text{if $\ell(ws_i) < \ell(w)$}\\
            0 & \text{otherwise}
        \end{cases}
    \end{equation*}
\end{lem}        

\begin{proof}
Lemma~\ref{lem:fiber-bundle} implies that $P_1^{-1}(\openincGX_w)$ and $P_2(P_1^{-1}(\openincGX_w))$ are fiber bundles over $\grGr \times \Fl(n)$, with fibers $p_1^{-1}(X_w)$ and $p_2(p_1^{-1}(X_w))$. The map $P_2 : P_1^{-1}(\openincGX_w) \to P_2(P_1^{-1}(\openincGX_w))$ respects fibers, i.e. commutes with the two bundle projections to $\grGr \times \Fl(n)$. Moreover, after choosing an appropriate trivialization of these two bundles over an open set $U \subseteq \grGr \times \Fl(n)$ (as described in Lemma~\ref{lem:fiber-bundle}, for instance), this map $P_2$ gets identified with $\id \times p_2 : U \times p_1^{-1}(X_w) \to U \times p_2(p_1^{-1}(X_w))$. 

If $\ell(ws_i) > \ell(w)$, then $\dim p_2(p_1^{-1}(X_w)) \leq \dim X_w$, while if $\ell(ws_i) < \ell(w)$, then $p_2$ is one-to-one from a dense subset of $p_1^{-1}(X_w)$ onto a dense subset of $X_{ws_i}$ \cite[\S 10.3, Lemma 8]{youngtableaux}. This plus the previous paragraph implies that if $\ell(ws_i) > \ell(w)$, then $\dim P_2(P_1^{-1}(\openincGX_w)) \leq \dim \openincGX_w$, while if $\ell(ws_i) < \ell(w)$, then $P_2$ is one-to-one from a dense subset of $P_1^{-1}(\openincGX_w)$ onto a dense subset of $\openincGX_{ws_i}$. Now apply \eqref{eq:divided-difference-cases}.
\end{proof}

\begin{thm} \label{thm:back-stable-rep}
    For $w \in S_n$, the back-stable double Schubert polynomial $\bS_w(x;-y)$ represents the class $[\incGX_w]$.
\end{thm}

\begin{proof}
    The theorem holds for $w = w_0$ by Lemma~\ref{lem:vex-double-schubert-reps} and Theorem~\ref{thm:vex-double-schubert-polys}, as $w_0$ is vexillary. Therefore it holds for all $w$, since $\bS_{w}(x;-y)$ and $[\incGX_{w}]$ satisfy the same recurrence with base case $w = w_0$, by Proposition~\ref{prop:schubert-recurrence} and Lemma~\ref{lem:divided-diff-classes} respectively.
\end{proof}

\section{Preliminaries on involutions}
\label{sec:inv}

\subsection{Involution Schubert polynomials} \label{subsec:inv-schubert}
It will be convenient to write $\I^{\O}_n = \I_n$ and $\I^{\Sp}_n = \Ifpf_n$. Writing $K$ for either $\O$ or $\Sp$, the Bruhat-minimal element of $\IK_n$ is denoted $\idK_n$, so $\idO_n = 1_n$ and $\idfpf_{2r} = (1,2)(3,4)\cdots(2r-1,2r)$.

\begin{defn}[\cite{humphreysreflgroups}, Theorem 7.1]
    The \emph{Demazure product} $\circ$ is the unique associative product on $S_n$ satisfying
    \begin{equation*}
        w \circ s_i = \begin{cases}
            ws_i & \text{if $\ell(w) < \ell(ws_i)$}\\
            w & \text{if $\ell(w) > \ell(ws_i)$}
        \end{cases}
    \end{equation*}
    for any $w \in S_n$ and simple transposition $s_i = (i,i+1)$.
\end{defn}

It is not hard to check that $w^{-1} \circ z \circ  w$ is an involution whenever $z$ is, and that for any $z \in \IK_n$, there exists $w \in S_n$ such that $z = w^{-1} \circ \idK_n \circ w$.
\begin{defn}
    The set of \emph{atoms} of $z \in \IK_n$ is
    \begin{equation*}
        \AK_n \eqdef \{ w \in S_n : \text{$w^{-1} \circ \idK_n \circ w = z$, $\ell(w)$ minimal} \}.
    \end{equation*}
    A \emph{reduced involution word} of $z \in \IK_n$ is a word $a \in \Red(w)$ for some $w \in \AK(z)$. Equivalently, $a$ is a minimal-length word $a_1 \cdots a_\ell$ such that
    \begin{equation*}
        z = s_{a_\ell} \circ \cdots \circ s_{a_1} \circ \idK_n \circ s_{a_1} \circ \cdots \circ s_{a_\ell}.
    \end{equation*}
\end{defn}
Let $\iRK(z)$ be the set of reduced involution words of $z \in \IK_n$, so $\iRK(z) = \bigsqcup_{w \in \AK(z)} \Red(w)$. Let $\iellK(z)$ be the length of any word in $\iRK(z)$.

\begin{ex}
    \begin{align*}
        &(1,3) = s_1 \circ s_2 \circ 1 \circ s_2 \circ s_1 = s_1 \circ s_2 \circ s_1 = s_1 s_2 s_1\\
        &(1,3) = s_2 \circ s_1 \circ 1 \circ s_1 \circ s_2 = s_2 \circ s_1 \circ s_2 = s_2 s_1 s_2,
    \end{align*}
    and $\iR((1,3)) = \{{\bf 12}, {\bf 21}\}$ and $\A((1,3)) = \{231, 312\}$. As a fixed-point-free example,
    \begin{align*}
        &(1,4)(2,3) = s_3 \circ s_2 \circ (1,2)(3,4) \circ s_2 \circ s_3 = s_3 \circ (1,3)(2,4) \circ s_3\\
        &(1,4)(2,3) = s_1 \circ s_2 \circ (1,2)(3,4) \circ s_2 \circ s_1 = s_1 \circ (1,3)(2,4) \circ s_1,
    \end{align*}
    and $\iRfpf((1,4)(2,3)) = \{{\bf 23}, {\bf 21}\}$ and $\Afpf((1,4)(2,3)) = \{3124, 1342\}$. By contrast,
    \begin{equation*}
        \iR((1,4)(2,3)) = \{{\bf 1323}, {\bf 1232}, {\bf 3123}, {\bf 2132}, {\bf 2312}, {\bf 3212}, {\bf 3121}, {\bf 1321}\}
    \end{equation*}
    and $\A((1,4)(2,3)) = \{4213, 3412, 2431\}$.
\end{ex}

The next definition is identical to Definition~\ref{defn:schubert} for Schubert polynomials save that reduced words have been replaced by reduced involution words. Let $K$ be $\O$ or $\Sp$.
\begin{defn} \label{defn:inv-schubert}
    The \emph{back-stable involution Schubert polynomial} of $z \in \IK_n$ is
    \begin{equation*}
        \ibSK_z \eqdef \sum_{a \in \iRK(z)} \sum_{\substack{i \in \comp(a)} } x_{i_1} \cdots x_{i_{\ell}}.
    \end{equation*}
    The \emph{involution Schubert polynomial} of $z$ is
    \begin{equation*}
        \iSK_z \eqdef \ibSK_z\ev_{x_{-} \to 0} = \sum_{a \in \iRK(z)} \sum_{\substack{i \in \comp(a) \\ 1 \leq i} } x_{i_1} \cdots x_{i_{\ell}}.
    \end{equation*}
    The \emph{involution Stanley symmetric function} of $z$ is
    \begin{equation*}
        \iFK_z \eqdef \ibSK_z\ev_{x_{+} \to 0} = \sum_{a \in \iRK(z)} \sum_{\substack{i \in \comp(a) \\ i \leq -1} } x_{i_1} \cdots x_{i_{\ell}}.
    \end{equation*}
\end{defn}
Involution Schubert polynomials were introduced by Wyser and Yong \cite{wyser-yong-orthogonal-symplectic} and further studied by Hamaker, Marberg, and the author in \cite{HMP4,HMP1, HMP3,HMP5}, where involution Stanley symmetric functions were also investigated. They are homogeneous of degree $\iellK(z)$. Like Schubert polynomials, they satisfy a divided difference recurrence, as do the back-stable versions.  Recall that $\twocyc(z)$ is the number of $2$-cycles in $z$.

\begin{prop} \label{prop:inv-schubert-recurrence} 
    For $z \in \IK_n$ and $i \in [n-1]$,
    \begin{equation*}
        \partial_i \ibSK_z = \begin{cases}
            0 & \text{if $\ell(zs_i) > \ell(z)$ or $s_i \circ z \circ s_i \notin \IK_n$}\\
            2^{\twocyc(z) - \twocyc(s_i \circ z \circ s_i)} \ibSK_{s_i \circ z \circ s_i} & \text{if $\ell(zs_i) < \ell(z)$ and $s_i \circ z \circ s_i \in \IK_n$}
        \end{cases}
    \end{equation*}
\end{prop}

\begin{proof}
The fact that $\iRK(1_m \times z) = \{(a_1+m)\cdots (a_\ell+m) : a_1 \cdots a_\ell \in \iRK(z)\}$ for any $m$ (even $m$, in the case $K = \Sp$) implies $\ibSK_z = \lim_{m \to \infty} \iSK_{1_m \times z}(x_{-m \cdots n})$. Given that
\begin{align*}
\partial_i \ibSK_z &= \partial_i \lim_{m \to \infty} \iSK_{1_m \times z}(x_{-m \cdots n}) = \lim_{m \to \infty} \partial_i (\iSK_{1_m \times z}(x_{-m \cdots n}))\\
&= \lim_{m \to \infty} (\partial_{m+i} \iSK_{1_m \times z})(x_{-m \cdots n}),
\end{align*}
the lemma follows from Wyser and Yong's result that the involution Schubert polynomials $\iSK_z$ satisfy the stated divided difference recurrence \cite{wyser-yong-orthogonal-symplectic}.
\end{proof}

For $z \in \IK_n$, define $\openiXK_z = \{E_{\bullet} \in \Fl(n) : \text{$\rank(E_j \xrightarrow{f} E_i^*) =  \rank z_{[i][j]}$ for $i,j \in [n]$}\}$ for some fixed invertible $f : \CC^n \to {\CC^n}^*$, symmetric or skew-symmetric according to whether $K$ is $\O(n)$ or $\Sp(n)$. The sets $\openiXK_z$ for $z \in \IK_n$ are the  $K$-orbits on $\Fl(n)$  \cite{wyser-degeneracy-loci}.

\begin{thm}[\cite{wyser-yong-orthogonal-symplectic}] \label{thm:inv-schubert-classes}
$[\iX_y]$ is represented by $2^{\twocyc(y)} \iS_y$ for $y \in \I_n$, and $[\iXfpf_z]$ is represented by $\iSfpf_z$ for $z \in \Ifpf_n$.  \end{thm}

\subsection{Cohomology of $\LG(2n)$ and $\OG(2n)$}
\label{subsec:LG-OG-cohom}

For $d \geq 0$, let $Q_d$ be the symmetric function $\sum_{a+b=d} h_a e_b$, so
\begin{equation*}
    \frac{(1+x_1)(1+x_2)\cdots}{(1-x_1)(1-x_2)\cdots} = \sum_d Q_{d}(x_+).
\end{equation*}
Also define $P_0 = 1$ and $P_d = \frac{1}{2}Q_d$ for $d > 0$; one checks that $P_d$ has integer coefficients. Let $\Gamma_Q$ be the subring of $\Lambda$ generated by the $Q_d$, and $\Gamma_P$ the subring of $\Lambda$ generated by the $P_d$, so $\Gamma_Q \subseteq \Gamma_P$.

Recall that $\LG(2n)$ is the closed subvariety of $\Gr(n,2n)$ of subspaces isotropic with respect to the skew-symmetric form $((v_1, \omega_1), (v_2, \omega_2))^- = \omega_1(v_2) - \omega_2(v_1)$. The subvariety of $\Gr(n,2n)$ consisting of subspaces isotropic with respect to the symmetric form $((v_1, \omega_1), (v_2, \omega_2))^+ = \omega_1(v_2) + \omega_2(v_1)$ has two components, and $\OG(2n)$ is defined to be the component containing $\CC^n$.

If $E \subseteq \CC^n$ is a subspace, the subspace $E \oplus \ann E \subseteq {\CC^n} \oplus {\CC^n}^*$ is isotropic under $(-,-)^-$ and $(-,-)^+$. As in \S \ref{subsec:cohom}, the maps $\Fl(n) \to \LG(2n)$ and $\Fl(n) \to \OG(2n)$ defined by $E_{\bullet} \mapsto E_i \oplus \ann{E_i}$ induce maps on cohomology with
\begin{align*}
    c(\vb G^*) \mapsto c({\vb E_i}^* \oplus {\ann{\vb E_i}}^*) &= c({\vb E_i}^*)c(\CC^n / \vb E_i) = \frac{c(\vb E_i^*)}{c(\vb E_i)} = \frac{(1+x_1)\cdots (1+x_i)}{(1-x_1)\cdots(1-x_i)} = \sum_d Q_{d}(x_{1\cdots i}),
\end{align*}
where $\vb G$ is the tautological bundle on $\LG(2n)$ or $\OG(2n)$. This suggests identifying $c_d(\vb G^*)$ with $Q_d \in \Gamma_Q$, and indeed Pragacz \cite{pragacz-LG-OG} showed that sending $Q_d \mapsto c_d(\vb G^*)$ induces a ring isomorphism $\Gamma_Q / (\text{$Q_d$ for $d > n$}) \to H^*(\LG(2n))$.

The analogous map for $H^*(\OG(2n))$ is well-defined and injective but not surjective. However, it extends to an isomorphism $\Gamma_P / (\text{$P_d$ for $d > n$}) \to H^*(\OG(2n))$, again sending $2P_d = Q_d \mapsto c_d(\vb G^*)$ for $d > 0$ \cite{pragacz-LG-OG}.

\subsection{Schur $P$- and $Q$-functions}

Let $\lambda$ be a finite sequence of nonnegative integers of length $\ell = \ell(\lambda)$. Define
\begin{equation*}
    \lambda^+ = \begin{cases}
        (\lambda_1, \ldots, \lambda_{\ell}, 0) & \text{$\ell$ odd}\\
        (\lambda_1, \ldots, \lambda_{\ell}) & \text{$\ell$ even}
    \end{cases},
\end{equation*}
so $\lambda^+$ is always a sequence of even length.
\begin{defn}[\cite{Macdonald}, III.8] \label{defn:PQ}
    The \emph{Schur $Q$-function} of $\lambda$ is the symmetric function defined by:
    \begin{enumerate}[(i)]
        \item $Q_{\emptyset} = 1$.
        \item $Q_{(\lambda_1)} = Q_{\lambda_1} = \sum_{a+b=\lambda_1} h_a e_b$ as defined in \S \ref{subsec:LG-OG-cohom}.
        \item $Q_{(\lambda_1, \lambda_2)} = Q_{\lambda_1} Q_{\lambda_2} + 2 \sum_{p=1}^{\lambda_2} (-1)^p Q_{\lambda_1+p} Q_{\lambda_2-p}$ if $(\lambda_1, \lambda_2) \neq (0,0)$, and $Q_{(0,0)} = 0$.
        \item $Q_{\lambda} = \pf \left(Q_{(\lambda^+_i, \lambda^+_j)}\right)_{i,j \in [\ell(\lambda^+)]}$,
        where $\pf(A)$ is the Pfaffian of a matrix $A$.
    \end{enumerate}
    The \emph{Schur $P$-function} of $\lambda$ is $P_{\lambda} \eqdef 2^{-\ell(\lambda)}Q_{\lambda}$.
\end{defn}

A priori $P_{\lambda}$ is only a symmetric function with rational coefficients, but in fact it has integral coefficients. Note that we have not required that $\lambda$ be a partition or that all its parts be positive, and part (iv) above may involve $Q$-functions indexed by such $\lambda$. However, one can show that
\begin{equation} \label{eq:alternating-Q}
    Q_{(\lambda_1, \ldots, \lambda_i, \lambda_{i+1}, \ldots, \lambda_{\ell})} = -Q_{(\lambda_1, \ldots, \lambda_{i+1}, \lambda_{i}, \ldots, \lambda_{\ell})} = Q_{(\lambda_1, \ldots, \lambda_i, \lambda_{i+1}, \ldots, \lambda_{\ell}, 0)},
\end{equation}
which implies that the matrix in (iv) is skew-symmetric of even size. One might think of (iv) as analogous to the Jacobi-Trudi formula expressing a Schur function as a determinant of single-row Schur functions.

\begin{ex}
    \begin{equation*}
    Q_{431} = \pf \begin{pmatrix} Q_{44} & Q_{43} & Q_{41} & Q_{40} \\ Q_{34} & Q_{33} & Q_{31} & Q_{30} \\ Q_{14} & Q_{13} & Q_{11} & Q_{10} \\ Q_{04} & Q_{03} & Q_{01} & Q_{00} \end{pmatrix} = \pf \begin{pmatrix} 0 & Q_{43} & Q_{41} & Q_{4} \\ -Q_{43} & 0 & Q_{31} & Q_3 \\ -Q_{41} & -Q_{31} & 0 & Q_1 \\ -Q_4 & -Q_3 & -Q_1 & 0 \end{pmatrix}.
    \end{equation*}
\end{ex}

\begin{defn} A partition $\lambda$ is \emph{strict} if $\lambda_1 > \cdots > \lambda_{\ell} > 0$. \end{defn}
    Clearly $Q_{\lambda} \in \Gamma_Q$ and $P_{\lambda} \in \Gamma_P$, and in fact $\{Q_\lambda : \text{$\lambda$ strict}\}$ and $\{P_\lambda : \text{$\lambda$ strict}\}$ are $\ZZ$-bases for $\Gamma_Q$ and $\Gamma_P$. Pragacz showed that under the isomorphisms described in \S \ref{subsec:LG-OG-cohom}, the $Q_{\lambda}$ for $\lambda$ strict with $\lambda_1 \leq n$ represent the classes of the Schubert varieties in $\LG(2n)$, and the $P_{\lambda}$ for $\lambda$ strict with $\lambda_1 < n$ represent the classes of the Schubert varieties in $\OG(2n)$ \cite{pragacz-LG-OG}.

    Let $\lambda$ be a strict partition. The \emph{shifted Young diagram} of $\lambda$ is
    \begin{equation*}
        D_{\lambda}' = \{(i,j) : i \in [\ell(\lambda)], i \leq j \leq i+\lambda_i-1\}.
    \end{equation*}
    We draw $D_{\lambda}'$ as a set of boxes in matrix coordinates, with $(1,1)$ at the upper left:
    \begin{equation*}
        D_{532}' = \young(\hfil\hfil\hfil\hfil\hfil,:\hfil\hfil\hfil,::\hfil\hfil).
    \end{equation*}
    \begin{defn} A filling $T$ of the shifted diagram of $\lambda$ by entries from the alphabet $\{1' < 1 < 2' < 2 < \cdots\}$ is a \emph{marked shifted standard tableau} (of shape $\lambda$) if
        \begin{itemize}
            \item the entries are weakly increasing reading down columns and across rows;
            \item no column contains the same unprimed letter twice;
            \item no row contains the same primed letter twice.
        \end{itemize}
        More generally, if $S = \{\cdots < a_{-1} < a_0 < a_1 < \cdots\}$ is any set of integers, we let $\mshSYT(\lambda,S)$ denote the set of marked shifted standard tableaux of shape $\lambda$ on the alphabet $\{\cdots a_{-1}' < a_{-1} < a_0' < a_0 < a_1' < a_1 < \cdots \}$.
    \end{defn}
    
    If $i$ is an unprimed integer we set $\lceil i \rceil = \lceil i' \rceil = i$ and $\epsilon(i) = 1$ and $\epsilon(i') = -1$. If $T$ is a filling of $D'_{\lambda}$ and $(i,j) \in D'_{\lambda}$, let $T(i,j)$ be the entry of $T$ in row $i$ and column $j$.
    \begin{ex}
        \newcommand{\oneprime}{1'}
        \newcommand{\fourprime}{4'}
        \newcommand{\sixprime}{6'}
        \raisebox{2mm}{$\young(\oneprime 3\fourprime 45,:\fourprime 4\sixprime,::5\sixprime)$} is a marked shifted standard tableau of shape $(5,3,2)$. We have $T(2,4) = 6'$, so $\lceil T(2,4) \rceil = 6$ and $\epsilon(T(2,4)) = -1$.
    \end{ex}
    
    The monomial expansion of a Schur $Q$-function can be expressed in terms of shifted tableaux \cite[\S III.8]{Macdonald}:
    \begin{equation} \label{eq:Q-expansion}
        Q_{\lambda} = \sum_{T \in \mshSYT(\lambda,\NN)} \prod_{(i,j) \in D'_{\lambda}} x_{\lceil T(i,j) \rceil}.
    \end{equation}
    We will need certain multivariate generalizations of Schur $Q$-functions introduced by Ivanov; the next definition is \cite[Theorem 4.3]{ivanov-interpolation-PQ}.
    \begin{defn} \label{defn:multiparameter-PQ} Let $\lambda$ be a strict partition and $t$ a sequence of at least $\lambda_1$ indeterminates. The \emph{multiparameter Schur $Q$-function} associated to $\lambda$ is
        \begin{equation*}
            Q_{\lambda}(x; t) = \sum_{T \in \mshSYT(\lambda,\NN)} \prod_{(i,j) \in D'_{\lambda}} (x_{\lceil T(i,j) \rceil} - \epsilon(T(i,j))t_{j-i+1}).
        \end{equation*}
    \end{defn}

\subsection{Degeneracy locus formulas}
\label{subsec:degeneracy-loci-LG-OG}
The Kempf-Laksov formula (Theorem~\ref{thm:kempf-laksov}) expressed the class of a generic Grassmannian degeneracy locus as a determinant in certain Chern classes, which reduces to the Jacobi-Trudi formula in a special case. We will need an analogous formula for isotropic Grassmannian degeneracy loci $\Omega$. These formulas, originally due to Kazarian \cite{kazarian} and expanded upon by Anderson and Fulton \cite{anderson-fulton}, express $[\Omega]$ as a Pfaffian in appropriate Chern classes. These Pfaffians reduce to the Pfaffian formulas for $Q_{\lambda}$ and $P_{\lambda}$ in special cases. In parallel with Definition~\ref{defn:PQ}, suppose $c(1), \ldots, c(\ell)$ are formal power series with constant term $1$ and that $\lambda$ is a finite sequence of nonnegative integers of length $\ell$, and define $Q_{\lambda}(c(1), \ldots, c(\ell))$ by
\begin{enumerate}[(i)]
    \item $Q_{\emptyset}() = 1$.
    \item $Q_{(\lambda_1)}(c(1)) = c(1)_{\lambda_1}$
    \item $Q_{(\lambda_1, \lambda_2)}(c(1), c(2)) = c(1)_{\lambda_1}c(2)_{\lambda_2} + 2\sum_{p=1}^{\lambda_2} (-1)^p c(1)_{\lambda_1+p}c(2)_{\lambda_2-p}$ if $(\lambda_1,\lambda_2) \neq (0,0)$, while $Q_{(0,0)}(c(1), c(2)) = 0$.
    \item $Q_{\lambda}(c(1), \ldots, c(\ell)) = \pf \left (Q_{(\lambda^+_i, \lambda^+_j)}(c(i), c(j))\right)_{i,j \in [\ell(\lambda^+)]}$.
\end{enumerate}
When $\ell(\lambda^+) > \ell(\lambda) = \ell$, the formula in (iv) refers to $c(\ell+1)$, which we take to be $1$. With this convention, $Q_{(\lambda_i, 0)}(c(i), c(\ell+1)) = Q_{(\lambda_i)}(c(i))$. Of course, the $c(i)$ must satisfy certain relations in order for the matrix in (iv) to be skew-symmetric, but these relations will always hold for us.

Suppose $\vb{V}$ is a rank $2n$ vector bundle over a smooth variety $X$ with a rank $n$ subbundle $\vb{G}$ and a flag of subbundles $\vb{H}_{\mu_1} \subseteq \cdots \subseteq \vb{H}_{\mu_s} \subseteq \vb V$, where $\rank \vb H_i = \rank n - i + 1$. Also assume that $\vb{V}$ is equipped with a nondegenerate skew-symmetric form, and that $\vb{G}$ and $\vb{H}_i$ are isotropic with respect to this form. Fix a sequence $k_{\bullet} = (k_1 < \cdots < k_s)$ of integers, setting $k_0 = 0$. For each $k \in [k_s]$, define $\lambda_{k} = \mu_p + k_p - k$ and $c(k) = c(\vb{V}) \,/\, (c(\vb{G})c(\vb{H}_{\mu_p}))$, where $p$ is such that $k_{p-1} < k \leq k_p$.  The corresponding \emph{Lagrangian Grassmannian degeneracy locus} $\Omega^{\LG}$ is the closure of the locus in $X$ over which $\dim(\vb{G} \cap \vb{H}_{\mu_p}) = k_p$ for each $p \in [s]$.

\begin{thm}[\cite{anderson-fulton, kazarian}] \label{thm:lagrangian-locus} If $\lambda$ is a strict partition and $\Omega^{\LG}$ has codimension $|\lambda|$, then $[\Omega^{\LG}] = Q_{\lambda}(c(1), \ldots, c(k_s))$.
\end{thm}
These Pfaffians can be expressed in terms of the multiparameter Schur $Q$-functions. Suppose the partial flag of isotropic bundles $\vb{H}_{\mu_1} \subseteq \cdots \subseteq \vb{H}_{\mu_s} \subseteq \vb V$ extends to a complete isotropic flag $0 = \vb H_{n+1} \subseteq \vb H_{n} \subseteq \cdots \subseteq \vb H_1$ in $\vb V$. Set $t_1' = 0$ and $t_{i+1}' = c_1(\vb H_i/\vb H_{i+1})$ for $i \in [n]$ and $t_i' = 0$ for $i > n+1$. The multiparameter Schur $Q$-function $Q_{\lambda}(x;t)$ is in $\Gamma_Q[t_1, t_2, \ldots]$ \cite[Proposition 2.12]{ivanov-interpolation-PQ}. The relations between the generators $Q_1, Q_2, \ldots$ of $\Gamma_Q$ are described in \cite{anderson-fulton-vex}, and these relations also hold between the Chern classes $\left(\frac{c(\vb V)}{c(\vb G)c(\vb{H}_1)}\right)_d$ where $d \geq 1$, so there is a well-defined ring map $\Gamma_Q[t_1, t_2, \ldots] \to H^*(X)$ sending $Q_d \mapsto \left(\frac{c(\vb V)}{c(\vb G)c(\vb{H}_1)}\right)_d$ and $t_i \mapsto t_i'$.
\begin{prop}[\cite{anderson-fulton-vex}] \label{prop:degeneracy-locus-ivanov-PQ} After identifying $Q_{d}(x_1, x_2, \ldots)$ with $\left(\frac{c(\vb V)}{c(\vb G)c(\vb{H}_1)}\right)_d$ as just described, $Q_{\lambda}(c(1), \ldots, c(k_s)) = Q_{\lambda}(x; -t')$. \end{prop}

We will also need an orthogonal version of Theorem~\ref{thm:lagrangian-locus}, which will be a little more complicated. As above, let $\vb{V}$ be a rank $2n$ vector bundle over a smooth variety $X$ with a rank $n$ subbundle $\vb{G}$ and a flag of subbundles $\vb{H}_{\mu_1} \subseteq \cdots \subseteq \vb{H}_{\mu_s} \subseteq \vb V$---but now we use the slightly different convention $\rank \vb H_i = n - i$. Assume that $\vb{V}$ is equipped with a nondegenerate symmetric form, and that $\vb{G}$ and $\vb{H}_i$ are isotropic with respect to this form. Fix a sequence $k_{\bullet} = (k_1 < \cdots < k_s)$ of integers, setting $k_0 = 0$. Just as above, for each $k \in [k_s]$, define $\lambda_{k} = \mu_p + k_p - k$ where $p$ is such that $k_{p-1} < k \leq k_p$.  The corresponding \emph{orthogonal Grassmannian degeneracy locus} $\Omega^{\OG}$ is the closure of the locus in $X$ over which $\dim(\vb{G} \cap \vb{H}_{\mu_i}) = k_p$ for each $p \in [s]$.

If $c(1), \ldots, c(\ell)$ are formal power series with constant term $1$ over a ring in which $2$ is invertible, define $P_{\lambda}(c(1), \ldots, c(\ell)) = 2^{-\ell(\lambda)}Q_{\lambda}(c(1), \ldots, c(\ell))$. For simplicity in stating the next theorem, we assume that the flag $\vb H_\bullet$ includes a maximal isotropic bundle $\vb H_0$. We also assume that $X$ is connected, which implies that $\delta \eqdef \dim(\vb G_x \cap (\vb{H}_0)_x) \pmod 2$ is constant on $X$. For $k \in [k_s]$, let $p$ be such that $k_{p-1} < k \leq k_p$, and define
\begin{equation*}
c(k) = \frac{c(\vb V)}{c(\vb G)c(\vb H_{\mu_p})} + (-1)^{\delta} c_{\mu_p}(\vb H_0 / \vb H_{\mu_{p}}).
\end{equation*}
\begin{thm}[\cite{anderson-fulton, kazarian}] \label{thm:orthogonal-locus} If $\lambda$ is a strict partition and $\Omega^{\OG}$ has the expected codimension $|\lambda|$, then $[\Omega^{\OG}] = P_{\lambda}(c(1), \ldots, c(k_s))$.
\end{thm}

\begin{rem} Except in degree $\mu_p = \lambda_{k_p}$, the classes $c(k)$ agree with $c(\vb V)/(c(\vb G)c(\vb H_{\mu_p}))$, which was the definition of $c(k)$ used in Theorem~\ref{thm:lagrangian-locus}. The correction terms $(-1)^{\delta} c_{\mu_p}(\vb H_0 / \vb H_{\mu_{p}})$ are independent of the choice of $\vb H_0$, which we will exploit later by choosing different $\vb H_0$ for different $k$.\end{rem}

\subsection{Vexillary involutions} \label{subsec:vex-inv}
In this subsection we give definitions and prove lemmas analogous to those of \S \ref{subsec:vex} for vexillary involutions.
\begin{defn}
    We define two analogues of the Rothe diagram for $y \in \I_n$:
    \begin{equation*}
    \iDO(y) = D(y) \cap \{(i,j) : i \geq j\} \quad \text{and} \quad  \iDSp(y) = D(y) \cap \{(i,j) : i > j\}.
    \end{equation*}
    \end{defn}
    One can show that the size of $\iDK(y)$ is $\iellK(y)$ and that the $K$-orbit $\openiXK_y$ is defined by rank conditions coming from $\Ess(\iDK(y))$ (cf. Lemma~\ref{lem:degeneracy-locus-essential-set}) \cite{HMP1}.
    The next lemma follows from Lemma~\ref{lem:vex-chain} and the fact that $D(w^{-1}) = \{(j,i) : (i,j) \in D(w)\}$.
    \begin{lem} \label{lem:inv-vex-chain}
    An involution $y$ is vexillary if and only if $\Ess(\iDO(y))$ is a chain in the partial order $\SWNE$. \end{lem}
    
    \begin{defn}
        Let $y \in \I_n$. We define two involution codes $\icO(y)$ and $\icSp(y)$ as the row lengths of   $\iDO(y)$ and $\icSp(y)$ respectively. That is,
        \begin{equation*}
        \icO_i(y) = |\{j \in [i] : y(i) > j, y(j) > i\}| \,\,\, \text{and} \,\,\, \icSp_i(y) = |\{j \in [i{-}1] : y(i) > j, y(j) > i\}|.
        \end{equation*}
        We also define two involution shapes $\ishO(y)$ and $\ishSp(y)$, whose conjugates are respectively obtained by sorting the nonzero entries of $\icO(y)$ and of $\icSp(y)$.
    \end{defn}
    
    \begin{lem} \label{lem:vex-inv-code}
        Suppose $y \in \I_n$ is vexillary, and let $j'$ be maximal such that $(i',j') \in \Ess(\iDO(y))$ for some $i'$.  Then
        \begin{equation*}
            \icO_i(y) = \begin{cases}
                |\{j \in [i] : y(j) \neq j\}| & \text{if $i \leq j'$}\\
                c_i(y) & \text{if $i > j'$}
            \end{cases}
        \end{equation*}
        and
         \begin{equation*}
            \icSp_i(y) = \begin{cases}
                |\{j \in [i-1] : y(j) \neq j\}| & \text{if $i \leq j'$}\\
                c_i(y) & \text{if $i > j'$}
            \end{cases}.
        \end{equation*}
    \end{lem}
    
    \begin{proof}
        We proceed by proving a series of claims:
        \begin{enumerate}[(a)]
            \item If $i \leq j'$, then $y(i) = i$ or $y(i) > j'$: If $i = j'$ this holds since $(i',j') \in D(y)$ implies $y(j') > i' \geq j'$, so assume $i < j'$. Note that $y(i) \neq j'$ because otherwise $(i',j') \notin D(y)$. But now if $i < y(i) < j'$, then we have $i < y(i) < j' < y(j')$, which gives a $2143$ pattern in $y$, contradicting the assumption that $y$ is vexillary. If $y(i) < i$, the same contradiction occurs in positions $y(i) < i < j' < y(j')$, so we must have $y(i) = i$ or else $y(i) > j'$.
            \item If $i \leq j'$, then $\{j \in [i] : (i,j) \in \iDO(y)\} = \{j \in [i] : y(j) \neq j\}$: If $y(j) = j$, then $(i,j) \notin \iDO(y)$ for any $i$. Conversely, suppose $y(j) \neq j$ for some $j \in [i]$. Since $j \leq i \leq j'$, (a) implies that $y(j) > j'$ and $y(i) \geq \min(i,j') = i$. But then $y(j) > i$ and $y(i) \geq j$, so $(i,j) \in \iDO(y)$.
            \item  If $i > j'$, then $\icO_i(y) = c_i(y)$:  If $y(i) > i$, then $(i,i) \in \iDO(y)$. But then the connected component of $(i,i)$ in $\iDO(y)$ contains some essential set element $(a,b)$ with $b \geq i > j'$, contradicting the choice of $j'$. It follows that $y(i) \leq i$, which implies that $\icO_i(y) = \icSp_i(y) = c_i(y)$.
        \end{enumerate}
        Taking cardinalities, part (b) shows $\icO_i(y) = |\{j \in [i] : y(j) \neq j\}|$. It also shows that $\{j \in [i-1] : (i,j) \in \iDSp(y)\} = \{j \in [i-1] : y(j) \neq j\}$, so $\icSp_i(y)$ is the size of $|\{j \in [i-1] : y(j) \neq j\}|$ as claimed.
    \end{proof}

    \begin{lem} \label{lem:inv-code-parts}
        Suppose $y \in \I_n$ is vexillary. Write $\Ess(\iDO(y)) = \{(i_1,j_1) \SWNEneq \cdots \SWNEneq (i_s,j_s)\}$, and define $k_p = j_p - \rank y_{[i_p][j_p]}$ for $p \in [s]$. Then $\{\icO_1(y), \ldots, \icO_n(y)\} \setminus \{0\} = [k_s]$, and if $k \in [k_s] \setminus \{k_1, \ldots, k_s\}$ there is a unique $i \in [n]$ with $\icO_i(y) = k$.
    \end{lem}
    
    \begin{proof}
        Recall the notation $D_{iS} = D_{iS}(y) = \{j \in S : (i,j) \in D(y)\}$.
        \begin{enumerate}[(a)]
            \item $D_{j_s[j_s]} = D_{i_s[j_s]}$: The northwest corner closure property of $D(y)$ mentioned in the proof of Lemma~\ref{lem:vex-shape} implies that $(j_s, j_s) \in D(y)$, given that $D(y)$ contains $(i_s,j_s)$ and $(j_s,i_s)$. If $(i_s,j) \in D(y)$ and $j \leq j_s$, then $(i_s,j) \SWNE (j_s,j_s)$, so the same closure property implies $(j_s,j) \in D(y)$; thus, $D_{i_s[j_s]} \subseteq D_{j_s[j_s]}$.
            
            Conversely, suppose $(j_s,j) \in D(y)$ but $(i_s,j) \notin D(y)$ for some $j \leq j_s$. Then we must have $j_s < y(j) \leq i_s$, so $D_{y(j)[j,\infty)} = \emptyset$. But this means that the portion of the connected component of $(j_s,j_s) \in \iDO(y)$ southeast of $(j_s,j_s)$ lies entirely above row $y(j)$; in particular, there is an essential set cell in a row strictly above row $y(j)$. Given that $y(j) \leq i_s$, this contradicts the maximality of $(i_s,j_s) \in \Ess(\iDO(y))$ with respect to $\SWNE$. Thus, $D_{j_s[j_s]} \subseteq D_{i_s[j_s]}$ as well.
              \item If $k \in [k_s]$ then $k = \icO_i(y)$ for some $i$: Part (a) shows that $\icO_{j_s}(y) = |D_{j_s[j_s]}| = |D_{i_s[j_s]}| = k_s$. By Lemma~\ref{lem:vex-inv-code}, the sequence $\icO_1(y), \icO_2(y), \ldots, \icO_{j_s}(y)$ is weakly increasing with consecutive differences in $\{0,1\}$, and starts with $\icO_1(y) \in \{0,1\}$ and ends with $\icO_{j_s}(y) = k_s$. Thus, $\{\icO_i(y) : i \in [n]\} \setminus \{0\} \supseteq [k_s]$.
              \item If $i \in [n]$ then $\icO_i(y) \in \{0\} \cup [k_s]$: We saw in part (b) that if $i \leq j_s$ then $\icO_i(y) \leq \icO_{j_s}(y) = k_s$. Suppose $i > j_s$. By the maximality of $(i_s,j_s) \in \Ess(\iDO(y))$, there cannot be any cells in $\iDO(y)$ right of column $j_s$, so $c_i(y) = |D_{i[j_s]}|$. If $(i,j) \in D(y)$ with $j \leq j_s$, then $(i,j) \SWNE (j_s,j_s)$, so the northwest closure property of $D(y)$ implies $(j_s,j) \in D(y)$ also. Thus $\icO_i(y) = |D_{i[j_s]}| \leq |D_{j_s[j_s]}| = \icO_{j_s}(y) = k_s$.
            \item If $k \in [k_s] \setminus \{k_1, \ldots, k_s\}$ then $k = \icO_i(y)$ for a unique $i$: Suppose that $\icO_i(y) = \icO_{i'}(y) = k$ where $i < i'$. Since $k \notin \{k_1, \ldots, k_s\}$, we cannot have $k \in \{c_1(y), \ldots, c_n(y)\}$ by Lemma~\ref{lem:code-ks}, and so Lemma~\ref{lem:vex-inv-code} forces $i < i' \leq j_s$. Lemma~\ref{lem:vex-inv-code} also says that since $\icO_i(y) = \icO_{i'}(y)$, every member of $[i+1,i']$ is a fixed point of $y$. But if $y(i') = i'$, then $k = \icO_{i'}(y) = c_{i'}(y)$, so $k \in \{k_1, \ldots, k_s\}$ by Lemma~\ref{lem:code-ks}, a contradiction. \qedhere
         \end{enumerate}
    \end{proof}
    
    \begin{lem} \label{lem:inv-shape}
        Suppose $y \in \I_n$ is vexillary. Write $\Ess(\iDO(y)) = \{(i_1,j_1) \SWNEneq \cdots \SWNEneq (i_s,j_s)\}$, and define $k_p = j_p - \rank y_{[i_p][j_p]}$ for $p \in [s]$ and $k_0 = 0$. Then $\ell(\ishO(y)) = k_s$, and $\ishO(y)_{k} = i_p-j_p+1+k_p-k$ for $k \in [k_s]$ where $k_{p-1} < k \leq k_p$.
    \end{lem}
    
    \begin{proof}
       Lemma~\ref{lem:inv-code-parts} implies that $\ell(\ishO(y)) = k_s$. Set $R_p = \{i \in [j_p,i_p] : (i,j_p) \notin \iDO(y)\}$ for each $p$. We first show that
            \begin{equation} \label{eq:code-set}
                \{i : \icO_i(y) \geq k_p\} = ([j_p,i_p] \setminus R_p) \cup [j_p-|R_p|, j_p-1].
            \end{equation}
            Note that the set on the lefthand side has size $\ishO(y)_{k_p}$, while the set on the righthand side has size $i_p-j_p+1$.
    
        \begin{enumerate}[(a)]
            \item If $i \in [j_p-|R_p|, j_p-1]$, then $y(i) \neq i$: Assume for the sake of contradiction that $i \in [j_p-|R_p|, j_p-1]$ has $y(i) = i$. Note that if $i' \in R_p$, then $y(i') < j_p \leq i'$, so $y(i')$ is not a fixed point of $y$. This means $\{y(i') : i' \in R_p\} \neq [j_p-|R_p|, j_p-1]$, since the set on the right contains a fixed point by assumption. So, take $i' \in R_p$ such that $y(i') \notin [j_p-|R_p|, j_p-1]$. Then  $y(i') < j_p-|R_p| \leq i$. But if this happens, then $i < j_p \leq i' < i_p < y(j_p)$ and $y(i') < j_p-|R_p| \leq y(i) = i < j_p < y(i_p)$, so $y$ contains a $2143$ pattern in positions $i, i', i_p, y(j_p)$, a contradiction.
            \item If $i \leq j_p$, then $\icO_i(y) \geq k_p$ if and only if $i \in [j_p-|R_p|, j_p-1]$.
            Part (a) of the proof of Lemma~\ref{lem:vex-inv-code} shows that $\rank y_{[j_p][j_p]} = |\{j \leq j_p : y(j) = j\}|$, so that $\icO_{j_p}(y) = j_p - \rank_{[j_p][j_p]}(y)$ by the same lemma. Also, $|R_p| = \rank y_{(j_p,i_p][j_p]}$, so
            \begin{equation} \label{eq:cjp-calc}
                \icO_{j_p}(y) = j_p - \rank_{[j_p][j_p]}(y) = j_p - (\rank_{[i_p][j_p]}(y) - |R_p|) = k_p + |R_p|.
            \end{equation}
            Lemma~\ref{lem:vex-inv-code} and part (a) now show that $\icO_{j_p-|R_p|+r}(y) = k_p+r$ for $0 \leq r \leq |R_p|$ and that $\icO_j(y) < \icO_{j_p-|R_p|}(y) = k_p$ for $j < j_p-|R_p|$.
            \item If $i > j_s \geq j_p$, then $\icO_i(y) \geq k_p$ if and only if $i \in [j_p,i_p] \setminus R_p$: Lemma~\ref{lem:vex-inv-code} says $\icO_i(y) = c_i(y)$, so $\icO_i(y) \geq k_p$ if and only if $(i,j_p) \in D(y)$ and $i \leq i_p$ by Lemma~\ref{lem:code-ks}.
            \item If $j_s \geq i > j_p$, then $\icO_i(y) \geq k_p$ if and only if $i \in [j_p,i_p] \setminus R_p$: On the one hand, for any such $i$ we have $\icO_i(y) \geq \icO_{j_p}(y) = k_p + |R_p| \geq k_p$, using \eqref{eq:cjp-calc} and  Lemma~\ref{lem:vex-inv-code}. On the other hand, we claim that all such $i$ are in $[j_p,i_p] \setminus R_p$. Suppose otherwise, so that $i \in R_p$, and hence $y(i) < j_p$. Then $D_{i[j_p,\infty)} = \emptyset$, so the portion of the connected component of $(j_p,j_p) \in \iDO(y)$ southeast of $(j_p,j_p)$ lies entirely above row $i$; in particular, there is an essential set cell in a row strictly above row $i$. Given that $i \leq j_s \leq i_s$, this contradicts the maximality of $(i_s,j_s) \in \Ess(\iDO(y))$ with respect to $\SWNE$.    \end{enumerate}
        We have proven equation \eqref{eq:code-set}, which implies $\ishO(y)_{k_p} = i_p-j_p+1$ by taking cardinalities. If $k_{p-1} < k < k_p$, then the fact that $k$ appears exactly once in $\icO(y)$ (Lemma~\ref{lem:inv-code-parts}) implies that $\ishO_k(y) = \ishO_{k+1}(y)+1$, so that $\ishO_k(y) = \ishO_{k_p}(y) + k_p - k$.
    \end{proof}

    We conclude with an analogue of the last lemma for $\iDSp(y)$ instead of $\iDO(y)$.

\begin{defn} Say $y \in \I_n$ is \emph{Sp-vexillary} if it is vexillary and $\Ess(\iDSp(y))$ is a chain under $\SWNE$. \end{defn}
The involution $y = (1,3)(2,5)$ is vexillary but has
\begin{equation*}
    \arraycolsep=3pt \def\arraystretch{0.9}
    \iDO(y) = \begin{array}{ccccc}
          \sq & \cdot & \times &\cdot & \cdot \\
          \sq & \blacksquare & \cdot&\cdot & \times \\
          \times & \cdot & \cdot &\cdot &\cdot \\
          \cdot & \blacksquare & \cdot & \times &\cdot \\
          \cdot & \times & \cdot & \cdot & \cdot
      \end{array} \qquad \text{and} \qquad \iDSp(y) = \begin{array}{ccccc}
        \cdot & \cdot& \times &\cdot &\cdot \\
        \blacksquare & \cdot & \cdot & \cdot & \times \\
        \times & \cdot & \cdot  & \cdot & \cdot \\
        \cdot & \blacksquare & \cdot & \times & \cdot \\
        \cdot & \times & \cdot & \cdot & \cdot
    \end{array}
\end{equation*}
with essential sets highlighted in black, so $y$ is not Sp-vexillary. The involution $y = (1,2)(3,4)$ is not vexillary, but $\Ess(\iDSp(y)) = \emptyset$ is a chain.

Recall that $\idfpf_{2r} = (1,2)(3,4)\cdots(2r{-}1,2r)$ is the minimal-length element of $\Ifpf_{2r}$.
    \begin{lem} \label{lem:fpf-inv-shape} Suppose $y'$ is Sp-vexillary and that $y = y' \times \idfpf_{2r}$ or $y = \idfpf_{2r} \times y'$. Write $\Ess(\iDSp(y)) = \{(i_1,j_1) \SWNEneq \cdots \SWNEneq (i_s,j_s)\}$, and define $k_p = j_p - \rank y_{[i_p][j_p]}$ for $p \in [s]$ and $k_0 = 0$. Then $\ell(\ishSp(y)) = k_s$, and $\ishSp(y)_{k} = i_p-j_p+k_p-k$ for $k \in [k_s]$ where $k_{p-1} < k \leq k_p$. 
        % Moreover, $\ishSp(y)$ is obtained from $\ishO(y)$ by decrementing each part and deleting any $0$'s that result.
    \end{lem}

    \begin{proof}
        Replacing $y'$ with $y' \times \idfpf_{2r}$ does not change essential sets or the sequences $i_p$, $j_p$, and $k_p$. Replacing $y'$ with $\idfpf_{2r} \times y'$ replaces $i_p$ by $i_p + 2r$ and $j_p$ by $j_p + 2r$, and does not change $k_p$. In both cases, the partition $\ishSp(y')$ does not change, nor do the quantities $i_p-j_p+k_p-k$. The truth of the lemma for $y'$ would therefore imply it for  $y$, so we can assume that $y$ itself is Sp-vexillary.
        
        Let $\Ess(\iDO(y)) = \{(i_1', j_1'), \ldots, (i_t', j_t')\}$ and $k_p' = j_p' - \rank y_{[i_p'][j_p']}$ for $p \in [t]$. By Lemma~\ref{lem:vex-inv-code}, $\icSp_1(y) = 0$ and $\icSp_{i}(y) = \icO_{i-1}(y)$ if $1 < i \leq j_t'$, while $\icSp_i(y) = \icO_i(y)$ if $i > j_t'$. Thus, $\ishSp(y)^t$ is obtained from $\ishO(y)^t$ by deleting a part of size $\icO_{j_t'}(y)$. Equivalently, $\ishSp(y)$ is obtained from $\ishO(y)$ by decrementing its first $\icO_{j_t'}(y)$ parts, deleting any $0$'s that result. But part (a) of the proof of Lemma~\ref{lem:inv-code-parts} says $\icO_{j_t'}(y) = k_t' = \ell(\ishO(y))$, so $\ishSp(y)$ is obtained from $\ishO(y)$ by decrementing \emph{every} part.

        Given this, Lemma~\ref{lem:inv-shape} shows that for $k \in [k_t']$ we have $\ishSp(y)_k = i_p' - j_p' + k_p' - k$ where $k_{p-1}' < k \leq k_p'$, ignoring $0$'s. To complete the proof, we must show that $k_s < k \leq k_t'$ if and only if $\ishO(y)_k = 1$, so that $\ishSp(y)$ has length $k_s$ as claimed, and that if $p \leq s$ then $i_p'-j_p'+k_p' = i_p-j_p+k_p$. We must also see that $k_{p-1}' < k \leq k_p'$ if and only if $k_{p-1} < k \leq k_p$ whenever $k \leq k_s$; this will follow from the fact that $s \leq t$ and $k_s \leq k_s'$, which will be clear in each case.
        \begin{enumerate}[(a)]
            \item Suppose $\Ess(\iDO(y))$ contains no cell on the main diagonal. Then $\Ess(\iDO(y)) = \Ess(\iDSp(y))$, so $s = t$ and $k_p = k_p'$ and $(i_p,j_p) = (i_p',j_p')$ for $p \in [s]$. Also, $\ishO(y)_k = i_p'-j_p'+k_p'-k+1 > 1$ for all $k \in [k_s]$ since $(i_p',j_p')$ is always strictly below the main diagonal.
            \item Suppose $\Ess(\iDO(y))$ contains a cell on the main diagonal. Since $\Ess(\iDO(y))$ is a chain under $\SWNE$, it contains at most one diagonal cell, and this cell (if it exists) must be $(i_s',j_s') = (j_s',j_s')$.
            \begin{enumerate}[(i)]
                \item Suppose $(j_s',j_s'-1) \in \iDO(y)$. Then $\Ess(\iDSp(y)) = \Ess(\iDO(y)) \setminus \{(j_s',j_s')\} \cup \{(j_s',j_s'-1)\}$. Thus, $s = t$ and $k_p = k_p'$ and $(i_p,j_p) = (i_p',j_p')$ for $p < s$. Moreover, $(i_s,j_s) = (j_s',j_s'-1)$, and $(j_s',j_s') \in D(y)$ implies $\rank y_{[j_s'][j_s']} = \rank y_{[j_s'][j_s'-1]}$, so $k_s = k_s'-1$ and $i_s'-j_s'+k_s' = i_s-j_s+k_s$. Since $\ishO(y)$ is a strict partition, $\ishO(y)_{k_s'} = i_s'-j_s'+1 = 1$ is its only part of size $1$, so it does hold that $\ishO(y)_k = 1$ if and only if $k \in (k_s,k_s'] = (k_s'-1,k_s']$.
                \item Suppose $(j_s',j_s'-1) \notin \iDO(y)$. Then $\Ess(\iDSp(y)) = \Ess(\iDO(y)) \setminus \{(j_s',j_s')\}$, so $t = s+1$. Lemma~\ref{lem:k-step} shows that $k_s = k_t'-1$, and just as in part (ii) we conclude that $\ishO(y)_k = 1$ if and only if $k \in (k_s, k_t']$. \qedhere
             \end{enumerate}
         \end{enumerate}
    \end{proof}

\section{$\GL(n)$-orbits on $\LG(2n) \times \Fl(n)$ and $\OG(2n) \times \Fl(n)$}
\label{sec:GL-orbits}
Let $g \in \GL(n)$ act on $U \in \Gr(n,2n)$ by $g \cdot U = (g \oplus (g^*)^{-1})(U)$. This defines a $\GL(n)$-action on $\LG(2n)$ and on $\OG(2n)$.
\subsection{Description of orbits}

Define
\begin{align*}
&\grLG \eqdef \{\graph(f) \mid \text{$f : \CC^n \to {\CC^n}^*$ symmetric and invertible}\}\\
&\grOG \eqdef \{\graph(f) \mid \text{$f : \CC^n \to {\CC^n}^*$ skew-symmetric and invertible}\}.
\end{align*}
The next lemma shows that $\grLG = \LG(2n) \cap \grGr$ and $\grOG = \OG(2n) \cap \grGr$, so $\grLG$ and $\grOG$ are open sets in $\LG(2n)$ and $\OG(2n)$.
\begin{lem} $\graph(f) \in \LG(2n)$ if and only if $f : \CC^n \to {\CC^n}^*$ is symmetric, and $\graph(f) \in \OG(2n)$ if and only if $f : \CC^n \to {\CC^n}^*$ is skew-symmetric. \end{lem}

    \begin{proof} It is easy to check that $(\graph(f), \graph(f))^- = 0$ if and only if $f$ is symmetric, and $(\graph(f), \graph(f))^+ = 0$ if and only if $f$ is skew-symmetric, so we just need to see in the latter case that $\graph(f)$ is actually in $\OG(2n)$. Let 
            \begin{equation*}
                Z = \left\{U \in \Gr(n,2n) : \langle U, U \rangle^+ = 0\right\},
            \end{equation*}
            so $\OG(2n)$ is the irreducible component of $Z$ containing $\CC^n \oplus 0$.  Two isotropic subspaces $U_1, U_2 \in Z$ are in the same component of $Z$ if and only if $\dim(U_1 \cap U_2) \equiv n \pmod{2}$. The fact that $f$ is invertible and skew-symmetric implies both that $n$ is even and that $\CC^n \cap \graph(f) = 0$, so indeed $\CC^n$ and $\graph(f)$ are in the same component of $Z$.
        \end{proof}
Given $y \in \I_n$, let $\openincSGX_y \subseteq \grLG \times \Fl(n)$ be
\begin{equation*}
     \{(\graph(f), E_\bullet) : \text{$\rank(E_j \xrightarrow{f} E_i^*) = \rank y_{[i][j]}$ for $i,j \in [n]$, $f$ symmetric}\}.
\end{equation*}
Given $z \in \Ifpf_n$, let $\openincSSGX_z \subseteq \grOG \times \Fl(n)$ be
\begin{equation*}
 \{(\graph(f), E_\bullet) : \text{$\rank(E_j \xrightarrow{f} E_i^*) = \rank z_{[i][j]}$ for $i,j \in [n]$, $f$ skew-symmetric}\}.
\end{equation*}
More generally, for any $y \in \I_n$, it will be convenient to define $\incSSGX_y \subseteq \OG(2n) \times \Fl(n)$ as the closure of 
\begin{equation*}
    \{(\graph(f), E_\bullet) : \text{$\rank(E_j \xrightarrow{f} E_i^*) = \rank y_{[i][j]}$ for $(i,j) \in \Ess(\iDSp(y))$, $f$ skew-symmetric}\}.
\end{equation*}
The proof of Lemma~\ref{lem:SGX-description}(b) below implies that if $y$ is fixed-point-free, the two definitions of $\incSSGX_y$ agree. If $y$ is not fixed-point-free, the notation is slightly misleading, since we are not taking $\incSSGX_y$ to be the closure of $\openincSSGX_y$, and indeed we leave the latter undefined. Replacing $\Ess(\iDSp(y))$ with $[n] \times [n]$ in the general definition of $\incSSGX_y$ above gives the empty set if $y$ is not fixed-point-free.

\begin{prop} \label{prop:orbit-description-LG-OG} The $\GL(n)$-orbits on $\grLG \times \Fl(n)$ are the sets $\openincSGX_y$ for $y \in \I_n$. The $\GL(n)$-orbits on $\grOG \times \Fl(n)$ are the sets $\openincSSGX_z$ for $z \in \Ifpf_n$. \end{prop}

\begin{proof}
    This follows by an argument analogous to the proof of Proposition~\ref{prop:orbit-description-Gr}. Sending $\graph(f)$ to the matrix of $f$ defines an isomorphism from $\grLG$ to the space of invertible symmetric matrices $\GL^{\text{sym}}(n)$, under which the $\GL(n)$-action on $\grLG$ corresponds to the action of $g \in \GL(n)$ on $A \in \GL^{\text{sym}}(n)$ by $g \cdot A = (g^t)^{-1}Ag^t$. It then suffices to show that the $B_n^+$-orbits on $\GL^{\text{sym}}(n)$ are the sets
    \begin{equation*}
        \{M \in \GL^{\text{sym}}(n) : \text{$\rank M_{[i][j]} = \rank y_{[i][j]}$ for $i,j \in [n]$}\}
    \end{equation*}
    for $y \in \I_n$. Similarly, we must see that the $B_n^+$-orbits on the space of invertible skew-symmetric matrices $\GL^{\text{ssym}}(n)$ are the sets
    \begin{equation*}
        \{M \in \GL^{\text{ssym}}(n) : \text{$\rank M_{[i][j]} = \rank z_{[i][j]}$ for $i,j \in [n]$}\}
    \end{equation*}
    for $z \in \Ifpf_n$. These statements were proven in \cite{bagno-cherniavsky, cherniavsky} and in \cite{szechtman}.
\end{proof}

In the $\openincSSGX_z$ case, Proposition~\ref{prop:orbit-description-LG-OG} is not very interesting when $n$ is odd, given that both $\Ifpf_n$ and $\grOG$ are empty. One might instead take $\grOG'$ to be the set of graphs of skew-symmetric maps $\CC^n \to {\CC^n}^*$ of maximal rank $2\lfloor n/2\rfloor$, in which case the $\GL(n)$-orbits on $\grOG' \times \Fl(n)$ are indexed by maximal rank skew-symmetric $n \times n$ $(0,1,{-1})$-matrices with at most one $\pm 1$ in each row and column.

\begin{lem} Each $\openincSGX_y$ and $\openincSSGX_y$ is a fiber bundle over $\grLG$ and $\grOG$, respectively, as well as over $\Fl(n)$. \end{lem}
\begin{proof} Apply Lemma~\ref{lem:fiber-bundle}. \end{proof}

The fibers 
\begin{align*}
    &\openincSGX_y(\graph(f)) = \{E_\bullet : \text{$\rank(E_j \xrightarrow{f} E_i^*) = \rank y_{[i][j]}$ for $i,j \in [n]$}\}\\
    &\openincSSGX_z(\graph(f)) = \{E_\bullet : \text{$\rank(E_j \xrightarrow{f} E_i^*) = \rank z_{[i][j]}$ for $i,j \in [n]$}\}
\end{align*}
for fixed $f$, either symmetric or skew-symmetric as appropriate, were called $\openiXO_y$ and $\openiXSp_z$ in \S \ref{subsec:inv-schubert}.  Since $\GL(n)$ acts transitively on $\grLG$ with stabilizers isomorphic to $\O(n)$, and transitively on $\grOG$ with stabilizers isomorphic to $\Sp(n)$, Proposition~\ref{prop:orbit-description-LG-OG} recovers the fact mentioned in \S \ref{subsec:inv-schubert} that these fibers are the $\O(n)$- and $\Sp(n)$-orbits on $\Fl(n)$.

\begin{lem} \label{lem:SGX-description} Let $y \in \I_n$ and $z \in \Ifpf_n$. \hfill
    \begin{enumerate}[(a)]
    \item $\openincSGX_y$ and $\openincSSGX_z$ are irreducible, with codimensions $\iellO(y)$ and $\iellSp(z)$ respectively.    
    \item $\incSGX_y$ is the closure in $\LG(2n) \times \Fl(n)$ of
    \begin{equation*}
    \{(\graph(f), E_\bullet) : \text{$\rank(E_j \xrightarrow{f} E_i^*) = \rank y_{[i][j]}$ for $(i,j) \in \Ess(\iDO(y))$}\},
    \end{equation*}
    and $\incSSGX_z$ is the closure in $\OG(2n) \times \Fl(n)$ of
    \begin{equation*}
    \{(\graph(f), E_\bullet) : \text{$\rank(E_j \xrightarrow{f} E_i^*) = \rank z_{[i][j]}$ for $(i,j) \in \Ess(\iDSp(z))$}\}.
    \end{equation*}
    \item $\incSGX_y(\graph(f)) = \overline{\openincSGX_y(\graph(f))}$ for any $\graph(f) \in \grLG$. Also, if $E_\bullet \in \Fl(n)$ then $\incSGX_y(E_\bullet) = \overline{\openincSGX_y(E_\bullet)}$. Likewise for $\incSSGX_z$.
    \item The intersection $\incSGX_y(\graph(f)) = \incSGX_y \cap (\{\graph(f)\} \times \Fl(n) \times \Fl(n))$ is transverse, and likewise for $\incSSGX_z$.
    \end{enumerate}
    \end{lem}
        \begin{proof} \hfill
            \begin{enumerate}[(a)]
            \item Given that $\openincSGX_y$ and $\openincSSGX_z$ are orbits of an action of the irreducible group $\GL(n)$ by Proposition~\ref{prop:orbit-description-LG-OG}, they are irreducible. Just as in Lemma~\ref{lem:GX-description}(a), their codimensions are the codimensions of the fibers $\openincSGX_y(\graph(f)) = \iXO_y$ and $\openincSSGX_z(\graph(f)) = \iXSp_z$ over $\Fl(n)$, which are $\iellO(y)$ and $\iellSp(z)$ respectively. This last statement is implicit in Theorem~\ref{thm:inv-schubert-classes}, for instance, since $\codim \iXO_y = \deg [\iXO_y] = \deg \iS_y = \iellO(y)$ and likewise for $\iXSp_z$.
    
            \item Proceed as in the proof of Lemma~\ref{lem:GX-description}(b), replacing Lemma~\ref{lem:degeneracy-locus-essential-set} with the following argument. Let $SM_y^{=,\text{ess}}$, $SM_y^{\leq,\text{ess}}$, $SM_y^{=}$, and $SM_y^{\leq}$ be the sets of symmetric $g \in \Hom(\CC^n, {\CC^n}^*)$ for which the ranks $\rank(\CC^i \hookrightarrow \CC^n \xrightarrow{g} {\CC^n}^* \twoheadrightarrow  {\CC^j}^*)$ are either equal to or at most the ranks $\rank y_{[i][j]}$, either for $(i,j) \in \Ess(\iDO(y))$ or for all $(i,j) \in [n] \times [n]$. We must see that $\overline{SM}_y^{\hspace{1pt}=} = \overline{SM}_y^{\hspace{1pt}=,\text{ess}}$. The inequalities $\rank(E_j \xrightarrow{f} E_i^*) \leq \rank y_{[i][j]}$ for $(i,j) \in \Ess(\iDO(y))$ logically imply those for all $(i,j) \in [n] \times [n]$ \cite[Proposition 3.16]{HMP1}. Thus, ${SM}_y^{\hspace{1pt}\leq,\text{ess}} = {SM}_y^{\hspace{1pt}\leq}$, while $\overline{SM}_y^{\hspace{1pt}=} = {SM}_y^{\hspace{1pt}\leq}$ by \cite[Lemma 5.2]{bagno-cherniavsky}. Since $\overline{SM}_y^{\hspace{1pt}=} \subseteq \overline{SM}_y^{\hspace{1pt}=,\text{ess}} \subseteq {SM}_y^{\hspace{1pt}\leq,\text{ess}}$ by definition, we are done. The same argument works in the skew-symmetric case, using \cite{cherniavsky}.
            \item See the proof of Lemma~\ref{lem:GX-description}(c).
            \item See the proof of Lemma~\ref{lem:GX-description}(d).
        \end{enumerate}
    \end{proof}

\subsection{Cohomological formulas}
\label{subsec:inv-back-stable-reps}
In this subsection we show that $[\incSGX_{y}]$ and $[\incSSGX_{z}]$ are represented by the back-stable involution Schubert polynomials $2^{\twocyc(y)}\ibS_y$ and $\ibSfpf_z$. The proof follows the same outline as that of Theorem~\ref{thm:back-stable-rep}, so we mostly refer back to the results of \S \ref{subsec:back-stable-reps}, indicating differences when necessary.

\begin{lem} \label{lem:grassmannian-locus-LG-OG} If $y \in \I_n$ is vexillary then $\incSGX_y$ is the closure of the locus
    \begin{equation*}
        \{(\graph(f),E_\bullet) : \text{$\dim(U \cap (E_j \oplus \ann{E_i})) = j - \rank y_{[i][j]}$ for $(i,j) \in \Ess(\iDO(y))$} \}.
    \end{equation*}
    If $y \in \I_n$ is Sp-vexillary, then $\incSSGX_y$ is the closure of
    \begin{equation*}
        \{(\graph(f),E_\bullet) : \text{$\dim(U \cap (E_j \oplus \ann{E_i})) = j - \rank y_{[i][j]}$ for $(i,j) \in \Ess(\iDSp(y)))$} \}.
    \end{equation*}
\end{lem}
\begin{proof} Apply the argument of Lemma~\ref{lem:grassmannian-locus}, using Lemma~\ref{lem:SGX-description}(b) and replacing Schubert varieties in $\Gr(n,2n)$ with Schubert varieties in $\LG(2n)$ or $\OG(2n)$.
\end{proof}

Suppose $y \in \I_n$ is vexillary, so we can write $\Ess(\iDO(y)) = \{(i_1, j_1) \SWNEneq \cdots \SWNEneq (i_s, j_s)\}$. Lemma~\ref{lem:grassmannian-locus-LG-OG} shows that $\incSGX_y$ is an example of a Lagrangian Grassmannian degeneracy locus as described in \S \ref{subsec:degeneracy-loci-LG-OG}. To be specific:
\begin{itemize}
    \item[$\scriptstyle \blacktriangleright$] $\vb V$ is the trivial bundle ${\CC^n} \oplus {\CC^n}^*$ over the first factor of $X = \LG(2n) \times \Fl(n)$;
    \item[$\scriptstyle \blacktriangleright$] $\vb G$ is the tautological bundle over the first factor of $X$;
    \item[$\scriptstyle \blacktriangleright$] $\vb{H}_{\mu_p} = \vb{E}_{j_p} \oplus \ann{\vb{E}}_{i_p}$, where $\vb E_{\bullet}$ is the tautological flag of bundles over the second factor of $X$, so $\mu_p = n-\rank(\vb E_{j_p} \oplus \ann{\vb{E}}_{i_p})+1 = i_p-j_p+1$ for $p \in [s]$.
    \item[$\scriptstyle \blacktriangleright$] $k_p = j_p - \rank y_{[i_p][j_p]}$ for $p \in [s]$.
    \item[$\scriptstyle \blacktriangleright$] $\lambda_{k} = \mu_p+k_p-k$ for $k \in [k_s]$, where $p$ is such that $k_{p-1} < k \leq k_p$.
\end{itemize}
Using this data,
\begin{equation} \label{eq:SGX-chern-class}
    c(k) = \frac{c(\vb V)}{c(\vb G)c(\vb{E}_{j_p} \oplus \ann{\vb{E}}_{i_p})} = \frac{1}{c(\vb G)} \frac{(1+x_1) \cdots (1+x_{i_p})}{(1-x_1)\cdots (1-x_{j_p})},
\end{equation}
where $k_{p-1} < k \leq k_p$. Identify $1/c(\vb G)$ with $\sum_d Q_d(x_-)$ as in \S \ref{subsec:LG-OG-cohom}, so
\begin{equation*}
    c(k)_d = h_d(x_{-\infty\cdots i_p} \superslash x_{-\infty\cdots j_p}) = \sum_{a+b = d} Q_a(x_{-})h_b(x_{1\cdots i_p} \superslash x_{1\cdots j_p}).
\end{equation*}
Set $c'(k) = c(k)\ev_{c(\vb G)\to 1}$, so $c'(k)_d = h_d(x_{1\cdots i_p} \superslash x_{1\cdots j_p})$.

\begin{lem} \label{lem:vex-double-schubert-LG-reps} Let $y \in \I_n$ be vexillary, and fix  $\graph(f) \in \grLG$. The class $[\incSGX_y]$ is represented by $Q_{\lambda}(c(1), \ldots, c(k_s))$, and $[\incSGX_y(\graph(f))]$ is represented by $Q_{\lambda}(c'(1), \ldots, c'(k_s))$, where $\lambda = \ishO(y)$. \end{lem}

\begin{proof}
     Lemma~\ref{lem:inv-shape} shows that $\lambda$ as defined above equals $\ishO(y)$, and that this partition is strict. By Lemma~\ref{lem:SGX-description}(a), $\codim \incSGX_y = \hat \ell(y) = |\ishO(y)| = |\lambda|$. This is the expected codimension for $\incSGX_y$ as a Lagrangian Grassmannian degeneracy locus $\Omega^{\LG}$ with respect to the data described above,  so $[\incSGX_y] = Q_\lambda(c(1), \ldots, c(k_s))$ by Theorem~\ref{thm:lagrangian-locus}. As in the proof of Lemma~\ref{lem:vex-double-schubert-reps}, this implies that $[\incSGX_y(\graph(f))] = Q_\lambda(c'(1), \ldots, c'(k_s))$.
\end{proof}

\begin{thm} \label{thm:vex-inv-schubert-LG-polys} Let $y \in \I_n$ be vexillary. As polynomials in $x$, $Q_{\lambda}(c'(1), \ldots, c'(k_s))$ equals $2^{\twocyc(y)}\iS_{y}(x)$ and $Q_{\lambda}(c(1), \ldots, c(k_s))$ equals $2^{\twocyc(y)}\ibS_{y}(x)$, where $\lambda = \ishO(y)$. \end{thm}

    \begin{proof}
        Fix $\graph(f) \in \LG(2n)$. If $y$ is vexillary, so is $y \times 1$, and $\iDO(y \times 1) = \iDO(y)$. Thus, the single polynomial $Q_{\lambda}(c'(1), \ldots, c'(k_s))$ represents the classes $[\incSSGX_{y \times 1_m}(\graph(f))]$ for all $m$. By Theorem~\ref{thm:inv-schubert-classes}, $2^{\twocyc(y)}\iS_y$ represents $[\incSSGX_{y \times 1_m}(\graph(f))]$ when $m = 0$, and in fact for all $m$ since $\iS_y = \iS_{y\times 1}$ (\cite[Theorem 2]{wyser-yong-orthogonal-symplectic}). By Lemma~\ref{lem:unique-stable-rep}, we conclude that $Q_{\lambda}(c'(1), \ldots, c'(k_s)) = 2^{\twocyc(y)}\iS_y$. The same limiting argument given in the proof of Theorem~\ref{thm:vex-double-schubert-polys} now yields $Q_{\lambda}(c(1), \ldots, c(k_s)) = 2^{\twocyc(y)}\ibS_y$.
    \end{proof}

In \cite{HMP4}, Pfaffian formulas for $\iS_y$ were given in the case that $y$ is I-Grassmannian, meaning that $\Ess(\iDO(y)) \subseteq \{m\} \times \NN$ for some $m$ (see \S \ref{sec:grassmannians}). Theorem~\ref{thm:vex-inv-schubert-LG-polys} generalizes those formulas to all vexillary $y$, but is an improvement even when $y$ is I-Grassmannian: \cite{HMP4} expresses $\iS_y$ as a Pfaffian of polynomials $\iS_{y'}$ where $\ishO(y')$ has two rows, but does not give explicit formulas for the two-row case.

    \begin{lem} \label{lem:divided-diff-classes-LG-OG} For $y \in \I_n$ and $z \in \Ifpf_n$,
        \begin{equation*}
            \partial_i [\incSGX_y] = \begin{cases}
                [\incSGX_{s_i \circ y \circ s_i}] & \text{if $\ell(ys_i) < \ell(y)$ and $\twocyc(s_i \circ y \circ s_i) = \twocyc(y)$}\\
                2[\incSGX_{s_i \circ y \circ s_i}] & \text{if $\ell(ys_i) < \ell(y)$ and $\twocyc(s_i \circ y \circ s_i) = \twocyc(y)-1$}\\
                0 & \text{otherwise}
            \end{cases}
        \end{equation*}
        and
        \begin{equation*}
            \partial_i [\incSSGX_z] = \begin{cases}
                [\incSSGX_{s_i \circ z \circ s_i}] & \text{if $\ell(zs_i) < \ell(z)$ and $s_i \circ z \circ s_i \in \Ifpf_n$}\\
                0 & \text{otherwise}
            \end{cases}
        \end{equation*}
    \end{lem}        
    
    \begin{proof}  
    Fixing an invertible and symmetric or skew-symmetric $f$ as appropriate, Wyser and Yong proved these recurrences hold for the classes $[\incSGX_y(\graph(f))]$ and $[\incSSGX_z(\graph(f))]$ (Proposition~\ref{prop:inv-schubert-recurrence} and Theorem~\ref{thm:inv-schubert-classes}). The lemma then follows by the argument of Lemma~\ref{lem:divided-diff-classes}.
    \end{proof}

\begin{thm} \label{thm:back-stable-rep-LG}
    For $y \in \I_n$, the back-stable involution Schubert polynomial $2^{\twocyc(y)}\ibS_y$ represents the class $[\incSGX_y]$.
\end{thm}

\begin{proof}
    Theorem~\ref{thm:vex-inv-schubert-LG-polys} implies that $2^{\twocyc(w_0)}\ibS_{w_0}$ represents $[\incSGX_{w_0}]$, so the theorem follows from the matching divided difference recurrences of Proposition~\ref{prop:inv-schubert-recurrence} and Lemma~\ref{lem:divided-diff-classes-LG-OG}.
\end{proof}

Suppose $y \in \I_n$ is Sp-vexillary, so we can write $\Ess(\iDSp(y)) = \{(i_1, j_1) \SWNEneq \cdots \SWNEneq (i_s, j_s)\}$. Lemma~\ref{lem:grassmannian-locus-LG-OG} implies that $\incSSGX_y$ is an example of an orthogonal Grassmannian degeneracy locus as described in \S \ref{subsec:degeneracy-loci-LG-OG}. To be specific:
\begin{itemize}
    \item[$\scriptstyle \blacktriangleright$] $\vb V$ is the trivial bundle ${\CC^n} \oplus {\CC^n}^*$ over the first factor of $X = \OG(2n) \times \Fl(n)$;
    \item[$\scriptstyle \blacktriangleright$] $\vb G$ is the tautological bundle over the first factor of $X$;
    \item[$\scriptstyle \blacktriangleright$] $\vb{H}_{\mu_p} =   \vb{E}_{j_p} \oplus \ann{\vb{E}}_{i_p}$, where $\vb E_{\bullet}$ is the tautological flag of bundles over the second factor of $X$, so $\mu_p = n-\rank(\vb E_{j_p} \oplus \ann{\vb{E}}_{i_p}) = i_p-j_p$ for $p \in [s]$.
    \item[$\scriptstyle \blacktriangleright$] $k_p = j_p - \rank y_{[i_p][j_p]}$ for $p \in [s]$.
    \item[$\scriptstyle \blacktriangleright$] $\lambda_{k} = \mu_p+k_p-k$ for $k \in [k_s]$, where $p$ is such that $k_{p-1} < k \leq k_p$.
\end{itemize}

For each $p$, $\vb H_0 = \vb{E}_{j_p} \oplus \ann{\vb{E}}_{j_p}$ is a maximal isotropic bundle containing $\vb{E}_{j_p} \oplus \ann{\vb{E}}_{i_p}$. In accordance with \S \ref{subsec:degeneracy-loci-LG-OG}, we define
\begin{align} \label{eq:SSGX-chern-class}
    c(k) &= \frac{c(\vb V)}{c(\vb G)c(\vb{E}_{j_p} \oplus \ann{\vb{E}}_{i_p})} + (-1)^{\delta}c_{\mu_p}\left( \frac{\vb{E}_{j_p} \oplus \ann{\vb{E}}_{j_p}}{\vb{E}_{j_p} \oplus \ann{\vb{E}}_{i_p}} \right)  \nonumber \\
    &= \frac{1}{c(\vb G)} \frac{(1+x_1)\cdots (1+x_{i_p})}{(1-x_1)\cdots (1-x_{j_p})} + (-1)^{j_p} \left[(1+x_{j_p+1})\cdots (1+x_{i_p})\right]_{i_p-j_p} \nonumber \\
    &= \frac{1}{c(\vb G)} \frac{(1+x_1)\cdots (1+x_{i_p})}{(1-x_1)\cdots (1-x_{j_p})} + (-1)^{j_p} x_{j_p+1} \cdots x_{i_p}
\end{align}
for $k \in [k_s]$, where $k_{p-1} < k \leq k_p$.  Here, the quantity $\delta = \dim(\vb G_x \cap (\vb{E}_{j_p} \oplus \ann{\vb{E}}_{j_p})_x) \pmod{2}$ is independent of $x \in X$ since $X$ is connected, and we compute it by choosing $x = (\CC^n \oplus 0, E_\bullet)$ for some $E_\bullet \in \Fl(n)$, so $\delta = \dim((\CC^n \oplus 0) \cap (E_{j_p} \oplus \ann{E_{j_p}})) = j_p$.

The choice of $\vb H_0$ as $\vb E_{j_p} \oplus \ann{\vb E_{j_p}}$ was somewhat arbitrary, and could be replaced by $\vb E_{j} \oplus \ann{\vb E_{j}}$ for any $j \in [j_p,i_p]$, but in fact one can check that all such choices of $j$ give exactly the same expression in \eqref{eq:SSGX-chern-class}. As before, we identify $1/c(\vb G)$ with $\sum_d Q_d(x_-)$. Set $c'(k) = c(k)\ev_{c(\vb G) \to 1}$.

\begin{lem} \label{lem:vex-double-schubert-OG} 
Let $n$ be even, and fix $\graph(f) \in \grOG$. Suppose $y = y' \times \idfpf_{2r}$ or $y = \idfpf_{2r} \times y'$ where $y' \in \I_n$ is Sp-vexillary, and that $\codim \incSSGX_y = |\iDSp(y)|$. Then $[\incSSGX_y]$ is represented by $P_{\lambda}(c(1), \ldots, c(k_s))$, and $[\incSSGX_y(\graph(f))]$ is represented by $P_{\lambda}(c'(1), \ldots, c'(k_s))$, where $\lambda = \ishSp(y)$. If moreover $y \in \Ifpf_n$, then $\iSfpf_{y}(x) = P_{\lambda}(c'(1), \ldots, c'(k_s))$ as polynomials. \end{lem}

\begin{proof}
    Analogous to the proofs of Lemma~\ref{lem:vex-double-schubert-LG-reps} and Theorem~\ref{thm:vex-inv-schubert-LG-polys}, replacing Lemma~\ref{lem:inv-shape} by Lemma~\ref{lem:fpf-inv-shape}.
\end{proof} 

If $y' \in \Ifpf_n$ is vexillary, then it is Sp-vexillary and the hypotheses of Lemma~\ref{lem:vex-double-schubert-OG} hold for $y = y' \times \idfpf_{2r}$ or $y = \idfpf_{2r} \times y'$, because $\codim \incSSGX_y = \iellSp(y) = |\iDSp(y)|$ by Lemma~\ref{lem:SGX-description}. In \S\ref{sec:grassmannians}, we develop some combinatorial tools for working with rank conditions defining $\incSSGX_y$ which will reveal $\incSSGX_y$ to be a Grassmannian degeneracy locus in more cases than the ones considered here. This will give Pfaffian formulas under hypotheses more general and less awkward than those of Lemma~\ref{lem:vex-double-schubert-OG} (Theorem~\ref{thm:fpf-vex-formula}).

\begin{thm} \label{thm:back-stable-rep-OG}
    For $z \in \Ifpf_n$, the back-stable involution Schubert polynomial $\ibSfpf_z$ represents the class $[\incSSGX_z]$.
\end{thm}

\begin{proof}
    Analogous to the proof of Theorem~\ref{thm:back-stable-rep-LG}, using Lemma~\ref{lem:vex-double-schubert-OG}.
\end{proof}

From Theorems~\ref{thm:back-stable-rep-LG} and \ref{thm:back-stable-rep-OG} we obtain geometric interpretations of the involution Stanley symmetric functions $\iF_y = \ibS_y|_{x_{+} \to 0}$ and $\iFfpf_z = \ibSfpf_z|_{x_+\to 0}$.
\begin{thm} \label{thm:inv-stanley-interpretation}
    Fix $E_{\bullet} \in \Fl(n)$. For $y \in \I_n$ and $z \in \Ifpf_n$, the symmetric functions $2^{\twocyc(y)}\iF_y$ and $\iFfpf_z$ represent the classes in $H^*(\LG(2n))$ and $H^*(\OG(2n))$ of \emph{involution graph Schubert varieties}: the closures of, respectively,
    \begin{equation*}
        \openincSGX_y(E_\bullet) = \{\graph(f) \in \grLG : \text{$\rank(E_j \xrightarrow{f} E_i^*) = \rank y_{[i][j]}$ for $i,j \in [n]$}\}
    \end{equation*}
    and 
    \begin{equation*}
        \openincSSGX_z(E_\bullet) = \{\graph(f) \in \grOG : \text{$\rank(E_j \xrightarrow{f} E_i^*) = \rank z_{[i][j]}$ for $i,j \in [n]$}\}.
    \end{equation*}
\end{thm}

\begin{cor} \label{cor:schur-Q-positivity}
    $2^{\twocyc(y)}\iF_y$ is Schur $Q$ positive, and $\iFfpf_z$ is Schur $P$ positive.
\end{cor}
\begin{proof}
    Pragacz \cite{pragacz-LG-OG} showed that the classes of the Schubert varieties in $H^*(\LG(n)) \simeq \Gamma_Q/(Q_{n+1}, Q_{n+2}, \ldots)$ are represented by Schur $Q$-functions, which implies that $f \in \Gamma_Q$ represents the class of a variety if and only if $f$ is Schur $Q$ positive modulo the ideal $(Q_{n+1}, Q_{n+2}, \ldots)$. Since $2^{\twocyc(y)}\iF_y = 2^{\twocyc(y \times 1^m)}\iF_{y \times 1^m}$ for all $m$, Theorem~\ref{thm:inv-stanley-interpretation} implies that $2^{\twocyc(y)}\iF_y$ is Schur $Q$ positive modulo $\bigcap_m (Q_{n+m+1}, Q_{n+m+2}, \ldots) = 0$. The same argument works for $\iFfpf_z$.
\end{proof}

\section{I-Grassmannian and Sp-vexillary involutions} \label{sec:grassmannians}

\begin{defn}
    An involution of the form $(\phi_1,m+1)(\phi_2,m+2)\cdots (\phi_k,m+k)$ where $\phi_1 < \cdots < \phi_k < m$ is called \emph{I-Grassmannian}, or \emph{$m$-I-Grassmannian} if we wish to specify $m$. 
\end{defn}

\begin{prop}[\cite{HMP4}, Proposition-Definition 4.16]
    An involution $y \in \I_n$ is $m$-I-Grassmannian if and only if $\Ess(\iDO(y)) \subseteq \{m\} \times \NN$.
\end{prop}

\begin{ex}
  $y = (1,5)(2,6)(4,7)$ is $4$-I-Grassmannian,
    \begin{equation*}
          \arraycolsep=3pt \def\arraystretch{0.9}
          \iDO(y) = \begin{array}{ccccccc}
                \sq &  &  &  & & &  \\
                \sq & \sq &  &  & & & \\
              \sq & \sq & \times &  & & & \\
                \sq & \sq & \cdot & \sq & & & \\
                \times & \cdot & \cdot & \cdot & \cdot & &\\
                \cdot & \times & \cdot & \cdot & \cdot & \cdot &\\
                \cdot & \cdot & \cdot & \times & \cdot & \cdot & \cdot
            \end{array}
    \end{equation*}
    and $\ishO(y) = (4,2,1)$ and $\Ess(y) = \{(4,2),(4,4)\}$.
\end{ex}
In general, if $y = (\phi_1,m+1) \cdots (\phi_k,m+k)$ is $m$-I-Grassmannian then the parts of $\ishO(y)$ are the nonzero column lengths of $\iDO(y)$, so $\ishO(y)_i = m-\phi_i+1$ for $i \in [k]$.

\begin{thm} \label{thm:i-grassmannians}
    Fix $E_\bullet \in \Fl(n)$. If $y \in \I_n$ is $m$-I-Grassmannian, then $\incSGX_y(E_\bullet) \subseteq \LG(2n)$ is a Schubert variety whose class in $H^*(\LG(2n))$ is represented by $Q_{\ishO(y)}$.
\end{thm}

\begin{proof} 
    Write $y = (\phi_1,m+1)\cdots (\phi_k,m+k)$ where $\phi_1 < \cdots < \phi_k \leq m \leq m+k \leq n$. Considering the form of $\iDO(y)$, it is clear that $\Ess(\iDO(y)) \subseteq \{(m,\phi_p) : p \in [k]\}$ and that $\phi_p - \rank y_{[m][\phi_p]} = p$ for $p \in [k]$. Lemma~\ref{lem:grassmannian-locus-LG-OG} and Lemma~\ref{lem:SGX-description}(c) therefore imply that $\incSGX_y(E_\bullet)$ is the closure of
\begin{equation} \label{eq:LG-schubert} \{U \in \grLG : \text{$\dim(U \cap (E_{\phi_p} \oplus \ann{E_m})) = p$ for $p \in [k]$}\}.   \end{equation}
(To be precise, in \eqref{eq:LG-schubert} we are imposing \emph{all} rank conditions coming from row $n$ of $\iDO(y)$, not just those coming from the essential set. However, Lemma~\ref{lem:grassmannian-locus-LG-OG} shows that after taking closures, we get the same variety whether we impose rank conditions for all $(i,j)$ or only the essential ones. Hence, the same is true if any intermediate set of rank conditions is imposed, as are we doing here.)
    The closure of \eqref{eq:LG-schubert} is also the Schubert variety in $\LG(2n)$ labeled by the strict partition $(m+1-\phi_1, \ldots, m+1-\phi_k) = \ishO(y)$ with respect to the isotropic flag
    \begin{equation*}
        0 \oplus \ann{E_n} \subseteq 0 \oplus \ann{E_{n-1}} \cdots \subseteq 0 \oplus \ann{E_m} \subseteq E_1 \oplus \ann{E_m} \subseteq \cdots \subseteq E_m \oplus \ann{E_m},
    \end{equation*}
    and its class is represented by $Q_{\ishO(y)}$ \cite[\S 6]{pragacz-LG-OG}.
\end{proof}

\begin{cor} \label{cor:I-Grassmannian-formula} If $y$ is I-Grassmannian, then $\iF_y = P_{\ishO(y)}$. \end{cor}
    \begin{proof}
        Say $y$ is $r$-I-Grassmannian, and set $y' = 1_{n-r} \times y$ for some fixed $n \geq r$, so that $y'$ is $n$-Grassmannian and $\iF_{y'} = \iF_y$. Theorem~\ref{thm:i-grassmannians} and Theorem~\ref{thm:inv-stanley-interpretation} imply that $2^{\twocyc(y')}\iF_{y'}$ equals $Q_{\ishO(y')} = 2^{\ell(\ishO(y'))}P_{\ishO(y')}$ modulo the ideal $(Q_{n+1}, Q_{n+2}, \ldots)$, given that they both represent the same class in $H^*(\LG(2n))$. Since $\twocyc(y')$, $\ishO(y')$, and $\iFO_{y'}$ are all independent of $n$, letting $n \to \infty$ shows that $2^{\twocyc(y)}\iF_{y} = 2^{\ell(\ishO(y))}P_{\ishO(y)}$, which is equivalent to the claimed formula.
    \end{proof}

The formula of Corollary~\ref{cor:I-Grassmannian-formula} was obtained earlier in \cite{HMP4}, where it was used as a base case for a recurrence of the form $\iF_y = \sum_z \iF_z$; via this recurrence one can compute the Schur P expansion of $\iF_y$, and in particular deduce that it is Schur P positive. Theorem~\ref{thm:i-grassmannians} provides a new and geometrically natural reason to consider I-Grassmannian involutions.

Recall that $w \in S_n$ is vexillary if it avoids the permutation pattern $2143$. Stanley showed that $F_w$ is a single Schur function $s_{\lambda}$ if and only if $w$ is vexillary \cite{stanleysymm}. In \cite{HMP4} it was shown that $2^{\twocyc(y)} \iFO_y = Q_{\mu}$ for some strict partition $\mu$ if and only if $y$ is vexillary, and our results recover one of these implications.

\begin{thm} If $y$ is vexillary, then $2^{\twocyc(y)}\iF_y = Q_{\ishO(y)}$. \end{thm}
    \begin{proof}
        Apply Theorem~\ref{thm:vex-inv-schubert-LG-polys}, setting the variables $x_+$ to zero.
    \end{proof}

We now turn to fixed-point-free involutions and $\OG(2n)$. The obvious guess to define an ``fpf-I-Grassmannian'' involution $z$ would be to require that $\Ess(\iDSp(z)) \subseteq \{m\} \times \NN$. However, this condition is too restrictive: it does imply that $\incSSGX_z(E_\bullet)$ is a Schubert variety just as in Theorem~\ref{thm:i-grassmannians}, but one cannot obtain a Schubert variety for every strict partition this way. For instance, no such $z$ has $\ishSp(z) = (3,1)$. The problem is that seemingly different sets of rank conditions on a skew-symmetric map can turn out to be equivalent in ways not seen with symmetric maps. For instance, if $z = (1,3)(2,5)(4,6)$ then
\begin{equation*}
    \arraycolsep=3pt \def\arraystretch{0.9}
    \iDSp(z) = \begin{array}{cccccc}
        \cdot & \cdot & \times & \cdot & \cdot & \cdot\\
        \circ & \cdot & \cdot & \cdot & \times & \cdot\\
        \times & \cdot & \cdot & \cdot & \cdot & \cdot\\
        \cdot & \circ & \cdot & \cdot & \cdot & \times\\
        \cdot & \times & \cdot & \cdot & \cdot & \cdot\\
        \cdot & \cdot & \cdot & \times & \cdot & \cdot\\
    \end{array}
\end{equation*}
Therefore $\graph(f) \in \openincSSGX_{z}(E_\bullet)$ if and only if $\rank(f : E_2 \to E_1^*) \leq 0$ and $\rank(f : E_4 \to E_2^*) \leq 1$, corresponding to the two elements of $\Ess(\iDSp(z))$. But $\rank(f : E_4 \to E_2^*) \leq 1$ implies $\rank(f : E_2 \to E_2^*) \leq 1$, which implies $\rank(f : E_2 \to E_2^*) \leq 0$ because a skew-symmetric map must have even rank, which in turn implies $\rank(f : E_2 \to E_1^*) \leq 0$. Thus, $\openincSSGX_z(E_\bullet)$ is actually defined by the single rank condition $\rank(f : E_4 \to E_2^*) \leq 1$, so is a Schubert variety in $\LG(12)$.

\begin{defn}[\cite{HMP5}] Suppose $z \in \Ifpf_{2r}$. Let $z = (a_1, b_1)\cdots (a_r, b_r)$ be the disjoint cycle decomposition of $z$, where $a_i < b_i$ for each $i$. Define involutions $\dearcR(z)$ and $\dearcL(z)$ as the product of $(a_i, b_i)$ over all $i$ such that, respectively,
    \begin{itemize}
        \item $a_i < b_j < b_i$ for some $j$;
        \item $a_i < a_j < b_i$ for some $j$.
    \end{itemize}
\end{defn}

\begin{ex}
    If $z = (1,3)(2,5)(4,6)$, then $\dearcR(z) = (2,5)(4,6)$ and $\dearcL(z) = (1,3)(2,5)$. Drawing $z$ as a perfect matching, one obtains $\dearcR(z)$ by removing all arcs which do not have the right endpoint of another arc underneath them, and similarly for $\dearcL(z)$ changing ``right'' to ``left'':
    \begin{center} $z = $
        \begin{tikzpicture}
            %python
            %matching([3,5,1,6,2,4])
            \node[circle, fill, inner sep=0pt, minimum size = 1mm] (664544164) at (0.0,0) {}; %%%inserted by pymacro
            \node[circle, fill, inner sep=0pt, minimum size = 1mm] (664544165) at (0.333,0) {}; %%%inserted by pymacro
            \node[circle, fill, inner sep=0pt, minimum size = 1mm] (664544166) at (0.667,0) {}; %%%inserted by pymacro
            \node[circle, fill, inner sep=0pt, minimum size = 1mm] (664544167) at (1.0,0) {}; %%%inserted by pymacro
            \node[circle, fill, inner sep=0pt, minimum size = 1mm] (664544168) at (1.333,0) {}; %%%inserted by pymacro
            \node[circle, fill, inner sep=0pt, minimum size = 1mm] (664544169) at (1.667,0) {}; %%%inserted by pymacro
            \draw (664544164) to[bend left=65] (664544166);\draw (664544165) to[bend left=75] (664544168);\draw (664544167) to[bend left=65] (664544169); %%%inserted by pymacro
            %%%
        \end{tikzpicture} \hspace{1cm} $\dearcR(z) = $
        \begin{tikzpicture}
            %python
            %matching([1,5,3,6,2,4])
            \node[circle, fill, inner sep=0pt, minimum size = 1mm] (909762557) at (0.0,0) {}; %%%inserted by pymacro
            \node[circle, fill, inner sep=0pt, minimum size = 1mm] (909762558) at (0.333,0) {}; %%%inserted by pymacro
            \node[circle, fill, inner sep=0pt, minimum size = 1mm] (909762559) at (0.667,0) {}; %%%inserted by pymacro
            \node[circle, fill, inner sep=0pt, minimum size = 1mm] (909762560) at (1.0,0) {}; %%%inserted by pymacro
            \node[circle, fill, inner sep=0pt, minimum size = 1mm] (909762561) at (1.333,0) {}; %%%inserted by pymacro
            \node[circle, fill, inner sep=0pt, minimum size = 1mm] (909762562) at (1.667,0) {}; %%%inserted by pymacro
            \draw (909762558) to[bend left=75] (909762561);\draw (909762560) to[bend left=65] (909762562); %%%inserted by pymacro
            %%%
        \end{tikzpicture} \hspace{1cm} $\dearcL(z) = $
        \begin{tikzpicture}
            %python
            %matching([3,5,1,4,2,6])
            \node[circle, fill, inner sep=0pt, minimum size = 1mm] (41690034) at (0.0,0) {}; %%%inserted by pymacro
            \node[circle, fill, inner sep=0pt, minimum size = 1mm] (41690035) at (0.333,0) {}; %%%inserted by pymacro
            \node[circle, fill, inner sep=0pt, minimum size = 1mm] (41690036) at (0.667,0) {}; %%%inserted by pymacro
            \node[circle, fill, inner sep=0pt, minimum size = 1mm] (41690037) at (1.0,0) {}; %%%inserted by pymacro
            \node[circle, fill, inner sep=0pt, minimum size = 1mm] (41690038) at (1.333,0) {}; %%%inserted by pymacro
            \node[circle, fill, inner sep=0pt, minimum size = 1mm] (41690039) at (1.667,0) {}; %%%inserted by pymacro
            \draw (41690034) to[bend left=65] (41690036);\draw (41690035) to[bend left=75] (41690038); %%%inserted by pymacro
            %%%
        \end{tikzpicture}
    \end{center}
\end{ex}

Recall that for general $y \in \I_n$, we define $\incSSGX_y$ as the closure in $\OG(2n) \times \Fl(n)$ of
\begin{equation*}
    \{(\graph(f), E_\bullet) : \text{$\rank(E_j \xrightarrow{f} E_i^*) = \rank y_{[i][j]}$ for $(i,j) \in \Ess(\iDSp(y))$, $f$ skew-symmetric}\}.
\end{equation*}
As a technical crutch in the next few lemmas, let us also define  $\incSSGX_y^{\leq}$ as the closure in $\OG(2n) \times \Fl(n)$ of
\begin{equation*}
    \{(\graph(f), E_\bullet) : \text{$\rank(E_j \xrightarrow{f} E_i^*) \leq \rank y_{[i][j]}$ for $(i,j) \in \Ess(\iDSp(y))$, $f$ skew-symmetric}\}.
\end{equation*}

\begin{lem} \label{lem:dearcL-ess} Let $y \in \I_n$, and let $(a,b)$ be a cycle in $y$ with $a < b$ and such that $a < i < b$ implies $y(i) < a$. Set $y' = y(a,b)$. Then $\incSSGX_y = \incSSGX_{y'}$. 
\end{lem}

\begin{proof} We prove the stronger fact that $\iDSp(y') = \iDSp(y)$ and $\rank y'_{[i][j]} = \rank y_{[i][j]}$ for $(i,j) \in \iDSp(y)$. Observe that $D(y)$ and $D(y')$ agree outside of the square $[a,b] \times [a,b]$, and contain the following respective configurations:
        \begin{center}
        \begin{tikzpicture}
        \fill[gray!20!white] (0,0) -- (1.5,0) -- (1.5,-.5) -- (.5,-.5) -- (.5,-1.5) -- (0,-1.5) -- (0,0);
        \draw (1.5,0) -- (0,0) -- (0,-1.5);
        \draw (1.5,-.5) -- (.5,-.5) -- (.5,-1.5);
        \draw (-.2,.2) node {$\circ$};
        \draw (0.5, 0.2) node[right] {$\times$};
        \draw (-.2,-.5) node[below]  {$\times$};
        \draw (-.7, .2) node {$a$};
        \draw (-.7, -.5) node[below] {$b$};
        \draw (-.2, .6) node {$a$};
        \draw (.5, .65) node[right] {$b$};
        \end{tikzpicture} \hspace{1in} \begin{tikzpicture}
            \fill[gray!20!white] (0,0) -- (1.5,0) -- (1.5,-.5) -- (.5,-.5) -- (.5,-1.5) -- (0,-1.5) -- (0,0);
            \draw (1.5,0) -- (0,0) -- (0,-1.5);
            \draw (1.5,-.5) -- (.5,-.5) -- (.5,-1.5);
            \draw (0.5, -.5) node[below right] {$\times$};
            \draw (-.2,.2) node  {$\times$};
            \draw (-.7, .2) node {$a$};
            \draw (-.7, -.5) node[below] {$b$};
            \draw (-.2, .6) node {$a$};
            \draw (.5, .65) node[right] {$b$};
            \end{tikzpicture}
        \end{center}
        Here $\times$ denotes a point $(i, w(i))$ and $\circ$ a point in $D(w)$, where $w$ is $y$ or $y'$ as appropriate. The shaded regions contain no $\times$ and hence no $\circ$, from which it follows that $D(y') = D(y) \setminus \{(a,a)\}$, so $\iDSp(y') = \iDSp(y)$. Also, if $(i,j) \in \iDSp(y)$ then $(i,j) \notin [a,b) \times [a,b)$, which in turn implies $\rank y_{[i][j]} = \rank y'_{[i][j]}$.

    \end{proof}

\begin{lem} \label{lem:dearcR-ess} Let $y \in \I_n$, and let $(a,b)$ be a cycle in $y$ with $a < b$. Assume that $a < i < b$ implies $y(i) > b$ and that $\rank y_{[b][b]}$ is even. Set $y' = y(a,b)$. Then
    \begin{enumerate}[(a)]
        \item The sets $\Ess(\iDSp(y'))$ and $\Ess(\iDSp(y))$ agree outside $[a,b] \times [a,b]$.
        \item $\rank y_{[i][j]} = \rank y'_{[i][j]}$ if $(i,j) \in \Ess(\iDSp(y'))$ or $(i,j) \notin [a,b) \times [a,b)$.
        \item $\incSSGX_{y'}^{\leq} \subseteq \incSSGX_{y}^{\leq}$.
    \end{enumerate}
\end{lem}

    \begin{proof} $D(y)$ and $D(y')$ agree outside of the square $[a,b] \times [a,b]$, and contain the following respective configurations:
\begin{equation} \label{eq:diagram-change} 
    \raisebox{-1cm}{\begin{tikzpicture}
    \draw (-.2,.2) node {$\circ$};
    \draw (.3,.2) node {$\cdots$};
    \draw (.7,.2) node {$\circ$};
    \draw (1.1,.2) node {$\circ$};
    \draw (1.5,.2) node {$\times$};
    \draw (-.2,-.1) node {$\vdots$};
    \draw (-.2,-.6) node {$\circ$};
    \draw (-.2,-1) node {$\circ$};
    \draw (-.2,-1.4) node {$\times$};
    \draw (1.1,-.6) node {$\circ$};
    \draw (.7,-1) node {$\circ$};
    \draw (.7,-.6) node {$\circ$};
    \draw (1.1,-1) node {$\circ$};
    \draw (.3,-.6) node {$\cdots$};
    \draw (.3,-1) node {$\cdots$};
    \draw (.7,-.1) node {$\vdots$};
    \draw (1.1,-.1) node {$\vdots$};
    \draw (-1, .2) node {$a$};
    \draw (-1, -1.4) node {$b$};
    \draw (-.2, .7) node {$a$};
    \draw (1.5, .75) node {$b$};
    \end{tikzpicture} \hspace{1in} \begin{tikzpicture}
        \draw (-.2,.2) node {$\times$};
        \draw (.7,-.1) node {$\cdots$};
        \draw (1.1,-.1) node {$\circ$};
        \draw (1.5,-.1) node {$\circ$};
        \draw (.2,-.1) node {$\circ$};
        \draw (.2,-.4) node {$\vdots$};
        \draw (.2,-.9) node {$\circ$};
        \draw (.2,-1.3) node {$\circ$};
        \draw (1.5,-1.3) node {$\times$};
        \draw (.7,-1.3) node {$\cdots$};
        \draw (.7,-.9) node {$\cdots$};
        \draw (1.1,-1.3) node {$\circ$};
        \draw (1.1,-.9) node {$\circ$};
        \draw (1.5,-.4) node {$\vdots$};
        \draw (1.1,-.4) node {$\vdots$};
        \draw (1.5,-.9) node {$\circ$};
        \draw (-1, .2) node {$a$};
        \draw (-1, -1.3) node {$b$};
        \draw (-.2, .7) node {$a$};
        \draw (1.5, .75) node {$b$};
    \end{tikzpicture}}
\end{equation}

From \eqref{eq:diagram-change} one can see that if $(i,j) \notin [a,b) \times [a,b)$, then $\rank y'_{[i][j]} = \rank y_{[i][j]}$. Also, the only cell $(i,j) \in \Ess(\iDSp(y'))$ which could also be in $[a,b] \times [a,b]$ is $(b,b{-}1)$, and \eqref{eq:diagram-change} again shows that $\rank y'_{[b][b-1]} = \rank y_{[b][b-1]}$. This proves parts (a) and (b) of the lemma. As for part (c), we consider three cases:
\begin{enumerate}[(i)]
    \item Suppose $b = a+1$. Then $\Ess(\iDSp(y')) = \Ess(\iDSp(y))$, and $\incSSGX^{\leq}_y = \incSSGX^{\leq}_{y'}$ since these two varieties are defined by the same rank conditions.
    \item Suppose $b > a+1$ and $(b{+}1,b{-}1) \notin \iDSp(y)$. Then $(b{-}1,b{-}2) \in \Ess(\iDSp(y))$ and $\Ess(\iDSp(y')) = \Ess(\iDSp(y)) \setminus \{(b{-}1,b{-}2)\} \cup \{(b,b{-}1)\}$. Moreover, $\Ess(\iDSp(y)) \cap [a,b] \times [a,b] = \{(b{-}1,b{-}2)\}$ and $\Ess(\iDSp(y')) \cap [a,b] \times [a,b] = \{(b,b{-}1)\}$. Thus, $\incSSGX^{\leq}_{y'}$ is defined by the same rank conditions as $\incSSGX^{\leq}_y$ except that the condition
    \begin{equation} \label{eq:rank-condition-1}
        \rank(f : E_{b-2} \to E_{b-1}^*) \leq \rank y_{[b-1][b-2]}
    \end{equation}
    is replaced by
    \begin{equation} \label{eq:rank-condition-2}
        \rank(f : E_{b-1} \to E_b^*) \leq \rank y'_{[b][b-1]}.
    \end{equation}

    We must show that \eqref{eq:rank-condition-2} implies \eqref{eq:rank-condition-1}. Suppose that \eqref{eq:rank-condition-2} holds. Then
    \begin{align} 
        \rank(f : E_{b-2} \to E_{b-1}^*) &\leq \rank(f : E_{b-1} \to E_{b-1}^*) \nonumber \\
        & \leq \rank(f : E_{b-1} \to E_{b}^*) \label{eq:rank-inequalities-1}\\
        & \leq \rank y'_{[b][b-1]} \label{eq:rank-inequalities-2}\\
        & = 1+\rank y_{[b-1][b-1]} = 1+\rank y_{[b-1][b-2]}. \nonumber 
    \end{align}
    By assumption, $1+\rank y_{[b-1][b-1]} = -1+\rank y_{[b][b]}$ is odd. On the other hand, $\rank(f : E_{b-1} \to E_{b-1}^*)$ is even because $f$ is skew-symmetric. It follows that one of the inequalities \eqref{eq:rank-inequalities-1} and \eqref{eq:rank-inequalities-2} is strict, so that \eqref{eq:rank-condition-1} holds.

    \item Finally, suppose $b > a+1$ and $(b{+}1,b{-}1) \in \iDSp(y)$. Now $\Ess(\iDSp(y')) = \Ess(\iDSp(y)) \setminus \{(b{-}1,b{-}2)\}$, and no element of $\Ess(\iDSp(y'))$ lies in $[a,b) \times [a,b)$, so $\incSSGX^{\leq}_y$ is defined by the same rank conditions as $\incSSGX^{\leq}_{y'}$ together with the extra condition \eqref{eq:rank-condition-1}. We must see that this extra condition is actually implied by those defining $\incSSGX^{\leq}_{y'}$. 
    
    If $(i,j), (i',j') \in D(y)$ satisfy $(i',j') \in \{(i{+}1,j), (i,j{+}1)\}$, then $\rank y_{[i][j]} = \rank y_{[i'][j']}$, so $\rank(f : E_{j'} \to E_{i'}^*) \leq \rank y_{[i'][j']}$ implies $\rank(f : E_{j} \to E_{i}^*) \leq \rank y_{[i][j]}$. Thus, if the rank condition $\rank(f : E_{j} \to E_{i}^*) \leq \rank y_{[i][j]}$ holds for every $(i,j) \in \Ess(\iDSp(y))$, then it holds for every $(i,j) \in \iDSp(y)$ by induction. In particular, $(b,b-1) \in \iDSp(y')$, so \eqref{eq:rank-condition-2} holds on $\incSSGX^{\leq}_{y'}$, which implies \eqref{eq:rank-condition-1} by case (ii). \qedhere
\end{enumerate}
    \end{proof}

\begin{thm} \label{thm:dearc-equality}
    Suppose $z \in \Ifpf_n$. Then $\incSSGX_z = \incSSGX_{\dearcL(z)} = \incSSGX_{\dearcR(z)}$.
\end{thm}

\begin{proof}
    Since $\dearcL(z)$ is obtained from $z$ by a sequence of transformations $y \leadsto y'$ as in the statement of Lemma~\ref{lem:dearcL-ess}, that lemma implies that $\incSSGX_z = \incSSGX_{\dearcL(z)}$.

    As for $\dearcR(z)$, let $C$ be the set of cycles $(a,b)$ of $z$ such that $a < b$ and $a < i < b$ implies $z(i) > b$. The cycles in $C$ are necessarily non-nesting and non-crossing in the sense that $C = \{(a_1, b_1), \ldots, (a_k, b_k)\}$ where
    \begin{equation} \label{eq:NNNC}
        a_k < b_k < a_{k-1} < b_{k-1} < \cdots < a_1 < b_1.
    \end{equation}
    Set $z^p = z(a_1, b_1)\cdots (a_{p}, b_{p})$ for $0 \leq p \leq k$, so $z^0 = z$ and $z^k = \dearcR(z)$. The permutation matrices $z_{[b_p][b_p]}$ and $z^{p-1}_{[b_p][b_p]}$ are the same given that $a_1, b_1, \ldots, a_{p-1}, b_{p-1} > b_p$; and, since $z$ is fixed-point-free, $\rank z^{p-1}_{[b_p][b_p]} = \rank z_{[b_p][b_p]}$ is even. Each transformation $z^{p-1} \leadsto z^p$ is therefore of the form considered in Lemma~\ref{lem:dearcR-ess}(c), so that lemma implies $\incSSGX^{\leq}_{z^{p}} \subseteq \incSSGX^{\leq}_{z^{p-1}}$ for $p = 1, \ldots, k$, hence  $\incSSGX^{\leq}_{\dearcR(z)} \subseteq \incSSGX^{\leq}_z$.

    Now we claim that if $(i,j) \in \Ess(\iDSp(\dearcR(z)))$, then $\rank \dearcR(z)_{[i][j]} = \rank z_{[i][j]}$. If $(i,j)$ is not in any square $[a_p,b_p) \times [a_p,b_p)$, this is clear from Lemma~\ref{lem:dearcR-ess}(b). Suppose $(i,j) \in [a_p,b_p) \times [a_p,b_p)$. The squares $[a_p,b_p) \times [a_p,b_p)$ are disjoint by \eqref{eq:NNNC}, so if $q \neq p$ then $(i,j) \notin [a_q,b_q) \times [a_q,b_q)$, hence $\rank z^{q-1}_{[i][j]} = \rank z^{q}_{[i][j]}$ by Lemma~\ref{lem:dearcR-ess}(b). But given that $(i,j) \notin [a_q,b_q) \times [a_q,b_q)$ for $q > p$, Lemma~\ref{lem:dearcR-ess}(a) shows that $(i,j) \in \Ess(\iDSp(z^{p-1}))$. Thus $\rank z^{p-1}_{[i][j]} = \rank z^p_{[i][j]}$ by Lemma~\ref{lem:dearcR-ess}(b).

    Suppose $(\graph(f),E_\bullet) \in \openincSSGX_z$, so $\rank(E_j \xrightarrow{f} E_i^*) = \rank z_{[i][j]}$ for $i,j \in [n]$. The previous paragraph then implies that $\rank(E_j \xrightarrow{f} E_i^*) = \rank \dearcR(z)_{[i][j]}$ for $(i,j) \in \Ess(\iDSp(\dearcR(z)))$, so $(\graph(f), E_\bullet) \in \incSSGX_{\dearcR(z)}$. We now see
    \begin{equation*}
        \openincSSGX_z \subseteq \incSSGX_{\dearcR(z)} \subseteq \incSSGX^{\leq}_{\dearcR(z)} \subseteq \incSSGX^{\leq}_z.
    \end{equation*}
    As $\incSSGX_z = \incSSGX^{\leq}_z$ by the proof of Lemma~\ref{lem:SGX-description}(b), taking closures proves the theorem.

\end{proof}

\begin{lem} \label{lem:dearc-shapes} If $z \in \Ifpf_n$, then $\ishSp(z) = \ishSp(\dearcR(z)) = \ishSp(\dearcL(z))$. \end{lem}
\begin{proof}
    Since $\iDSp(z) = \iDSp(\dearcL(z))$ as per the proof of Lemma~\ref{lem:dearcL-ess}, we get $\ishSp(z) = \ishSp(\dearcL(z))$. As for $\dearcR(z)$, suppose $y$ and $y'$ are related as in Lemma~\ref{lem:dearcR-ess}. The diagrams \eqref{eq:diagram-change} show that $\icSp_i(y) = \icSp_i(y')$ if $i \notin [a,b]$, while for some $r$,
    \begin{gather*}
        \icSp_a(y), \icSp_{a+1}(y), \ldots, \icSp_{b-1}(y), \icSp_b(y) = r, r+1, \ldots, r{+}b{-}a{-}1, r\\
        \text{and}\\
        \icSp_a(y'), \icSp_{a+1}(y'), \ldots, \icSp_{b-1}(y'), \icSp_b(y') = r, r, r+1, \ldots, r{+}b{-}a{-}1.
    \end{gather*}
    So, $\ishSp(y) = \ishSp(y')$, which proves the lemma because $\dearcR(z)$ is obtained from $z$ by a sequence of transformations of the form $y \leadsto y'$, as per the proof of Lemma~\ref{thm:dearc-equality}.
\end{proof}

We can now state the Pfaffian formulas for $\iSfpf_z$ and $\ibSfpf_z$ promised in \S \ref{subsec:inv-back-stable-reps}. Recall that $y \in \I_n$ if it is vexillary and $\Ess(\iDSp(z))$ is a chain under $\SWNE$.
\begin{thm} \label{thm:fpf-vex-formula} Let $z\in \Ifpf_n$. Suppose $y \in \{z, \dearcR(z), \dearcL(z)\}$ is Sp-vexillary. Let $\Ess(\iDSp(y)) = \{(i_1,j_1) \SWNEneq \cdots \SWNEneq (i_s,j_s)\}$, and let $k_p = j_p - \rank y_{[i_p][j_p]}$ for $p \in [s]$ and $k_0 = 0$. Set $\lambda = \ishSp(z)$. For $k \in [k_s]$, set
    \begin{align*}
        c(k)_d &= h_d(x_{-\infty \cdots i_p} \superslash x_{-\infty \cdots j_p}) + (-1)^{j_p}x_{j_p+1}\cdots x_{i_p}\\
            &= (-1)^{j_p}x_{j_p+1}\cdots x_{i_p} + \sum_{a+b+c=d} Q_a(x_-) e_b(x_{1\cdots i_p})h_c(x_{1\cdots j_p})
    \end{align*}
    where $k_{p-1} < k \leq k_p$, and $c'(k) = c(k)\ev_{x_-\to 0}$. Then $\iSfpf_{z} = P_{\lambda}(c'(1), \ldots, c'(k_s))$ and $\ibSfpf_{z} = P_{\lambda}(c(1), \ldots, c(k_s))$.
\end{thm}
    
\begin{proof} Fix $\graph(f) \in \OG(2n)$. By Theorem~\ref{thm:dearc-equality} we have $\incSSGX_y = \incSSGX_z$. Also, $\ishSp(y) = \ishSp(z)$ by Lemma~\ref{lem:dearc-shapes}. Thus, $\codim \incSSGX_y = \iellSp(z) = |\ishSp(z)| = |\ishSp(y)|$, so Lemma~\ref{lem:vex-double-schubert-OG} shows that $P_{\ishSp(y)}(c'(1), \ldots, c'(k_s))$ represents the class $[\incSSGX_y(\graph(f))] = [\incSSGX_z(\graph(f))]$. By Theorem~\ref{thm:inv-schubert-classes},  $\iSfpf_{z}$ also represents $[\incSSGX_z(\graph(f))]$.

All of these statements still hold when $z$ is replaced by $z \times \idfpf_{2r}$ for any $r$ and $y$ is replaced by $y \times 1_{2r}$ or $y \times \idfpf_{2r}$ accordingly. These replacements do not change $P_{\ishSp(y)}(c'(1), \ldots, c'(k_s))$ or $\iSfpf_z$, so both of these polynomials represent the classes $[\incSSGX_{z \times \idfpf_{2r}}(\graph(f))]$ for all $r$, hence are equal by Lemma~\ref{lem:unique-stable-rep}. To conclude that $\ibSfpf_{z} = P_{\lambda}(c(1), \ldots, c(k_s))$, apply a limiting argument as in the proof of Theorem~\ref{thm:vex-double-schubert-polys}. \end{proof}

    The next definition gives an important class of fixed-point-free involutions to which Theorem~\ref{thm:fpf-vex-formula} applies.
    \begin{defn}
        A fixed-point-free involution $z$ is \emph{m-fpf-I-Grassmannian} if $\dearcR(z)$ is m-I-Grassmannian, or simply \emph{fpf-I-Grassmannian}. 
    \end{defn}
    
    \begin{ex}
        Let $z = (1,3)(2,5)(4,6)$. Then $\dearcR(z) = (2,5)(4,6)$, which is $4$-I-Grassmannian with shifted shape $(3,1)$, so $z$ is fpf-I-Grassmannian with fpf shifted shape $(2)$. Note that $z$ itself is \emph{not} I-Grassmannian, nor is $\dearcL(z) = (1,3)(2,5)$.
    \end{ex}

\begin{thm} \label{thm:fpf-i-grassmannians}
    Fix $E_\bullet \in \Fl(n)$. If $z \in \Ifpf_n$ is fpf-I-Grassmannian, then $\incSSGX_z(E_\bullet) \subseteq \OG(2n)$ is a Schubert variety whose class in $H^*(\OG(2n))$ is represented by $P_{\ishSp(z)}$.
\end{thm}

\begin{proof} 
    Write $\dearcR(z) = (\phi_1,m+1)\cdots (\phi_k,m+k)$ where $\phi_1 < \cdots < \phi_k \leq m \leq m+k \leq n$. As in the proof of Theorem~\ref{thm:i-grassmannians}, we compute $\Ess(\iDSp(\dearcR(z)))$ and find that $\incSSGX_z(E_\bullet)$ is the closure of
\begin{equation} \label{eq:OG-schubert-equations} \{U \in \grOG : \text{$\dim(U \cap (E_{\phi_p} \oplus \ann{E_m})) = p$ for $p \in [k]$}\},  \end{equation}
    now applying Theorem~\ref{thm:dearc-equality} as well.  This is a Schubert variety in $\OG(2n)$ indexed by the strict partition $(n-\phi_1, \ldots, n-\phi_k) = \shfpf(z)$, and its class is represented by $P_{\shfpf(z)}$ \cite[\S 6]{pragacz-LG-OG}.
\end{proof}

\begin{cor} \label{cor:fpf-I-Grassmannian-formula} If $z$ is fpf-I-Grassmannian, then $\iFfpf_z = P_{\shfpf(z)}$. \end{cor}
    \begin{proof}
        Follows from Theorem~\ref{thm:fpf-i-grassmannians} and Theorem~\ref{thm:inv-schubert-classes} as in the proof of Corollary~\ref{cor:I-Grassmannian-formula}.
    \end{proof}

The formula of Corollary~\ref{cor:fpf-I-Grassmannian-formula} was obtained earlier in \cite{HMP5}, where it was used as a base case for a recurrence of the form $\iFfpf_z = \sum_{z'} \iFfpf_{z'}$; via this recurrence one can compute the Schur P expansion of $\iFfpf_z$, and in particular deduce that it is Schur P positive. Theorems~\ref{thm:fpf-i-grassmannians} and \ref{thm:dearc-equality} provide natural geometric reasons to consider fpf-I-Grassmannian involutions and the $\dearc$ operations.

\begin{thm} \label{thm:fpf-vexillary} If $z$, $\dearcL(z)$, or $\dearcR(z)$ is Sp-vexillary, then $\iFSp_z = P_{\ishSp(z)}$. \end{thm}
    \begin{proof}
        Apply Theorem~\ref{thm:fpf-vex-formula}, setting the variables in $x_+$ to zero.
    \end{proof}

The converse of Theorem~\ref{thm:fpf-vexillary} is false: $z = (1,3)(2,4)(5,7)(6,8)$ has $\iFfpf_z = P_2$, but none of $z$, $\dearcL(z)$, or $\dearcR(z)$ are Sp-vexillary or even vexillary. In \cite{HMP5}, the fixed-point-free involutions $z$ such that $\iFfpf_z = P_{\mu}$ for some strict partition $\mu$ were characterized as those avoiding a finite list of patterns.

\section{Tableau formulas}
\label{sec:tableaux}
In this section we use Proposition~\ref{prop:degeneracy-locus-ivanov-PQ} and Theorem~\ref{thm:vex-inv-schubert-LG-polys} to give tableau formulas for $\ibS_y$ when $y \in \I_n$ is vexillary.
\begin{defn}
    An \emph{essential path} for a vexillary involution $y \in \I_n$ is a sequence of points $(a_1, b_1), (a_2, b_2), \ldots, (a_{n+1}, b_{n+1})$ in $[0,n] \times [0,n]$ such that 
    \begin{itemize}
        \item $(a_{r+1}, b_{r+1}) \SWNEneq (a_r, b_r)$ for $1 \leq r \leq n$;
        \item $a_1 = b_1$;
        \item Each $(i,j) \in \Ess(\iDO(y))$ occurs as some $(a_r, b_r)$.
    \end{itemize}
\end{defn}
Note that an essential path necessarily ends at $(n,0)$. Thinking of $[0,n] \times [0,n]$ as a graph with edges between horizontally or vertically adjacent lattice points, an essential path is a directed path from the diagonal to the southwest corner $(n,0)$ which moves south or west at each step, and hits every point in $\Ess(\iDO(y))$.

\begin{ex} \label{ex:essential-paths} Here are two possible essential paths for $y = (1,5)(2,4)$. Cells in $\iDO(y)$ are marked by circles and cells in $\Ess(\iDO(y))$ by shaded circles, while $\times$ indicates a point $(i,y(i))$; the lower left corner is $(5,0)$:
    
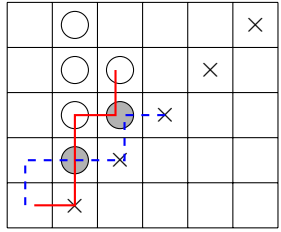
\begin{center}
        \begin{tikzpicture}[scale=0.6]
            \draw (-1,0) grid (5,5);
            \draw (0.5,0.5) node {$\times$};
            \draw (1.5,1.5) node {$\times$};
            \draw (2.5,2.5) node {$\times$};
            \draw (3.5,3.5) node {$\times$};
            \draw (4.5,4.5) node {$\times$};
            \draw[fill=black!30!white] (0.5,1.5) circle[radius=0.3cm];
            \draw (0.5,2.5) circle[radius=0.3cm];
            \draw (0.5,3.5) circle[radius=0.3cm];
            \draw (0.5,4.5) circle[radius=0.3cm];
            \draw[fill=black!30!white] (1.5,2.5) circle[radius=0.3cm];
            \draw (1.5,3.5) circle[radius=0.3cm];
            \draw[red,thick] (1.4,3.5) -- (1.4,2.5) -- (0.5,2.5) -- (0.5,1.5) -- (0.5,0.5) -- (-0.4,0.5);
            \draw[blue,thick,dashed] (2.5,2.5) -- (1.6,2.5) -- (1.6,1.5) -- (0.5,1.5) -- (-0.6,1.5) -- (-0.6,0.5);
        \end{tikzpicture}
    \end{center}
\end{ex}

Suppose $\epath$ is an essential path for some $y \in \I_n$. Define an $(n+1)$-tuple $x^{\epath}$ by $x^{\epath}_1 = 0$ and 
\begin{equation*}
    x^{\epath}_{r+1} = \begin{cases}
        x_{i+1} & \text{if the $r$\textsuperscript{th} step of $\epath$ moves from row $i$ to row $i+1$}\\
        -x_{j} & \text{if the $r$\textsuperscript{th} step of $\epath$ moves from column $j$ to column $j-1$}
    \end{cases}
\end{equation*}
If $\epath$ is the solid red path starting at $(2,2)$ in Example~\ref{ex:essential-paths}, $x^{\epath}$ is $(0, x_3, -x_2, x_4, x_5, -x_1)$. If $\epath$ is the dashed blue path starting at $(3,3)$, then $x^{\epath}$ is $(0, -x_3, x_4, -x_2, -x_1, x_5)$.

Recall the multiparameter Schur $Q$-function $Q_{\lambda}(x; t)$ from Definition~\ref{defn:multiparameter-PQ}.
\begin{thm} \label{thm:tableau-formula-LG}
    Suppose $y \in \I_n$ is vexillary with involution shape $\ishO(y) = \lambda$, and $\epath$ is an essential path for $y$ starting at diagonal position $(j,j)$. Then $2^{\twocyc(y)}\ibS_y = Q_{\lambda}(x_{-\infty\cdots j}; -x^{\epath})$.   
\end{thm}

\begin{proof} Let $\epath = \{(a_{n+1}, b_{n+1}) \SWNEneq \cdots \SWNEneq (a_1,  b_1)\}$, and set $\vb E(r) = \vb E_{b_r} \oplus \ann {\vb{E}}_{a_r}$ for each $r \in [n+1]$, where $\vb E_\bullet$ is the flag of tautological bundles over $\Fl(n)$. The chain
\begin{equation*}
0 = \vb E(n+1) \subseteq \vb E(n) \subseteq \cdots \subseteq \vb E(1)
\end{equation*}
is then a complete isotropic flag. As per Lemma~\ref{lem:grassmannian-locus-LG-OG}, $\incSGX_y$ is a Lagrangian Grassmannian degeneracy locus defined by conditions on intersections with certain of the bundles $\vb E(r)$, namely those for which $(a_r, b_r) \in \Ess(\iDO(y))$. Accordingly, by Proposition~\ref{prop:degeneracy-locus-ivanov-PQ} $[\incSGX_y]$ is represented by $Q_{\lambda}(x; t')$ where $t'_1 = 0$ and $t'_{r+1} = c_1(\vb E(r) / \vb E(r+1))$, and $Q_d(x)$ is identified with $\left(\frac{1}{c(\vb G)c(\vb \E(1))}\right)_d$.
 
Under our identifications from \S \ref{subsec:LG-OG-cohom}, $1/c(\vb G)$ is represented by $\sum_{d\geq 0} Q_d(x_{-}) = \prod_{i=1}^{\infty} \frac{1+x_{-i}}{1-x_{-i}}$ while
\begin{equation*}
\frac{1}{c(\vb E(1))} = \frac{1}{c(\vb E_j)c(\ann{\vb E}_j)}  = \frac{1}{c(\vb E_j)c((\CC^{2n} / \vb E_j)^*)} = \frac{c(\vb E_j^*)}{c(\vb E_j)} = \prod_{i=1}^j \frac{1+x_i}{1-x_i}
\end{equation*}
where $j = a_1 = b_1$. Thus, $\left(\frac{1}{c(\vb G)c(\vb \E(1))}\right)_d$ is identified with $Q_d(x_{-\infty \cdots j})$. Also,
\begin{align*}
t_{r+1}' &= c_1(\vb E(r) / \vb E(r+1)) = c_1\left(\frac{\vb E_{b_r}}{\vb E_{b_{r+1}}} \oplus \left(\frac{\vb E_{a_{r+1}}}{\vb E_{a_r}}\right)^*\right)\\
&= \begin{cases}
c_1((\vb E_{a_{r+1}}/\vb E_{a_r})^*) = x_{a_r+1} & \text{if $b_{r+1} = b_r$ and $a_{r+1} = a_r+1$}\\
c_1(\vb E_{b_{r}}/\vb E_{b_{r+1}}) = -x_{b_r} & \text{if $a_{r+1} = a_r$ and $b_{r+1} = b_r-1$}
\end{cases}\\
&= x^{\epath}_{r+1}.
\end{align*} 
 
It follows that $Q_{\lambda}(x_{-\infty \cdots j}; -x^{\epath})$ represents the class $[\incSGX_y]$. We claim that it represents the classes $[\incSGX_{y \times 1_m}]$ for all $m$. Let $\epath^+$ be the path $\epath$ with a step from $(n,0)$ to $(n{+}1,0)$ appended at the end. Then $\epath^+$ is an essential path for $y \times 1$ since $\Ess(\iDO(y \times 1)) = \Ess(\iDO(y))$, so $Q_{\lambda}(x_{-\infty \cdots j}; -x^{\epath^+})$ represents $[\incSGX_{y \times 1}]$ by what we just showed. Definition~\ref{defn:multiparameter-PQ} makes clear that $Q_{\lambda}(x; t)$ is independent of $t_r$ for $r > \lambda_1$; since $\lambda_1 < n$, we see that $Q_{\lambda}(x_{-\infty \cdots j}; -x^{\epath}) = Q_{\lambda}(x_{-\infty \cdots j}; -x^{\epath^+})$ also represents $[\incSGX_{y \times 1}]$, and by induction every class $[\incSGX_{y \times 1_m}]$. Since $2^{\twocyc(y)}\ibS_y$ also represents these classes by Theorem~\ref{thm:back-stable-rep-LG}, Lemma~\ref{lem:unique-stable-rep} forces $Q_{\lambda}(x_{-\infty \cdots j}; -x^{\epath}) = 2^{\twocyc(y)}\ibS_y$.
\end{proof}  

Theorem~\ref{thm:tableau-formula-LG} writes $2^{\twocyc(y)}\ibS_y$ as a sum of products of expressions $x_i \pm x_j$, each product indexed by a tableau. It would also be interesting to express $2^{\twocyc(y)}\ibS_y$ as a sum of honest monomials associated to tableaux. In the simple case that $\Ess(\iDO(y))$ has one element, which implies $y = (a,m+1)(a+1,m+2)\cdots (a+k-1,m+k)$ where $a+k-1 \leq m$, one can show that
\begin{equation*}
    2^{\twocyc(y)}\ibS_y = \sum_{T} \prod_{(i,j) \in D_{\ishO(y)}'} x_{\lceil T(i,j) \rceil} 
\end{equation*}
where $T$ runs over marked shifted semistandard tableaux on $(-\infty,m]$ in which every entry exceeding $k$ is primed.

\bibliographystyle{plain}
\bibliography{../../bib/algcomb}

\begin{thebibliography}{10}

\bibitem{anderson-double-schubert}
D.~Anderson.
\newblock Double {Schubert} polynomials and double {Schubert} varieties.
\newblock https://people.math.osu.edu/anderson.2804/papers/geomschpolyn.pdf,
  2007.

\bibitem{anderson-fulton-vex}
D.~Anderson and W.~Fulton.
\newblock Degeneracy loci, {Pfaffians}, and vexillary signed permutations in
  types {B}, {C}, and {D}.
\newblock arXiv:1210.2066v1, 2012.

\bibitem{anderson-fulton}
D.~Anderson and W.~Fulton.
\newblock Chern class formulas for classical-type degeneracy loci.
\newblock arXiv:1504.03615v2, 2015.

\bibitem{bagno-cherniavsky}
E.~Bagno and Y.~Cherniavsky.
\newblock Congruence {$B$-orbits} and the {Bruhat} poset of involutions of the
  symmetric group.
\newblock {\em Discrete Math.}, 312:1289--1299, 2012.

\bibitem{bergeron-billey}
N.~Bergeron and S.~Billey.
\newblock R{C}-graphs and {S}chubert polynomials.
\newblock {\em Experiment. Math.}, 2(4):257--269, 1993.

\bibitem{billeyjockuschstanley}
S.~Billey, W.~Jockusch, and R.~P. Stanley.
\newblock Some combinatorial properties of {Schubert} polynomials.
\newblock {\em J. Algebraic Combin.}, 2:345--374, 1993.

\bibitem{borelcohom}
Armand Borel.
\newblock Sur la cohomologie des espaces fibr{\'e}s principaux et des espaces
  homog{\`e}nes de groupes de {Lie} compacts.
\newblock {\em Annals of Mathematics}, 57:115--207, 1953.

\bibitem{broecker-tom-dieck}
T.~Br{\"o}cker and T.~{tom Dieck}.
\newblock {\em Representations of Compact {Lie} Groups}.
\newblock Springer-Verlag, 1985.

\bibitem{cherniavsky}
Y.~Cherniavsky.
\newblock On involutions of the symmetric group and congruence {$B$-orbits} of
  anti-symmetric matrices.
\newblock {\em Internat. J. Algebra Comput.}, 21:841--856, 2011.

\bibitem{fulton-double-schubert}
W.~Fulton.
\newblock {Flags, {S}chubert polynomials, degeneracy loci, and determinantal
  formulas}.
\newblock {\em Duke Mathematical Journal}, 65:381--420, 1992.

\bibitem{youngtableaux}
W.~Fulton.
\newblock {\em Young Tableaux: With Applications to Representation Theory and
  Geometry}.
\newblock Cambridge University Press, 1997.

\bibitem{HMP4}
Z.~Hamaker, E.~Marberg, and B.~Pawlowski.
\newblock Schur {P}-positivity and involution {Stanley} symmetric functions.
\newblock {\em Int. Math. Res. Notices}, rnx274.

\bibitem{HMP1}
Z.~Hamaker, E.~Marberg, and B.~Pawlowski.
\newblock Involution words: counting problems and connections to {Schubert}
  calculus for symmetric orbit closures.
\newblock {\em J. Combin. Theory Ser. A}, 160:217--260, 2018.

\bibitem{HMP3}
Z.~Hamaker, E.~Marberg, and B.~Pawlowski.
\newblock Transition formulas for involution {S}chubert polynomials.
\newblock {\em Selecta Math. (N.S.)}, 24(4):2991--3025, 2018.

\bibitem{HMP5}
Z.~Hamaker, E.~Marberg, and B.~Pawlowski.
\newblock Fixed-point-free involutions and {Schur} {P}-positivity.
\newblock {\em J. Combin.}, to appear.

\bibitem{double-flag-varieties}
X.~He, K.~Nishiyama, H.~Ochiai, and Y.~Oshima.
\newblock On orbits in double flag varieties for symmetric pairs.
\newblock {\em Transform. Groups}, 18:1091--1136, 2013.

\bibitem{humphreysreflgroups}
James Humphreys.
\newblock {\em Reflection Groups and Coxeter Groups}.
\newblock Cambridge University Press, 1990.

\bibitem{ivanov-interpolation-PQ}
V.~N. Ivanov.
\newblock Interpolation analogs of {Schur} {$Q$-functions}.
\newblock {\em J. Math. Sci.}, 237:5495--5507, 2005.

\bibitem{kazarian}
M.~Kazarian.
\newblock On {Lagrange} and symmetric degeneracy loci.
\newblock Preprint, 2000.

\bibitem{kempf-laksov}
G.~Kempf and D.~Laksov.
\newblock The determinantal formula of {Schubert} calculus.
\newblock {\em Acta Math.}, 132:153--162, 1974.

\bibitem{positroidjuggling}
A.~Knutson, T.~Lam, and D.~Speyer.
\newblock Positroid varieties: Juggling and geometry.
\newblock {\em Compos. Math.}, 149:1710--1752, 2013.

\bibitem{back-stable-schubert}
T.~Lam, S.~Lee, and M.~Shimozono.
\newblock Back stable {Schubert} calculus.
\newblock arXiv:1806.11233, 2018.

\bibitem{lascouxschutzenbergerschubert}
A.~Lascoux and M-P. Sch{\"u}tzenberger.
\newblock Polyn{\^o}mes de {Schubert}.
\newblock {\em Comptes Rendus des S{\'e}ances de l'Acad{\'e}mie des Sciences.
  S{\'e}rie I. Math{\'e}matique}, 294:447--450, 1982.

\bibitem{lascouxschutzenbergertree}
A.~Lascoux and M-P. Sch{\"u}tzenberger.
\newblock Schubert polynomials and the {Littlewood-Richardson} rule.
\newblock {\em Lett. Math. Phys.}, 10:111--124, 1985.

\bibitem{Macdonald}
I.~Macdonald.
\newblock {\em Symmetric Functions and Hall Polynomials}.
\newblock Oxford University Press, 1995.

\bibitem{pragacz-LG-OG}
P.~Pragacz.
\newblock Algebro-geometric applications of {S}chur {$S$}- and
  {$Q$}-polynomials.
\newblock In {\em Topics in invariant theory ({P}aris, 1989/1990)}, volume 1478
  of {\em Lecture Notes in Math.}, pages 130--191. 1991.

\bibitem{stanleysymm}
R.~P. Stanley.
\newblock On the number of reduced decompositions of elements of {Coxeter}
  groups.
\newblock {\em European J. Combin.}, 5:359--372, 1984.

\bibitem{szechtman}
F.~Szechtman.
\newblock Equivalence and congruence of matrices under the action of standard
  parabolic subgroups.
\newblock {\em Electron. J. Lin. Alg.}, 16:325--333, 2007.

\bibitem{geometric-littlewood-richardson}
R.~Vakil.
\newblock A geometric {Littlewood-Richardson} rule.
\newblock {\em Annals of Mathematics}, 164(2):371--421, 2006.

\bibitem{wachs-schubert}
M.~Wachs.
\newblock Flagged {Schur} functions, {Schubert} polynomials, and symmetrizing
  operators.
\newblock {\em J. Combin. Theory Ser. A}, 40:276--289, 1985.

\bibitem{wyser-degeneracy-loci}
B.~J. Wyser.
\newblock K-orbit closures on ${G}/{B}$ as universal degeneracy loci for
  flagged vector bundles with symmetric or skew-symmetric bilinear form.
\newblock {\em Transform. Groups}, 18:557--594, 2013.

\bibitem{wyser-yong-orthogonal-symplectic}
B.~J. Wyser and A.~Yong.
\newblock Polynomials for symmetric orbit closures in the flag variety.
\newblock {\em Transform. Groups}, 22:267--290, 2017.

\end{thebibliography}

\end{document}